\documentclass[a4paper]{amsart}
\usepackage[utf8]{inputenc}

\title[Smooth Cartan triples]{Smooth Cartan triples and Lie twists over Hausdorff \'etale Lie groupoids}

\author{Anna Duwenig}
\address{Department of Mathematics, KU Leuven, Leuven, Belgium}
\email{anna.duwenig@kuleuven.be}

\author{Aidan Sims}
\address{School of Mathematics and Applied Statistics, University of Wollongong, Wollongong, Australia}
\email{asims@uow.edu.au}


\date{\today}

\thanks{This research was supported by Australian Research Council Discovery Project DP220101631. Anna Duwenig was supported by Methusalem grant METH/21/03 –- long term structural funding of the Flemish Government, and an FWO Senior Postdoctoral Fellowship (project number 1206124N). We thank Valentina Wheeler for a number of helpful discussions, and we thank Sven Raum and Jonathan Taylor for careful reading and helpful discussions that helped us fix some mistakes.
}

\usepackage[backend=biber, 
	giveninits=true,
	hyperref=true,
	maxcitenames=3,
	maxbibnames=100,
	block=none,
     url=false, 
	isbn=false]{biblatex}	
	\AtEveryBibitem{%
        \clearfield{note}%
		\clearfield{mrclass}%
		\clearfield{mrnumber}%
		\clearfield{mrreviewer}%
		\clearfield{doi}%
	}
\usepackage{tikz-cd}
\usepackage{tensor}
    \newcommand{\fp}[2]{{\mskip -2.0mu}\tensor*[_{#1}]{\ast}{_{#2}}{\mskip -2.0mu}} 
    \newcommand{\fpsr}{\fp{}{}}

\newcommand{\PB}[3]{
#2 \colon #1 \to #3
}
\usepackage{mathtools} 
\usepackage{amsmath}
\usepackage{amssymb}
\usepackage{amsfonts}
\usepackage{amsthm}
\usepackage{xcolor}
\usepackage[normalem]{ulem}
\usepackage{mathrsfs}  
\usepackage{enumitem}
        \makeatletter
			\def\namedlabel#1#2{\begingroup#2%
			\def\@currentlabel{#2}%
			\phantomsection\label{#1}\endgroup
		}
		\makeatother
\usepackage{graphicx}
\usepackage[
    margin=1in
]{geometry}
\usepackage{hyperref}
\usepackage{stmaryrd} 

\newcommand{\cA}{\mathcal{A}} 
\newcommand{\cB}{\mathcal{B}} 
\newcommand{\cR}{\mathcal{R}} 
\newcommand{\cS}{\mathcal{S}} 

\newcommand{\z}{^{\scalebox{.7}{$\scriptstyle (0)$}}} 
\newcommand{\inv}{^{\scalebox{.7}{$\scriptstyle -1$}}}
\newcommand{\comp}{^{\scalebox{.7}{$\scriptstyle (2)$}}}
\newcommand{\vecrm}[1]{\vec{\mathrm{#1}}}

\newcommand{\dd}{\mathrm{d}} 

\DeclareMathOperator{\dom}{dom}
\DeclareMathOperator{\supp}{supp}
\DeclareMathOperator{\Ph}{Ph} 
\newcommand{\etale}{{\'e}tale}
\newcommand{\Etale}{{\'E}tale}

\let\mbb\mathbb

\theoremstyle{plain} 
    \newtheorem{thm}{Theorem}[section]
    \newtheorem{prop}[thm]{Proposition}
    \newtheorem{lemma}[thm]{Lemma}
    \newtheorem{corollary}[thm]{Corollary}

\theoremstyle{remark}
    \newtheorem{remark}[thm]{Remark}

\theoremstyle{definition} 
    \newtheorem{example}[thm]{Example}
    \newtheorem{defn}[thm]{Definition}
    \newtheorem{ntn}[thm]{Notation}

\numberwithin{equation}{section}

\makeatletter
\renewcommand{\paragraph}{\@startsection{paragraph}{4}    {0ex}
   {0ex}
   {-3.25ex plus -1ex minus -0.2ex}
   {\normalfont\normalsize\itshape}
   }
\makeatother

\newcommand{\suppo}{\supp^{\mathrm{o}}}
\newcommand{\twist}{E}
\newcommand{\condExp}{P}
\newcommand{\ses}{ \mbb{T} \times G\z \overset{\iota}{\to} \twist \overset{\pi}{\to} G } 
\newcommand{\sestriple}{\ses}

\begin{document}
\newcommand{\PropE}{\ref{property:U infty}} 
\newcommand{\PropB}{\ref{property:S infty}} 
\newcommand{\PropC}{\ref{property:M infty}} 
\newcommand{\PropCC}{\ref{property:M infty B}} 
\newcommand{\PropD}{\ref{property:I infty}} 

\newcommand{\PropEB}{\ref{property:S infty T}}
\newcommand{\PropEC}{\ref{property:M infty T}}
\newcommand{\PropED}{\ref{property:I infty T}}

\newcommand{\PropBz}{\ref{property:S infty,etale}}
\newcommand{\PropCz}{\ref{property:M infty,etale}}
\newcommand{\PropDz}{\ref{property:I infty,etale}}

\newcommand{\PropEstar}{\ref{property:U,C*}}
\newcommand{\PropBstar}{\ref{property:S,C*}}
\newcommand{\PropCstar}{\ref{property:M,C*}}
\newcommand{\PropDstar}{\ref{property:I,C*}}

\begin{abstract}
We describe how to recover a Lie structure on a twist over a Hausdorff
\etale\ groupoid from functional-analytic data in the spirit of Connes'
reconstruction theorem for manifolds. We first characterise when a smooth
structure on the unit space of a Hausdorff \etale\ groupoid can be extended
to a Lie-groupoid structure on the whole groupoid. We introduce Lie twists
over Hausdorff Lie groupoids, building on Kumjian's notion of a twist over
a topological groupoid. We establish necessary and sufficient conditions on
a family of sections of a twist over a Lie groupoid under which the twist
can be made into a Lie twist so that all the specified sections are smooth.
We use these results in the setting of twists over \etale\ groupoids to
describe conditions on a Cartan pair of $C^{*}$-algebras and a family of
normalisers of the subalgebra, under which Renault's Weyl twist for the
pair can be made into a Lie twist for which the given normalisers
correspond to smooth sections.
\end{abstract}

\maketitle

\section{Introduction}

The motivating idea of this paper is that Connes' reconstruction theorem for
manifolds \cite{Connes:manifolds} and  Kumjian--Renault theory for twists
over \'etale groupoids \cite{Kum:Diags, Renault:Cartan} can be employed in
tandem to understand the extent to which Lie-groupoid structures on the
components of topological twists over \etale\ groupoids can be recovered from
spectral and functional-analytic data.

Connes' theorem says that a spectral triple $(\mathscr{A}, \mathcal{H}, D)$
that satisfies slightly strengthened versions of the first five axioms laid
out in \cite{Connes:gravity} and in which the algebra $\mathscr{A}$ is unital
and commutative, determines a manifold $M$ that realises $\mathscr{A}$ as the
algebra $C^{\infty}(M)$ of smooth functions. This can be regarded as a deep
geometric strengthening of Gelfand duality for commutative $C^{*}$-algebras:
the Gelfand--Naimark Theorem says that every unital commutative
$C^{*}$-algebra $A$ is the algebra of continuous functions on a compact
Hausdorff space $X$; and, very roughly speaking, Connes' theorem describes
when this $X$ can be endowed with the structure of a manifold in terms of
data that specifies a subalgebra $\mathscr{A}$ of $A$ that is to consist of
smooth functions and an unbounded operator associated to $A$ that is to be a
Dirac-type differential operator.

Kumjian's and Renault's theorems can also be regarded as a strengthening of
the Gelfand--Naimark Theorem. They describe under what circumstances a given
commutative subalgebra $B$ of a $C^{*}$-algebra $A$ ``coordinatises $A$'' (to
borrow a phrase of Muhly's) in the sense that there is a topological
\emph{twist} $\ses$ (often called the \emph{Weyl twist}) over an effective
Hausdorff \'etale groupoid $G$ (called the \emph{Weyl groupoid}) with unit
space the Gelfand spectrum $\widehat{B}$ of $B$ such that the Gelfand
isomorphism $B \cong C_{0}(\widehat{B})$ extends to an isomorphism of $A$
onto the reduced twisted groupoid $C^{*}$-algebra $C^{*}_{r}(G; \twist)$. The
characterisation is in terms of {\em normalisers} of $B$ in $A$: elements $n
\in A$ satisfying $n^{*}Bn \cup nBn^{*} \subseteq B$. The point is that the
partial-isometric factors of polar decompositions of normalisers determine
partial homeomorphisms of $\widehat{B}$, which can be glued together into the
desired groupoid $G$; the twist $\twist$ then measures the phase-differences
between the partial-isometric factors of pairs of normalisers that determine
the same partial homeomorphisms. Proposition~4.1 of \cite{DGNRW:Cartan} shows
that any family of normalisers that densely spans $A$ is ``big enough" to
recover $G$ and~$\twist$.

Taken together, the results described in the preceding two paragraphs suggest
a natural question. Suppose that we are given a twist $\twist$ over a
Hausdorff effective \'etale groupoid $G$, and hence a corresponding Cartan
pair $B\subseteq A$ of $C^{*}$-algebras. Suppose in addition that $G$ and
$\twist$ carry smooth structures under which they are Lie groupoids and such
that the bundle map $\pi \colon \twist \to G$ is a submersion and the
inclusion $\iota \colon \mbb{T} \times G\z \to \twist$ is smooth---that is,
the pair constitutes an $S^1$-central extension in the sense of
\cite[Definition~4.1]{BX:2003}, with the additional property that the
reduction of the extension to the unit space is a trivial smooth
$\mbb{T}$-bundle. There is then, in addition to the smooth subalgebra
$B^{\infty} = C^{\infty}(G\z)\cap C_{0}(G\z)$ of $B$, a natural collection of
``smooth normalisers" that densely span $C^{*}_{r}(G; \twist)$, namely the
smooth $\mbb{T}$-equivariant functions on $\twist$ whose open supports are
preimages of open bisections. The natural question, which we answer here, is,
``to what extent is this process reversible?" That is, given a Cartan pair
$B\subseteq A$ and a family $\mathscr{N}$ of normalisers that densely spans
$A$ and for which $\mathscr{N} \cap B$ is a dense $^{*}$-subalgebra of $B$,
what properties characterise the existence of a Lie-twist structure on the
Weyl twist for which $\mathscr{N}$ is a family of smooth sections and
$\mathscr{N} \cap B$ consists of smooth functions on $G\z$? Connes'  theorem
does the heavy lifting of reconstructing a geometric structure on
$\widehat{B}$ from spectral data, so it is reasonable to take as given that
$\widehat{B}$ comes already endowed with a smooth-manifold structure. We thus
have in hand the algebra $B^{\infty} = C^{\infty}(G\z)\cap C_{0}(G\z)$ of
smooth functions on $\widehat{B}$, and ask what is required of $\mathscr{N}$
to recover the remaining geometry on $G$ and $\twist$.

There are two parts to our answer to this question. The first part deals with
whether the Weyl groupoid $G$ can be made into a Lie groupoid. This is
reasonably straightforward. We prove in Section~\ref{sec:etale smooth units}
that if $G$ is an \'etale groupoid and $G\z$ is a manifold, then there is a
unique smooth structure on $G$ with respect to which the range map is smooth.
The source map is also smooth with respect to this structure if and only if
the partial homeomorphisms $r(\gamma) \mapsto s(\gamma)$, $\gamma \in U$,
determined by open bisections $U$ of $G$ are in fact diffeomorphisms, and in
this case $G$ is a Lie groupoid. Translating via Renault's theory into
statements about normalisers, this is equivalent to the condition that for
each normaliser of the Cartan subalgebra $B$, the partial homeomorphism on
$\widehat{B}$ determined by the partial-isometric factor in its polar
decomposition is a partial \emph{diffeomorphism}, and we prove that when
$n^*n$ and $nn^*$ are smooth, this can be characterised in terms of
conjugation by $n$ and $n^*$ carrying appropriate collections of smooth
functions to smooth functions.

The second part is more subtle and deals with the existence of a smooth
structure on the twist~$\twist$. We begin by analysing in
Section~\ref{sec:twists} what families $\{\sigma_{\alpha}\}_{\alpha}$ of
local sections of a twist $\ses$ over a Lie groupoid $G$ can be the smooth
sections for a Lie-twist structure on~$\twist$. Since our definition of a Lie
twist requires that $\iota$ be smooth, we require that whenever the domain
$U_{\alpha}$ of $\sigma_{\alpha}$ intersects $G\z$, the restriction of
$\sigma_{\alpha}$ to $U_{\alpha} \cap G\z$ is the image under $\iota$ of a
smooth section of $\mbb{T} \times G\z$; we call this condition \PropE, the
letter U standing for ``units." This provides a means of evaluating whether
the difference between various combinations of the $\sigma_{\alpha}$ are
smooth, and the remaining conditions exploit this. We require that on
intersections $U_{\alpha} \cap U_{\alpha}'$ the difference
$\sigma_{\alpha}(\gamma)\sigma_{\alpha'}(\gamma)\inv$ is the image under
$\iota$ of a smooth function (Condition~\PropB, for ``smooth structure"); and
that when $U_{\alpha_{1}} U_{\alpha_{2}}$ intersects $U_{\alpha}$, the
difference
$\sigma_{\alpha_{1}}(\gamma_{1})\sigma_{\alpha_{2}}(\gamma_{2})\sigma_{\alpha}(\gamma_{1}\gamma_{2})\inv$
is the image of a smooth function on its domain (Condition~\PropC, for
``multiplication"). Our first main theorem says, roughly speaking, that
Conditions \PropE~and~\PropB\ guarantee that there is a unique smooth
structure on $\twist$ under which it is a smooth principal $\mbb{T}$-bundle
and the $\sigma_{\alpha}$ are all smooth; and that, for this smooth
structure, the maps $\iota, \pi$ in the twist and the subspaces $\twist\z$
and $\pi\inv (G\z)$ all satisfy suitable smoothness conditions. We then prove
that $\twist$ is a Lie twist over $G$ with this smooth structure precisely
if, in addition, Condition~\PropC\ is satisfied. In Subsection~\ref{ssec:Lie
over etale}, we then specialise to the situation where $G$ is an \'etale Lie
groupoid. We use the \'etale property to re-cast Conditions~\PropE, \PropB,
and~\PropC\ in terms of smoothness of maps from open subsets of $G\z$ into
$\mbb{T}$.

To translate this into properties of families of normalisers, we start
Section~\ref{sec:Cstar} by making explicit the relationship between
$\mbb{T}$-equivariant functions on the twist $\twist$, which are dense in the
associated $C^*$-algebra, and sections of the twist itself. Specifically,
given a $\mbb{T}$-equivariant function $n$ on $\twist$, the map that sends a
point $\gamma$ in the image under $\pi$ of the open support of $n$ to the
unique point $e \in \pi\inv (\gamma)$ such that $n(e) > 0$ is a section
$\sigma_n$ of $\twist$, characterised by the formula $e = \Ph(n(e)) \cdot
\sigma_n(\pi(e))$. The sections $\sigma_{n}$ and $\sigma_{n'}$ coincide
precisely if $n$ and $n'$ differ by a positive continuous function on $G\z$.
Combining this with the faithful conditional expectation $\condExp \colon
C^{*}_{r}(G; \twist) \to C_{0}(G\z)$ extending restriction of functions, we
identify conditions on a family $\{n_{\alpha}\}_{\alpha}$ of normalisers,
phrased in terms of membership of images under $\condExp$ of combinations of
the $n_{\alpha}$ in algebras of smooth bounded functions on open subsets of
$G\z$, that are equivalent to Conditions~\PropE, \PropB, and~\PropC\ for the
associated sections. We show further that smoothness of these images can be
characterised by their behaving as multipliers of appropriate subalgebras of
$C^{\infty}_{c}(G\z)$.

With all this in hand, we are able to define what we call a \emph{smooth
Cartan triple}: a triple $(A, B,  \mathscr{N})$ such that $B\subseteq A$ is a
Cartan pair, $\widehat{B}$ is a smooth manifold, $C^{\infty}_{c}(\widehat{B})
\subseteq \mathscr{N} \cap B \subseteq C^{\infty}(G\z)\cap C_{0}(G\z)$ and
$\mathscr{N}$ is a family of normalisers of $B$ that also normalise
$B^{\infty}$ and have the property that $\Ph(\condExp(n))$,
$\Ph(\condExp(mk^{*}))$, and $\Ph(\condExp(nmk^{*}))$ are multipliers of the
appropriate subalgebras of $C^{\infty}_{c}(\widehat{B})$ for all $k,m,n \in
\mathscr{N}$. Our culminating theorem says that if $\ses$ is a Lie twist over
an effective Lie groupoid, then there is a smooth Cartan triple
$(C^{*}_{r}(G; \twist), C_{0}(G\z), \mathscr{N})$ such that $\mathscr{N}$
consists of smooth functions (for example, taking $\mathscr{N}$ to be all
smooth $\mathbb{T}$-equivariant functions compactly supported on preimages of
bisections suffices); and conversely that if $(A, B, \mathscr{N})$ is a
smooth Cartan triple, then there is a unique smooth structure on the Weyl
twist of  $B\subseteq A$, compatible with the manifold structure on
$\widehat{B}$, under which it becomes a Lie twist and the elements of
$\mathscr{N}$ are smooth. We finish with some brief remarks: one on how our
main theorem can be combined with Connes' reconstruction theorem to obtain a
correspondence between Lie twists and purely functional-analytic data; and
one on the effect of the choice of the family $\mathscr{N}$ on the resulting
smooth structure on~$\twist$.

We include an appendix on differential geometry where we establish our
notation and conventions, and collect a number of standard results that we
need along the way. It is organised into two sections: one about smooth
manifolds and their submanifolds in general; and one specifically relating to
smooth principal bundles.

\section{Preliminaries: Lie groupoids}\label{sec:prelims}

In this brief section, we introduce some preliminary material on the notions
of topological groupoids and Lie groupoids and establish a few of their basic
properties. In particular, when the groupoid in question is an \etale\ Lie
groupoid, we give an explicit description of an atlas for the natural smooth
structure on its space of composable pairs.

We will work with the following definition of a topological groupoid; we will
frequently drop the adjective and just say {\em groupoid}. In this paper,
groupoids are \textbf{always} locally compact Hausdorff topological
groupoids.

\begin{defn}[{cf.\ \cite[Remark~8.1.5]{Sims:gpds}}]
\label{dfn:Top gpd}
    A {\em (Hausdorff) topological groupoid} consists of
\begin{enumerate}[label= (G\arabic*), itemsep=0ex]
    \item\label{item:gpd:spaces} two locally compact Hausdorff spaces $G$
        and $G\z$,
    \item\label{item:gpd:inclusion map} an injection $G\z \hookrightarrow
        G$ that is a homeomorphism onto its range (under the subspace
        topology), so that we may regard $G\z$ as a subspace, called the
        \emph{unit space}, of $G$,
    \item\label{item:gpd:s and r} two continuous surjections \(s,r \colon
        G \to G\z  \) called the {\em source map} and the {\em range map},
    \item\label{item:gpd:multiplication map}  a continuous {\em
        composition} map \( G\comp = G \fpsr G \to G,
    (\gamma,\eta)\mapsto \gamma\eta,\text{ and} \)
    \item\label{item:gpd:inversion map}   a continuous {\em inversion} map
        \( G \to G,
    \gamma\mapsto \gamma\inv \),
\end{enumerate}
such that
\begin{enumerate}[label= (G\arabic*), resume, itemsep=0ex]
\item\label{item:gpd:s and r of comp} \(s(\gamma\eta) = s(\eta) \) and
    \(r(\gamma\eta) = r(\gamma) \) for every \((\gamma,\eta) \in G\comp\),
\item\label{item:gpd:associative} composition is associative,
\item\label{item:gpd:u unit} \(s(x) = r(x) = x \) for every \(x \in G\z \)
    and \(\gamma s(\gamma) = \gamma \) and \(r(\gamma) \gamma = \gamma \)
    for every \(\gamma \in G\),\text{ and}
\item\label{item:gpd:inverse} \(\gamma\inv \gamma = s(\gamma) \) and
    \(\gamma\gamma\inv = r(\gamma) \) for every \(\gamma \in G\).
\end{enumerate}
A groupoid is called {\em \etale} if its range and source maps are local
homeomorphisms.
\end{defn}

Our notation for groupoids follows \cite{Renault:gpd-approach, Sims:gpds}. We
will typically denote groupoids by $G$, except for groupoids that are twists
over another base groupoid, which we denote by $\twist$.

Given sets $X, Y, Z$ and functions $f \colon X \to Z$ and $g \colon Y \to Z$,
we write $X \fp{g}{f} Y$ for the fibred product
\[
    X \fp{f}{g} Y \coloneqq \{(x,y) \in X \times Y : f(x) = g(y)\}.
\]
If $X,Y$ are topological spaces, we give $X \fp{f}{g} Y$ the relative
topology inherited from the product space. In particular, if $G$ is a
groupoid, then for subsets $U,V \subseteq G$, the collection of composable
pairs in $U \times V$ is the fibred product over the range and source maps.
That is,
\[
U \fp{s}{r} V = (U \times V) \cap G\comp.
\]
In this specific instance, we will drop the $s$ and $r$ in the notation, and
just write $U \fpsr V \coloneqq U \fp{s}{r} V$; we will always retain the
subscripts for fibred products that are not of this specific form.

We write $UV$ for the collection of products of pairs in $U \fpsr V$:
\[
UV \coloneqq \{\gamma\eta : (\gamma,\eta) \in U \fpsr V\}.
\]
If $G$ is an \etale\ groupoid, then it admits a base of open sets $U
\subseteq G$ such that $r|_U$ and $s|_U$ are homeomorphisms onto open subsets
of its unit space $G\z$. In keeping with terminology from groupoid
$C^*$-algebras \cite{Renault:gpd-approach, Renault:Cartan} (as opposed to the
terminology used for Lie groupoids in, for example, \cite[Definitions 1.4.1,
1.4.8]{Mackenzie:2005:book}), we call such sets $U$ {\em open bisections} (or
sometimes just {\em bisections}). By \cite[Proposition~3.8]{Exel:BBMS}, if
$U$ and $V$ are open bisections, then $UV$ is also a (possibly empty) open
bisection.

We will be particularly interested in topological groupoids that come with
additional smooth structure. In the following, we will refer to a topological
manifold with a smooth structure simply as a ``manifold" rather than as a
``smooth manifold." Our terminology and notation for manifolds follow the
conventions of \cite{Lee:Intro}, and are detailed in
Appendex~\ref{sec:appendix} for ease of reference.

\begin{defn}[{\cite[Definitions 7.1.1 and 7.1.6]{Dufour2005}}]\label{dfn:Lie gpd}
    A {\em Lie groupoid} is a topological groupoid $G$ such that
\begin{enumerate}[label= (L\arabic*), itemsep=0ex]
    \item\label{item:Liegpd:mfds} $G$ and $G\z$ are smooth manifolds (under
        their given topological structures),
    \item\label{item:Liegpd:inclusion map}  the inclusion \(  G\z
        \hookrightarrow G \) is an embedding of manifolds (that is, in
        addition to~\ref{item:gpd:inclusion map} above, it is a smooth
        immersion),
    \item\label{item:Liegpd:s and r} the maps $s,r \colon G \to G\z$ are
        submersions,
    \item\label{item:Liegpd:multiplication map} the  composition map
        \(G\comp \to G, (\gamma,\eta)\mapsto \gamma\eta
    \), is smooth, and
    \item\label{item:Liegpd:inversion map}  the inversion map $G\to
        G,\gamma \mapsto \gamma\inv$,
    is smooth.
\end{enumerate}
A Lie groupoid is called {\em \etale} if its range and source maps are local
diffeomorphisms.
\end{defn}
The above definition makes sense: Since the fibred product $M\fp{f}{g}N$ of
two manifolds along transverse maps $f\colon M\to K$ and $g\colon N\to K$ is
itself a manifold (Proposition~\ref{prop:transverse maps implies fibred
product is mfd}) and since a submersion of manifolds is transverse to any map
(Remark~\ref{rmk:submersion implies transverse}), the subspace $G\comp$ of
$G\times G$ is indeed a manifold in its own right.\label{page:G2 submanifold}

\begin{remark}\label{rmk:inversion auto-smooth}
    As shown in \cite[Proposition 1.1.5]{Mackenzie:2005:book}, Assumption~\ref{item:Liegpd:inversion map} is redundant: the other conditions already imply that the inversion map is a diffeomorphism. This observation will simplify many of our proofs later on.
\end{remark}

\begin{remark}
    For \etale\ Lie groupoids, the manifold dimension of the unit space $G\z$ and of the morphism space $G$ are identical. Consequently, this dimension also coincides with the manifold dimension of the space  $G\comp$ of composable pairs: given smooth functions $f\colon M\to K$ and $g\colon N\to K$, the dimension of the fibred product $M\fp{f}{g}N$ is $\dim(M) + \dim(N) - \dim(K)$; see Proposition~\ref{prop:transverse maps implies fibred product is mfd} for a proof. In particular, for \etale\ $G$, we have $\dim(G\comp) = 2\dim(G)-\dim(G\z)=\dim(G)$.
\end{remark}

For practical reasons, we will now consider objects that are more general
than Lie groupoids. We will always assume that smooth structures on these
objects are compatible with the given underlying topologies where applicable.

\begin{defn}[cf.\ {\cite[p.\ 207]{Dufour2005}}]\label{dfn:smooth bisection}
    Suppose that $G$ is a topological groupoid and that both $G$ and $G\z$ are manifolds. An open bisection $B$ of $G$ is called a {\em smooth} bisection if $r|_{B}\colon B\to r(B)$ and  $s|_{B}\colon B\to s(B)$ are diffeomorphisms.
\end{defn}

 Suppose that $G$ is
 an \'etale
 groupoid and that both $G$ and $G\z$ are manifolds. Assume that both $r$ and $s$ are local diffeomorphisms (in particular, they are submersions by Lemma~\ref{lem:local diffeo is submersion}), so that $G\comp$ is an embedded submanifold of $G\times G$ by
 Proposition~\ref{prop:transverse maps implies fibred product is mfd}.
 We will now describe the charts of $G\comp$.

 Let $\mathcal{U}$ denote the set of smooth bisections of $G$. For a fixed
    $(\gamma_{1},\gamma_{2})\in G\comp$,
    fix $B_i'\in \mathcal{U}$ containing $\gamma_i$,
 and let $W\coloneqq s(B_{1}')\cap r(B_{2}')$. Notice that
    \begin{align*}
        B_{1}\coloneqq B_{1}'W
        =
        B_{1}'\cap s\inv (W)
    \quad\text{and}\quad
        B_{2}\coloneqq
        WB_{2}'
        =
        r\inv(W)\cap B_{2}'
    \end{align*}
    are open, since $r, s$ are continuous open maps (\cite[Corollary 11.6]{Tu:Intro}). As subsets of smooth bisections, they are therefore likewise smooth bisections, and we clearly have
    $s(B_{1}) = W= r(B_{2})$.

Now fix a chart $\varphi_{1}$ around $\gamma_{1}$ in $G$; by potentially
shrinking sets, we can assume without loss of generality that the domain of
$\varphi_{1}$ is exactly $B_{1}$. Writing $n$ for the manifold dimension of
$G$, we define
\[
    \varphi_{2}\coloneqq\varphi_{1}\circ (s|_{B_{1}})\inv \circ r|_{B_{2}}\colon B_{2}\approx \mbb{R}^{n},
\]
Since $s$ and $r$ are local diffeomorphisms, the map $\varphi\coloneqq
\varphi_{1}\times\varphi_{2}$ is a smooth chart around
    $(\gamma_{1},\gamma_{2})$,
    and it is easy to check that the set $\varphi\big(G\comp \cap (B_{1}\times B_{2})\big)=\varphi(B_{1} \fpsr B_{2})$ is exactly the diagonal $\{(\vecrm{v}, \vecrm{v}):\vecrm{v} \in \varphi(B_{1})\}$ in $\mbb{R}^{n}\times\mbb{R}^{n}$.
Thus, composing $\varphi$ with the diffeomorphism
$\mbb{R}^{2n}\to\mbb{R}^{2n},(\vecrm{x}, \vecrm{y})\mapsto
(\vecrm{x},\vecrm{x}-\vecrm{y})$, yields a chart that realises $B_{1} \fpsr
B_{2}$ as an $n$-dimensional embedded submanifold of $B_{1}\times B_{2}$.

Consequently, if $\mathrm{pr}_{1}$ denotes the projection onto the first
component of $G\times G$, then $(\varphi_{1}\circ \mathrm{pr}_{1}, B_{1}
\fpsr B_{2})$ is a chart for $G\comp$, and the collection of all such charts
forms a smooth atlas. Since the smooth structures of $G$ and $G\z$ are
compatible (meaning, we can turn one into the other by use of $r$ or $s$), we
can instead also describe an atlas on $G\comp$ as follows.

    \begin{lemma}\label{lem:structure on Gcomp in general}
       Suppose that $G$ is a topological groupoid and that both $G$ and $G\z$ are manifolds. Let $\cA=\{(W_{\alpha}, \psi_{\alpha})\}_{\alpha\in \mathfrak{A}}$ be a maximal smooth atlas of $G\z$. Assume that both $r$ and $s$ are local diffeomorphisms
       and let $\mathcal{U}$ denote the set of smooth bisections of $G$.  Let
       \[
            \mathfrak{C} \coloneqq
            \{
                (\alpha,B_{1},B_{2})\in \mathfrak{A}\times\mathcal{U}\times\mathcal{U}
                :
                s(B_{1}) = W_{\alpha}= r(B_{2})
            \}.
        \]
        For $\chi= (\alpha, B_{1}, B_{2})\in \mathfrak{C}$, let $V_{\chi}\coloneqq B_{1} \fpsr B_{2}$ and define $\Phi_{\chi}\colon V_{\chi}\to \mbb{R}^{n}$ by $\Phi_{\chi}(\gamma_{1},\gamma_{2})\coloneqq \psi_{\alpha}(s(\gamma_{1}))=\psi_{\alpha}(r(\gamma_{2}))$.
        Then the collection
        \(
            \mathcal{C}=\{(V_{\chi}, \Phi_{\chi})\}_{\chi\in \mathfrak{C}}
        \)
        is a smooth atlas for the standard manifold structure of $G\comp$.
    \end{lemma}

\section{\Etale\ groupoids with unit space a manifold}\label{sec:etale smooth units}

The goal of this section is to prove that if $G$ is an \etale\ groupoid in
which the unit space $G\z$ is a manifold, then there is a smooth structure on
$G$ making it into a Lie groupoid if and only if the open bisections of $G$
induce smooth maps between open subsets of $G\z$; and moreover that this
smooth structure is unique and admits an explicit atlas easily described in
terms of the range map and a given atlas for $G\z$.

\begin{prop}\label{prop:cR}
    Suppose that $G$ is an \etale\ groupoid and that $G\z$ is a manifold. For a smooth atlas $\{(W_{\alpha}, \psi_{\alpha}) : \alpha \in \mathfrak{A}\}$ for $G\z$, let
    \[
        \mathfrak{R} \coloneqq \left\{
            (B,\alpha
            ): B\text{ is an open bisection of $G$ such that }
                r(B)
                \subseteq
                W_{\alpha}
        \right\}
    \]
    and for each $\rho = (B,\alpha) \in \mathfrak{R}$ define $U_\rho = B$ and
    \[
        \varphi_{\rho}\colon U_\rho \to \mbb{R}^n,
        \quad
            \varphi_\rho (\gamma) \coloneqq  \psi_{\alpha} (r(\gamma)).
    \]
    Then $\cR \coloneqq \{(U_\rho, \varphi_\rho) : \rho \in \mathfrak{R}\}$ is an atlas for a smooth structure on $G$. Moreover, any other smooth structure $\cB$ on $G$ with respect to which $r\colon (G,\cB)\to G\z$ is a local diffeomorphism is compatible with $\cR$.
\end{prop}

\begin{proof}
    We claim that $\cR$ is a $C^{\infty}$-atlas for $G$
    in the sense of Definition~\ref{dfn:atlas}.
    In what follows, we will use the standard convention that for $\rho,\rho' \in \mathfrak{R}$, we write $U_{\rho,\rho'}$ for the intersection $U_\rho \cap U_{\rho'}$.

    To check~\ref{item:atlas:cover}, fix $\gamma\in G$.
    Since $G$ is \etale, it has a base of open bisections, so  we may pick a bisection $B'$ that contains
    $\gamma$.
    Since $\{W_{\alpha}\}_{\alpha}$ covers $G\z$, we may pick
    $\alpha$
    such that
    $r(\gamma)\in W_{\alpha}$.
    Now let $B\coloneqq B'\cap r\inv(W_{\alpha}).$ The set $B$ is open since
    $W_{\alpha}$
    and $B'$ are open and since $r$ is continuous; it is nonempty as it contains
    $\gamma$;
    it is a bisection because $B'$ is a bisection. Since
    $r(B)\subseteq W_{\alpha}$, we have $(B,\alpha)\in \mathfrak{R}$, and so $U_{(B,\alpha)}=B$ contains $\gamma$, which proves that $\{U_\rho\}_{\rho\in \mathfrak{R}}$ covers $G$.

    To check~\ref{item:atlas:homeo}, note that by construction, $\varphi_\rho$ is the composition of the homeomorphism $r|_{B}\colon B\to W_{\alpha}$ and the homeomorphism $\psi_{\alpha}\colon W_{\alpha} \to \psi_{\alpha}(W_{\alpha}) \subseteq\mbb{R}^n$; hence $\varphi_\rho$ is a homeomorphism onto its image.

    For~\ref{item:atlas:smooth}, fix $\rho=(B,\alpha)$ and $\rho'=(B',\alpha')\in\mathfrak{R}$ such that $U_{\rho,\rho'}\neq \emptyset$.  We have
    \begin{align*}
        \varphi_\rho|_{U_{\rho,\rho'}} \circ (\varphi_{\rho'}|_{U_{\rho,\rho'}})\inv
        &
        =
        (\psi_{\alpha} \circ r|_{B\cap B'}) \circ (\psi_{\alpha'} \circ r|_{B\cap B'})\inv
        \\
        &=
        \psi_{\alpha} \circ (r|_{B\cap B'}\circ r_{B\cap B'}\inv)  \circ \psi_{\alpha'}\inv
        =
        \psi_{\alpha} \circ \psi_{\alpha'}\inv,
    \end{align*}
    which is smooth by assumption on the atlas on $G\z$.

    To see that $r$ is a local diffeomorphism,
    we must check that, for any given open bisection $B$, the map $r|_{B}$ and its inverse are not only continuous but also smooth. By replacing $B$ by a smaller neighbourhood around any given point in $B$, we may assume without loss of generality that there exists $\alpha\in\mathfrak{A}$ such that
    $(B, \alpha)$ is an element of $\mathfrak{R}$. Then the map $\varphi_\rho \circ r|_{B}\circ \psi_{\alpha}\inv$, defined between open subsets of $\mbb{R}^n$, is given by
    \[
        \psi_{\alpha} \circ r|_{B} \circ \varphi_\rho\inv
        =
        \psi_{\alpha} \circ r|_{B} \circ (\psi_{\alpha} \circ r|_{B})\inv
        =
        \mathrm{id}.
    \]
    Similarly,
        \[
        \varphi_\rho \circ  r|_{B}\inv\circ \psi_{\alpha}\inv
        =
        (\psi_{\alpha} \circ r|_{B}) \circ  r|_{B}\inv\circ \psi_{\alpha}\inv
        =
        \mathrm{id}.
    \]
    Since it
    suffices to prove smoothness for {\em one} set of charts per point,
    we deduce that $r|_{B}$ and its inverse are smooth.

    Lastly, suppose that $\mathcal{B}=\{(V_\beta,\phi_\beta)\}_{\beta\in\mathfrak{B}}$ is another smooth structure on $G$ with respect to which $r$ is a local diffeomorphism; we
    claim that $\mathcal{B}$ and $\mathcal{R}$ are  compatible, meaning that their union is again a smooth atlas.
    By assumption on $\mathcal{B}$, whenever $\rho = (B,\alpha) \in \mathfrak{R}$ and $\beta \in \mathfrak{B}$ satisfy $V_\beta \subseteq r(B)$, the map
    \(
        \psi_{\alpha} \circ r|_{B} \circ \phi_\beta\inv
    \)
    is a smooth function between open subsets of $\mbb{R}^n$. That is,
    \[
        \varphi_{\rho}\circ \phi_\beta\inv
        =
        (\psi_{\alpha}\circ r|_{B})\circ \phi_\beta\inv
    \]
    is smooth. So the two atlases are compatible.
\end{proof}

\begin{remark}\label{rmk:cS}
    Analogously, we could have asked for the source map $s$ to become a local diffeomorphism. We do this by defining
    \begin{align*}
        \mathfrak{S}\coloneqq
        \left\{
            (B,\alpha
            ): B\text{ open bisection of $G$ such that }
                s(B)\subseteq W_{\alpha}
        \right\}
        .
    \end{align*}
    For each $\sigma=(B,\alpha
    )$, we let $V_{\sigma}\coloneqq  B$ and define
    \[
        \phi_{\sigma}\colon V_{\sigma} \to \mbb{R}^n,
        \quad
        \phi_\sigma (\gamma) \coloneqq  \psi_{\alpha} (s(\gamma)).
    \]
    Then by symmetry, the argument for $\cR$ yields that  $\cS\coloneqq \{(U_{\sigma},\phi_{\sigma})\}_{\sigma\in\mathfrak{S}}$ is likewise a $C^{\infty}$-atlas for $G$ (but {\em a priori} for a different smooth structure). In this case, $s$ is a local diffeomorphism.
\end{remark}

If $G$ is an \etale\ {\em Lie} groupoid, then for every  smooth open
bisection $B$ of $G$ in the sense of Definition~\ref{dfn:smooth bisection},
the composition $s|_{B} \circ  r|_{B}\inv$ is a diffeomorphism between open
subsets of $G\z$. The definition of a smooth bisection, of course, only makes
sense when $G$ carries a smooth structure, but the condition just described
does not; so it is a necessary condition for the existence of a smooth
structure on $G$ making it into a Lie groupoid. We make this formal with the
following definition.

\begin{defn}
    Suppose that $G$ is an \etale\ groupoid such that $G\z$ is a manifold. We say that an open bisection $B \subseteq G$ {\em acts smoothly on $G\z$} if the map
    \[
        s|_{B} \circ  r|_{B}\inv \colon r(B)\to s(B)
    \]
    is a diffeomorphism. We say that {\em $G$ acts smoothly on $G\z$} if it admits a cover by open bisections that act smoothly on $G\z$.
\end{defn}

\begin{remark}
    If $G$ is an \etale\ groupoid such that $G\z$ is a manifold and $G$ acts smoothly on $G\z$, then we can choose a base $\mathfrak{U}$ of open bisections $B$ for which $s|_{B} \circ  r|_{B}\inv$ is a diffeomorphism in such a way that each of $r(\mathfrak{U})$ and $s(\mathfrak{U})$ is an atlas for the unit space.
\end{remark}

\begin{lemma}\label{lem:lcdiff holds for every bisection}
    Suppose that $G$ is an \etale\ groupoid and that $G\z$ is a manifold. Then $G$ acts smoothly on $G\z$
    if and only if \emph{every} open bisection $B \subseteq G$ acts smoothly on $G\z$.
\end{lemma}
\begin{proof}
    Clearly if every bisection acts smoothly, then the collection of all open bisections is a cover of $G$ by open bisections that act smoothly. Conversely, assume that $\mathcal{B}$ is a cover of $G$ consisting of open bisections of $G$ that act smoothly on $G\z$. Let $U$ be {\em any} open bisection of $G$, and fix $\gamma \in U$. Since $\mathcal{B}$ covers $G$, there exists $B \in \mathcal{B}$ with $\gamma \in B$. Let $V \coloneqq U \cap B$. Then $V$ is a neighbourhood of $\gamma$. The homeomorphism $s|_U \circ  r|_{U}\inv$ agrees with the diffeomorphism $s|_{B} \circ  r|_{B}\inv$ on the open neighbourhood $r(V)$ of $r(\gamma)$, so it is differentiable at $r(\gamma)$; similarly its inverse is differentiable at $s(\gamma)$.
\end{proof}

\begin{lemma}\label{lem:lcdiff-implies-compatible}
    Suppose that $G$ is an \etale\ groupoid and that $G\z$ is a manifold. If $G$ acts smoothly on $G\z$, then the atlases $\cR$ of Proposition~\ref{prop:cR} and $\cS$ of Remark~\ref{rmk:cS} are
    compatible.
\end{lemma}
\begin{proof}
    Take $\rho=(B,\alpha)\in \mathfrak{R}$ and $\sigma=(B', \beta)\in \mathfrak{S}$ with $U\coloneqq B\cap B'\neq\emptyset$.
    Then
    \begin{align*}
        \varphi_\rho|_{U} \circ (\phi_{\sigma}|_{U})\inv
        &
        =
        (\psi_{\alpha} \circ r|_{U}) \circ (\psi_{\beta} \circ s|_{U})\inv
        \\
        &
        =
        \psi_{\alpha} \circ (r|_{U}\circ s|_{U}\inv)  \circ \psi_{\beta}\inv
        .
    \end{align*}
    This is a diffeomorphism since $(r|_{U}\circ s|_{U}\inv)= \left(s|_{U}\circ r|_{U}\inv\right)\inv$ is a diffeomorphism by assumption and since the maps $\psi_{\alpha},\psi_\beta$ are smooth charts of $G\z$. Analogously, the composition of $\phi_{\sigma}$ with $\varphi_\rho\inv$ is a diffeomorphism.
\end{proof}

Our next goal is to verify that $G$ is a Lie groupoid with respect to the
smooth structures we have considered.

\begin{prop}\label{prop:lcdiff iff etale Lie gpd}
Suppose that $G$ is an \etale\ groupoid and that $G\z$ is a manifold. Then
the following are equivalent:
\begin{enumerate}[label=\textup{(\arabic*)}]
\item\label{item:TFAE:lcdiff} $G$ acts smoothly on $G\z$;
\item\label{item:TFAE:Lie gpd} there is a manifold structure on $G$ with
    respect to which it is an \etale\ Lie groupoid; and
\item\label{item:TFAE:explicit Lie gpd} the manifold structure on $G$
    obtained from Proposition~\ref{prop:cR} is the unique manifold
    structure on $G$ with respect to which it is an \etale\ Lie groupoid.
\end{enumerate}
\end{prop}

\begin{proof}
Clearly \ref{item:TFAE:explicit Lie gpd}~implies~\ref{item:TFAE:Lie gpd}.
That \ref{item:TFAE:Lie gpd}~implies~\ref{item:TFAE:lcdiff} is immediate: If
$G$ is an \etale\ Lie groupoid, then $r$ and $s$ are local diffeomorphisms,
meaning that every point of $G$ has a neighbourhood $U$ such that $U$ is
diffeomorphic to $r(U)$ via $r$ and to $s(U)$ via $s$. So the collection of
such bisections covers $G$, and hence $G$ acts smoothly on $G\z$.

For \ref{item:TFAE:lcdiff}~implies~\ref{item:TFAE:explicit Lie gpd}, suppose
that $G$ acts smoothly on $G\z$. We first argue that the manifold structure
on $G$ obtained from  Proposition~\ref{prop:cR} makes it into an \etale\ Lie
groupoid. We will argue uniqueness at the end. We have already verified
Condition~\ref{item:Liegpd:mfds} of Definition~\ref{dfn:Lie gpd}.
    It remains to show the conditions revolving around differentiability, so we start with a
    smooth atlas $\cA=\{(W_{\alpha}, \psi_{\alpha})\}_{\alpha\in\mathfrak{A}}$  of $G\z$, and we equip $G$ with the smooth structure described in Proposition~\ref{prop:cR}.

For~\ref{item:Liegpd:inclusion map}, fix $x\in G\z$ and let $(W_{\alpha},
\psi_{\alpha})$ be a chart around $x$. If $B'$ is any open bisection around
$x$, then $B\coloneqq  B'\cap G\z\cap W_{\alpha}$ is also an open bisection
around $x$, because $G$ is \etale\ and so $G\z$ is open in $G$. Since
$B\subseteq G\z$, the map $r|_{B}$ is the identity on $B$; in particular,
$r(B)=B\subseteq W_{\alpha}$, and so $(B,\alpha)\in\mathfrak{R}$. It follows
that $\varphi_{(B,\alpha)}=\psi_{\alpha},$ and
\[
    \varphi_{(B,\alpha)}\circ i \circ \psi_{\alpha}\inv
\]
is the identity map on the open set $\psi_{\alpha}(B) \subseteq\mbb{R}^n$.
Thus $i$ is smooth and its differential is everywhere injective, so $i$ is an
immersion.

For~\ref{item:Liegpd:s and r}, note that the range and source maps are
surjective by assumption on the topological groupoid $G$. We have constructed
the smooth structure on $G$ exactly so that $r$ and $s$ are local
diffeomorphisms, and so they are submersions by Lemma~\ref{lem:local diffeo
is submersion}. (This also explains why, once we have proved that $G$
satisfies \ref{item:Liegpd:multiplication map}, it will follow that it is an
{\em \etale} Lie groupoid.)

For~\ref{item:Liegpd:multiplication map}, the idea of the proof is this: by
construction,  the smooth structure on $G$ locally looks like that on $G\z$,
and the same is true for the smooth structure on $G\comp$. A fibred product
$B_{1}\fpsr B_{2}$ of two bisections corresponds smoothly to (a subset of)
$B_{1}$ in $ G$ which corresponds to $s(B_{1})$ in $ G\z$. On the other hand,
we can identify the bisection $B_{1}B_{2}$ in $G$ with
$s(B_{1}B_{2})\subseteq s(B_{2})$ in $G\z$, and under these local
identifications, the multiplication map $G\comp\to G$ just becomes the map
\[
    s|_{B_{2}}\circ r|_{B_{2}}\inv\colon s(B_{1})\cap r(B_{2}) \to s(B_{1}B_{2})\subseteq s(B_{2}),
\]
which is smooth by assumption. To make this idea more precise, fix
$(\gamma_{1},\gamma_{2})\in G\comp$.
    Let $(W, \psi)$ be a chart around
    $s(\gamma_{1})=r(\gamma_{2})$.
    By shrinking $W$, we may assume without loss of generality that there exist smooth open bisections $B_{i}$ around
    $\gamma_{i}$
    such that $s(B_{1})=W=r(B_{2})$. Since we have already shown that $s$ and $r$ are local diffeomorphisms, Lemma~\ref{lem:structure on Gcomp in general} describes an atlas for $G\comp$, and in particular shows that $\Phi\colon B_{1} \fpsr B_{2}\to \mbb{R}^n$ defined by
        $\Phi(\eta_{1},\eta_{2})\coloneqq \psi(s(\eta_{1}))=\psi(r(\eta_{2}))$ is a chart around $(\gamma_{1}, \gamma_{2})$.

    Let $(W',\psi')$ be any chart around $s(\gamma_{2})$.
    By shrinking its domain and by shrinking $B_{2}$, we can assume without loss of generality that $\dom(\psi')=s(B_{2})=s(B_{1}B_{2})$.\footnote{To maintain the equality $s(B_{1})=W=r(B_{2})$, we must likewise shrink $B_{1}$ and $W$, and we adjust $\Phi$ accordingly. None of this changes that $B_{i}$ is an open bisection around $\gamma_{i}$: For $i=2$, this is by construction, and for $i=1$, since $\gamma_{2}\in B_{2}$, we have $r(\gamma_{2})\in r(B_{2}) = s(B_{1}) $, so $\gamma_{1}$
    remains an element of the (potentially shrunken) $B_{1}$.} Since the set $B_{1}B_{2}$ is also an open bisection of $G$ \cite[Proposition~3.8]{Exel:BBMS}, the map $\varphi\colon B_{1}B_{2}\to \mbb{R}^m$ given by $\gamma\mapsto \psi'(s(\gamma))$, is a smooth chart around $\gamma_{1}\gamma_{2}$
    by Remark~\ref{rmk:cS}. (This does not require that $B_{1}B_{2}$ is a smooth bisection.)

    Now, to see that the multiplication map $M \colon B_{1} \fpsr B_{2} \to B_{1}B_{2}$ is smooth at  $(\gamma_{1},\gamma_{2})$, consider the composition
    \[
    \begin{tikzcd}[row sep = tiny]
        \varphi \circ M \circ \Phi\inv \colon &[-3.5em] \mbb{R}^n
        \ar[r, "\Phi\inv"]
        &
        B_{1} \fpsr B_{2}
        \ar[r, "M"]
        &
        B_{1}B_{2}
        \ar[r, "\varphi"]
        &\mbb{R}^n
        \\
        &\psi(r(\eta_{2}))
        \ar[r, mapsto]
        &
        (\eta_{1}, \eta_{2})
        \ar[r, mapsto]
        &
        \eta_{1}\eta_{2}
        \ar[r, mapsto]
        &
        \psi'(s(\eta_{1}\eta_{2}))=
        \psi'(s(\eta_{2})).
    \end{tikzcd}
    \]
    In short, this is the map from $\psi(r(B_{2}))$ to $\psi'(s(B_{2}))$ satisfying  $\psi(r(\eta_{2}))\mapsto \psi'(s(\eta_{2}))$ for all $\eta_{2} \in B_{2}$.
    By assumption, $G$ acts smoothly on $G\z$, so by Lemma~\ref{lem:lcdiff holds for every bisection}, the map
    \[
        s|_{B_{2}} \circ r|_{B_{2}}\inv \colon r(B_{2})\to s(B_{2})
    \]
    is a diffeomorphism. Since $\psi$ and $\psi'$ are smooth by assumption, we conclude that the composition $\varphi\circ M \circ \Phi\inv$ is a smooth map. Since $\Phi$ is a chart around  $(\gamma_{1},\gamma_{2})$ and $\varphi$ is a chart around $\gamma_{1}\gamma_{2} = M(\gamma_{1},\gamma_{2})$,
    this proves that $M$ is a smooth map.
By \cite[Proposition 1.1.5]{Mackenzie:2005:book}, the inversion map is
smooth, and therefore $G$ is a Lie groupoid.

For uniqueness of this smooth structure, observe that any smooth structure on
$G$ under which it becomes an \etale\ Lie groupoid is in particular a smooth
structure on $G$ for which $r \colon G \to G\z$ is a local diffeomorphism; so
the uniqueness follows from Proposition~\ref{prop:cR}.
\end{proof}

\begin{lemma}\label{lem:lcdiff => all bisections are smooth}
    Suppose that $G$ is an \etale\ Lie groupoid.
    Then every open bisection of $G$ is a smooth bisection with respect to this smooth structure. In particular,
   if $B_{1}$ and $B_{2}$ are smooth bisections such that $s(B_{1}) \cap r(B_{2}) \not= \emptyset$, then $B_{1}B_{2}$ is also a smooth bisection.
\end{lemma}

\begin{proof}
    Fix an open bisection $B$. We must show that $r|_{B}\colon B\to r(B)$ is smooth with smooth inverse. It then follows immediately that, for each $B$, $s|_{B}\colon B\to s(B)$ is smooth with smooth inverse since $ I (B)=B\inv$ is also an open bisection and since $s|_{B}=r|_{B\inv}\circ  I $ and $(s|_{B})\inv = I  \circ (r|_{B\inv})\inv$ are compositions of smooth maps.

    Take any $\gamma\in B$, and let $(W,\psi)$ be a chart around $r(\gamma)$; we can assume without loss of generality that $W\subseteq r(B)$, and we let $B'\coloneqq r\inv (W) \cap B$, which is an open bisection around $\gamma$. By Proposition~\ref{prop:lcdiff iff etale Lie gpd}, \mbox{\ref{item:TFAE:Lie gpd}${}\implies{}$\ref{item:TFAE:explicit Lie gpd}}, the map $\varphi\colon B'\to \mbb{R}^n$ given by $\eta\mapsto \psi(r(\eta))$ is a chart around $\gamma$.     Consider the composition
    \[
    \begin{tikzcd}[row sep =  tiny]
        \psi \circ r|_{B'} \circ \varphi\inv \colon
        &[-3.5em]
        \mbb{R}^n \supseteq \varphi(B')
        \ar[r,"\varphi\inv"]
        &
        B'
        \ar[r, "r|_{B'}"]
        &
        r(B') \subseteq W
        \ar[r,"\psi"]
        &
        \mbb{R}^n
        \\
        &\psi(r(\eta))
        \ar[r,mapsto]
        &
        \eta
        \ar[r, mapsto]
        &
        r(\eta)
        \ar[r,mapsto]
        &
        \psi(r(\eta)).
    \end{tikzcd}
    \]
    This is just the identity map, which is trivially smooth. We have shown that, for each point  $\gamma$ in the domain of $r|_{B}$, there exists a neighbourhood $B'$ of $\gamma$ and smooth charts $\psi$ around $r(\gamma)$ and $\varphi$ around $\gamma$ such that $\psi\circ r|_{B'}\circ \varphi\inv$ is smooth; so by definition, $r|_{B}$ is a smooth map between manifolds.

    Since $B$ is a bisection, we can play the same game ``backwards'': start with an arbitrary $u\in r(B)$ with chart $(W,\psi)$, lift $u$ to a unique $\gamma\in B$,
    let $W'\coloneqq  r(B)\cap W$, and then consider the composition
    \[
    \begin{tikzcd}[row sep =  tiny]
        \varphi \circ r|_{B}\inv \circ \psi\inv \colon
        &[-3.5em]
        \mbb{R}^n
        &
        \ar[l,"\varphi"']
        B
        &
        \ar[l, " r|_{B}\inv"']
        r(B) \supset W '
        &
        \ar[l,"\psi\inv"']
        \psi(W')\subseteq \mbb{R}^n
        \\
        &\psi(r(\eta))
        \ar[r,mapsfrom]
        &
        h
        \ar[r, mapsfrom]
        &
        r(\eta)
        \ar[r,mapsfrom]
        &
        \psi(r(\eta))
    \end{tikzcd}
    \]
    This shows that $ r|_{B}\inv$ is smooth.

    The final claim follows since products of open bisections are open bisections \cite[Proposition~3.8]{Exel:BBMS}.
\end{proof}

\section{Twists over Lie groupoids}\label{sec:twists}

This section contains the bulk of the technical work of the paper. We consider topological twists $\ses$ in which $G$ is a Lie groupoid. Although our primary application is when $G$ is \etale, we do not require this hypothesis for most of the results in this section. We define what it means for the twist to be a {\em Lie twist} and establish a number of structural consequences of the definition. We show, as a reality check, that the topologically-trivial twist over a Lie groupoid determined by a smooth normalised $\mbb{T}$-valued $2$-cocycle is a Lie twist.%
\footnote{%
    We remind the reader that the group $\mathbb{T}$ admits a unique smooth structure making it a Lie group, namely the one determined by local logarithm functions (see, for example \cite[Exercise~20-11(c)]{Lee:Intro}); we always regard it as carrying that smooth structure.%
}%

We identify three conditions \PropB, \PropE~and~\PropC\ satisfied by any
collection of local smooth sections of a Lie twist (see
Definition~\ref{dfn:theProps}). Our main technical result, Theorem~\ref{thm:E
is mfd}, shows that, conversely, any family of local sections
satisfying~\PropB\ and with support covering $G$, uniquely determines a
smooth principal bundle structure on the twist. Our main result about Lie
twists, Theorem~\ref{thm:E is Lie} builds on this to shows that Conditions
\PropE~and~\PropC\ characterise when the twist is a Lie twist over $G$ under
this smooth principal bundle structure.

In a brief final subsection, Section~\ref{ssec:Lie over etale}, we restrict
attention to the situation where $G$ is an \etale\ Lie groupoid, and
reinterpret~\PropB, \PropC~and~\PropD\ in terms of smoothness of functions
defined on $G\z$ when the local sections in question are all supported on
bisections.

 \begin{defn}\label{dfn:twist2} Let $G$ be a locally compact Hausdorff groupoid, and regard $\mbb{T} \times G\z$ as a trivial group bundle with fibres $\mbb{T}$.  A {\em twist} 
 over $G$ consists of a locally compact Hausdorff groupoid $\twist$ and groupoid homomorphisms $\iota$, $\pi$ such that
  \[\ses\]
is a central groupoid extension, which means that
\begin{enumerate}[label=\textup{(T\arabic*)}]
    \item   $\iota\colon \mbb{T} \times G\z\to \pi\inv (G\z)$ is a
        homeomorphism that satisfies $\iota(1,\pi(y)) =y$ for all $y\in
        \twist\z$, where $\pi\inv (G\z)$ has the subspace topology from
        $\twist$;
    \item $\pi$ is a continuous, open surjection; and
    \item $\iota(z, \pi(r(e)))e = e\iota(z, \pi(s(e)))$ for all $e\in
        \twist$ and $z\in\mbb{T}$.
\end{enumerate}
\end{defn}

We remind the reader that the above definition implies that $\pi$ restricts
to a homeomorphism of~$\twist\z$ onto $G\z$ with inverse given by
$\iota|_{\{1\} \times G\z}$. It is well known that $\pi$ is a proper map and
a topological principal $\mbb{T}$-bundle map in the sense of
Definition~\ref{dfn:principal bdl}; see the discussion in
\cite{CDGaHV:2024:Nuclear}. The existence of local trivialisations around all
points of $ G$ implies that {\em any} continuous section of $\pi$ gives rise
to a topological local trivialisation, as explained in
Lemma~\ref{lem:sections => triv's}.

We will be interested in understanding what we call Lie twists over Lie
groupoids. The fundamental idea is that a Lie twist should be to its
underlying Lie groupoid as a topological twist is to its underlying
topological groupoid. So, roughly speaking, a Lie twist should be a
topological twist in which both groupoids are Lie groupoids, and the maps
respect their smooth structures. Our definition is related to, but slightly
more rigid than, \cite[Definition~4.1]{BX:2003} as discussed below.

\begin{defn}\label{dfn:Lie twist}
Let $G$ be a Lie groupoid. Regard $\mbb{T} \times G\z$ as a trivial smooth
group bundle with fibres~$\mbb{T}$. A {\em Lie twist}\, $\sestriple$ over $G$
consists of a Lie groupoid $\twist$ and groupoid homomorphisms $\iota \colon
\mbb{T} \times G\z \to \twist$ and $\pi \colon \twist \to G$ under which
$\twist$ is a topological twist as above and, in addition,
\begin{enumerate}[label=\textup{(LT\arabic*)}]
    \item\label{item:LT:pi submersion} the map $\pi \colon \twist \to G$ is
        a submersion, and
    \item\label{item:LT:iota smooth} the map $\iota \colon \mbb{T} \times
        G\z\to\twist$ is smooth.
\end{enumerate}
\end{defn}

Our definition is more restrictive than the notion of an $S^1$-central
extension in \cite[Definition~4.1]{BX:2003} because our definition implies
that $\iota \colon \mbb{T} \times G\z \to \pi\inv (G\z)$ is a diffeomorphism,
so $\PB{\pi\inv (G\z)}{\pi}{G\z}$ is the {\em trivial} principal
$\mbb{T}$-bundle over $G\z$ (see Lemma~\ref{lem:free stuff about Lie twists}
below). There are also cosmetic differences; for example, we do not insist
that $G\z$ and $\twist\z$ are equal, but identify them via $\pi$.

Our definition does not explicitly insist that $\pi\inv (G\z)$ is an embedded
submanifold, that $\iota$ is a diffeomorphism onto this embedded submanifold
or that $\PB{\twist}{\pi}{G}$ is a smooth principal $\mbb{T}$-bundle with
respect to the natural $\mbb{T}$-action, but this follows as we demonstrate
in the next two lemmas.

\begin{lemma}\label{lem:smooth principal bundle}
    Suppose that $ \sestriple $ is a topological twist over a groupoid $G$, and that $\twist$, $G\z$, and $G$ are manifolds (not necessarily Lie groupoids). Suppose that the left $\mbb{T}$-action on $\twist$ given by $z \cdot e = \iota(z, r(e))e$ is smooth and that $\pi$ is a submersion.
    Then $\PB{\twist}{\pi}{G}$ is a {\em smooth} principal $\mbb{T}$-bundle.
\end{lemma}

\begin{proof}
    By the Quotient Manifold Theorem \cite[Theorem~21.10]{Lee:Intro}, since $\mbb{T}$ is compact and acts smoothly and freely on $\twist$, the quotient space $\mbb{T}\backslash \twist$ has a manifold structure with respect to which $\PB{\twist}{q}{\mbb{T}\backslash \twist}$ is a smooth principal $\mbb{T}$-bundle. Moreover, the smooth structure on $\mbb{T}\backslash \twist$ is the {\em unique} smooth structure on $\mbb{T}\backslash\twist$ for which the quotient map $q\colon \twist\to \mbb{T}\backslash \twist$ is a submersion.

    The map $\pi$ induces a homeomorphism $\tilde{\pi}\colon \mbb{T}\backslash \twist\approx G$ given by $\tilde{\pi}(\mbb{T}\cdot e) = \pi(e)$.
    Let $M$ denote the set $\mbb{T}\backslash \twist$ when equipped with the manifold structure that the homeomorphism $\tilde{\pi}$ induces from $G$, i.e., $\tilde{\pi}\colon M\to G$ is a diffeomorphism by construction.
    Consider the following commutative diagram:
    \[
    \begin{tikzcd}
        & \twist \ar[ld, dashed, "\tilde{\pi}\inv\circ \pi"']\ar[rd, "\pi\text{  submersion}"]&
        \\
        M \ar[rr, "\tilde{\pi}", "\text{diffeo}"', <->] && G
    \end{tikzcd}
    \]
    By definition of the smooth structure on $M$, the map $\tilde{\pi}\inv\circ \pi$ is then also a submersion. By definition of $\tilde{\pi}$, we have $(\tilde{\pi}\inv\circ \pi)(e) = \mbb{T}\cdot e$. That is, $\tilde{\pi}\inv\circ \pi=q$ is the quotient map. By uniqueness of the smooth structure on $\mbb{T}\backslash \twist$, we conclude that $\mbb{T}\backslash \twist=M$, i.e., $\tilde{\pi}\colon \mbb{T}\backslash \twist\approx G$ is a diffeomorphism, not just a homeomorphism. Since $\PB{\twist}{q}{\mbb{T}\backslash \twist}$  is a smooth principal bundle, it follows that $\PB{\twist}{\pi}{G}$ is a smooth principal bundle.
\end{proof}

\begin{lemma}\label{lem:free stuff about Lie twists}
Suppose that $ \ses $ is a Lie twist. Then the following hold:
    \begin{enumerate}[label=\textup{(LT\arabic*)}, start=3]
    \item\label{item:LT:smooth action} The $\mbb{T}$-action on $\twist$
        given by $z \cdot e = \iota(z, r(e))e$ is smooth;
    \item\label{item:LT:principal bundle} $\PB{\twist}{\pi}{G}$ is a smooth
        principal $\mbb{T}$-bundle with respect to this action;
    \item\label{item:LT:embedded submfd} $\pi\inv (G\z)$ is an embedded
        submanifold of $\twist$; and
    \item\label{item:LT:iota diffeo} $\iota$ is a diffeomorphism of the
        trivial principal $\mbb{T}$-bundle onto this embedded submanifold.
    \end{enumerate}
\end{lemma}
\begin{proof}
Since $\twist$ is a Lie groupoid, multiplication in $\twist$ and the range
map are smooth. Since $\iota$ is smooth by assumption, the $\mbb{T}$-action
is a composition of smooth maps, and hence itself smooth. Since $\pi$ is a
submersion by assumption, it then follows from Lemma~\ref{lem:smooth
principal bundle} that $\PB{\twist}{\pi}{G}$ is a smooth principal bundle.
Since $G\z \subseteq G$ is an embedded submanifold, it follows from
\cite[Corollary 6.31]{Lee:Intro} that $\pi\inv (G\z)$ is an embedded
submanifold. By \cite[Theorem~5.29]{Lee:Intro}, the map $\iota$ is smooth
from $\mbb{T} \times G\z$ to $\pi\inv (G\z)$. Since $\iota$ is also a
homeomorphism between these principle $\mbb{T}$-bundles,
Lemma~\ref{lem:principal bundles are rigid} implies that it is a
diffeomorphism.
\end{proof}

The following lemma and corollary serve as a ``reality-check" that the
definition of a Lie twist, and the assumptions of Lemma~\ref{lem:smooth
principal bundle}, are reasonable, and in particular that smooth $2$-cocycles
on Lie groupoids give rise to Lie twists.

Recall that a (continuous) \emph{$2$-cocycle} on a topological groupoid $G$
is a continuous function $\mathbf{c} \colon G\comp \to \mbb{T}$ such that
$\mathbf{c}(\alpha,\beta)\mathbf{c}(\alpha\beta,\gamma) =
\mathbf{c}(\alpha,\beta\gamma)\mathbf{c}(\beta,\gamma)$ for all composable
triples $(\alpha,\beta,\gamma)$. The cocycle $\mathbf{c}$ is
\emph{normalised} if $\mathbf{c}(r(\gamma),\gamma) = 1 = \mathbf{c}(\gamma,
s(\gamma))$ for all $\gamma$.

\begin{lemma}\label{lem:cocycle:smooth principal bundle}
    Suppose that $G$ is a topological groupoid that is a manifold (but not necessarily a Lie groupoid), that $G\z$ is an embedded submanifold, and that $\mathbf{c}\colon G\comp \to \mbb{T}$ is a normalised $2$-cocycle. Let $\twist_{\mathbf{c}}\coloneqq \mbb{T} \times G$ be the groupoid with multiplication defined by
    \[
        (z_{1}, \gamma_{1})(z_{2}, \gamma_{2})
        =
        \bigl(\mathbf{c}(\gamma_{1},\gamma_{2})z_{1}z_{2}, \gamma_{1}\gamma_{2}\bigr)
        \text{ for } (\gamma_{1},\gamma_{2})\in G\comp
    \]
    and inversion by
    \[
        (z, \gamma)\inv =
        \bigl(
            \overline{\mathbf{c}(\gamma,\gamma\inv)z}, \gamma\inv
        \bigr).
    \]
    \begin{enumerate}[label=\textup{(\arabic*)}]
        \item\label{item:cocycle:smooth principal bundle} If
            $\twist_{\mathbf{c}}$ is equipped with the product-manifold
            structure, then its left $\mbb{T}$-action defined by $z \cdot
            (w, \gamma) = (zw, \gamma)$
    is smooth and the natural projection map $\pi\colon
    \twist_{\mathbf{c}}\to G$ is a submersion.
    \item\label{item:cocycle:twist Lie}  If $G$ is a Lie groupoid and
        $\mathbf{c}$ is smooth, then  $\twist_{\mathbf{c}}$ is also a Lie
        groupoid.
    \end{enumerate}
\end{lemma}

\begin{remark}
It is important that $\mathbf{c}$ is normalised. Otherwise $(1, r(\gamma))(1,
\gamma) = (\mathbf{c}(r(\gamma),\gamma), \gamma) \not= (1, \gamma)$, so the
inclusion $\iota \colon \mathbb{T} \times G\z \to \twist_{\mathbf{c}}$ does
not carry units to units.
\end{remark}

Applying  Lemma~\ref{lem:smooth principal bundle}, we get:
\begin{corollary}\label{cor:cocycle:smooth principal bundle}
     Suppose that $G$ is a topological groupoid that is a manifold, that $G\z$ is an embedded submanifold, and that $\mathbf{c}\colon G\comp \to \mbb{T}$ is a normalised continuous $2$-cocycle.
    \begin{enumerate}[label=\textup{(\arabic*)}]
        \item $\PB{\twist_{\mathbf{c}}}{\pi}{G}$ is a smooth principal
            $\mbb{T}$-bundle.
    \item If $G$ is a Lie groupoid and $\mathbf{c}$ is smooth, then $
        \mbb{T}\times G\z \overset{\iota}{\to} \twist_{\mathbf{c}}
        \overset{\pi}{\to} G$ is a Lie twist.
    \end{enumerate}
\end{corollary}
\begin{proof}[Proof of Lemma~\ref{lem:cocycle:smooth principal bundle}]
    As a manifold, $\twist_{\mathbf{c}}$ is the Cartesian-product manifold, so it is a (trivial) smooth principal $\mbb{T}$-bundle for the action $z \cdot (w, g) = (zw,g)$ irrespective of the properties of $\mathbf{c}$. In particular, $\pi$ is a smooth fibration and the action is smooth; and
    $\pi$ is a submersion since, in charts, it is just projection onto the first $m-1$ coordinates in Euclidean space.

    It is well-known that $ \twist_{\mathbf{c}}$ is a topological groupoid, so we only need to check the conditions about differentiability.
    Since the Cartesian product of two Lie groupoids is again a Lie groupoid, we know that the inclusion of $\{1\}\times G\z = (\mbb{T}\times G)\z$ into $\mbb{T}\times G$ is an immersion and  that the source map $\mbb{T}\times G \to G\z$, $(z,\gamma)\mapsto s(\gamma)$, is a submersion, just like the range map. Since none of these maps distinguish between $\mbb{T}\times G$ and $\twist_{\mathbf{c}}$, we conclude that $\twist_{\mathbf{c}}$ satisfies Conditions~\ref{item:Liegpd:mfds}, \ref{item:Liegpd:inclusion map}, and~\ref{item:Liegpd:s and r} of Definition~\ref{dfn:Lie gpd}.

    The  composition map \(\twist_{\mathbf{c}}\comp \to \twist_{\mathbf{c}}\) is smooth because it is built out of smooth functions: rearranging coordinates, multiplication in $G$, applying the smooth $2$-cocycle $\mathbf{c}$, and multiplication in $\mbb{T}$.
    This proves Condition~\ref{item:Liegpd:multiplication map}.
\end{proof}

    Given a family $ \{\sigma_{\alpha} \colon U_{\alpha} \to \pi\inv (U_{\alpha})\}_{ \alpha\in\mathfrak{A}}$ of sections of a map $\pi$,
        we will often write $\{(U_{\alpha},\sigma_{\alpha})\}_{\alpha}$ for short. If there is no ambiguity regarding the sets $U_{\alpha}$, we will even write $\{\sigma_{\alpha}\}_{\alpha}$.

        The following is our key technical definition. As we shall see in Theorem~\ref{thm:E is mfd} and Theorem~\ref{thm:E is Lie}, it lays out conditions on a family of sections of a twist over a Lie groupoid under which they determine a compatible Lie-groupoid structure on the twist.

\begin{defn}\label{dfn:theProps}
    Suppose that $ \ses $ is a topological twist over a Lie groupoid $G$, that $\{U_{\alpha}\}_{\alpha\in\mathfrak{A}}$ is a collection of open subsets of $G$, and that  $ \{\sigma_{\alpha} \colon U_{\alpha} \to \pi\inv (U_{\alpha})\}_{ \alpha\in\mathfrak{A}}$ is a family of continuous sections  of $\pi$.
    We define the following properties for the family $\{\sigma_{\alpha}\}_{\alpha\in\mathfrak{A}}$.
    \begin{itemize}
        \item[\namedlabel{property:U infty}{\normalfont(U$^{\infty}$)}] For
            each $\alpha$, there exists a smooth map $k_{\alpha}\colon
            U_{\alpha}\cap G\z\to\mbb{T}$ such that $\sigma_{\alpha} (x) =
            \iota(k_{\alpha}(x), x)$ for all $x\in U_{\alpha} \cap G\z$.
        \item[\namedlabel{property:S infty}{\normalfont(S$^{\infty}$)}] For
            each $\alpha,\alpha'$, the map  $U_{\alpha}\cap U_{\alpha'}\to
            \mbb{T} \times G\z$ given by $\gamma\mapsto \iota\inv
            (\sigma_{\alpha}(\gamma)\sigma_{\alpha'}(\gamma)\inv )$ is
            smooth.
        \item[\namedlabel{property:M infty}{\normalfont(M$^{\infty}$)}] For
            each $\alpha,\alpha_{1},\alpha_{2}\in\mathfrak{A}$, the map
            $\big\{(\gamma_{1}, \gamma_{2}) \in U_{\alpha_{1}} \fpsr
            U_{\alpha_{2}} : \gamma_{1}\gamma_{2} \in U_{\alpha}\big\} \to
            \mbb{T} \times G\z$ given by $(\gamma_{1},\gamma_{2})\mapsto
            \iota\inv \bigl(\sigma_{\alpha_{1}}
            (\gamma_{1})\sigma_{\alpha_{2}} (\gamma_{2})\sigma_{\alpha}
            (\gamma_{1}\gamma_{2})\inv \bigr)$ is smooth.
        \item[\namedlabel{property:I infty}{\normalfont(I$^{\infty}$)}] For
            each $\alpha,\alpha'\in\mathfrak{A}$, the map $U_{\alpha}\cap
            U_{\alpha'}\inv
            \to \mbb{T} \times G\z$ given by $\gamma\mapsto \iota\inv
            \bigl(\sigma_{\alpha}(\gamma)\sigma_{\alpha'}
            (\gamma\inv)\bigr)$ is smooth.
    \end{itemize}
\end{defn}

\begin{remark}\label{rmk:T-conditions}
Resume the notation of Definition~\ref{dfn:theProps}. Consider the
$\mbb{T}$-valued functions $f_{\alpha,\alpha'}  \colon U_{\alpha}\cap
U_{\alpha'}\to \mbb{T}$, $g_{\alpha, \alpha_{1}, \alpha_{2}} \colon
\big\{(\gamma_{1}, \gamma_{2}) \in U_{\alpha_{1}} \fpsr U_{\alpha_{2}} :
\gamma_{1}\gamma_{2} \in U_{\alpha}\big\} \to \mbb{T}$, and
$h_{\alpha,\alpha'} \colon U_{\alpha} \cap U_{\alpha'}\inv \to \mbb{T}$
defined implicitly by
\begin{align*}
\iota\bigl(f_{\alpha,\alpha'}(\gamma), r(\gamma)\bigr)
    &= \sigma_{\alpha}(\gamma)\sigma_{\alpha'}(\gamma)\inv,
    \\
\iota\bigl(g_{\alpha, \alpha_{1}, \alpha_{2}}(\gamma_{1},\gamma_{2}), r(\gamma_{1})\bigr)
    &= \sigma_{\alpha_{1}} (\gamma_{1})\sigma_{\alpha_{2}} (\gamma_{2}) \sigma_{\alpha}(\gamma_{1}\gamma_{2})\inv,\text{ and}
    \\
\iota\bigl(h_{\alpha,\alpha'}(\gamma), r(\gamma)\bigr)
    &= \sigma_{\alpha}(\gamma)\sigma_{\alpha'} (\gamma\inv).
\end{align*}
Then the map of~\PropB\ (respectively \PropC, respectively \PropD) is smooth
if and only if $f_{\alpha,\alpha'}$ (respectively $g_{\alpha, \alpha_{1},
\alpha_{2}}$, respectively $h_{\alpha,\alpha'}$) is smooth. In particular,
\PropB\ (respectively \PropC,
respectively \PropD)
holds if and only if condition~\PropEB\ (respectively \PropEC,
respectively \PropD)
 below holds:
\begin{itemize}
    \item[\namedlabel{property:S infty
        T}{\normalfont(S$^{\infty}_{\mathbb{T}}$)}] For all
        $\alpha,\alpha'$, the map  $f_{\alpha,\alpha'}$ is smooth.
    \item[\namedlabel{property:M infty
        T}{\normalfont(M$^{\infty}_{\mathbb{T}}$)}] For all
        $\alpha,\alpha_{1},\alpha_{2}\in\mathfrak{A}$, the map $g_{\alpha,
        \alpha_{1}, \alpha_{2}}$ is smooth.
    \item[\namedlabel{property:I infty
        T}{\normalfont(I$^{\infty}_{\mathbb{T}}$)}] For all
        $\alpha,\alpha'$, the map $h_{\alpha,\alpha'}$ is smooth.
\end{itemize}
Since Condition~\PropE\ is already is given in terms of $\mbb{T}$-valued
functions, any reasonable definition of a
Condition~(U$^{\infty}_{\mathbb{T}}$) would coincide with~\PropE.
\end{remark}

\begin{remark}
Intuitively, each of the four conditions in Definition~\ref{dfn:theProps}
represents a property that the $\sigma_{\alpha}$ must satisfy if $\twist$ is
a Lie twist and the $\sigma_{\alpha}$ are all smooth. Firstly, \PropE\
corresponds to the requirement that $\iota$ is a diffeomorphism onto
$\pi\inv(G\z)$: it says that if we use $\iota$ to identify $\pi\inv(G\z)$
with $\mbb{T} \times G\z$ and hence sections $G\z \to \pi\inv(G\z)$ with
functions $G\z \to \mbb{T}$, then where they are supported on $G\z$, the
$\sigma_{\alpha}$ are smooth functions into $\mbb{T}$. The letter U stands
for ``units." With \PropE\ in hand, we can measure relative smoothness of
sections when the products of their values land in $\pi\inv(G\z)$, and this
motivates the remaining properties:

Condition~\PropB\ corresponds to compatibility of the sections
$\sigma_{\alpha}$: if $\sigma_{\alpha}$ and $\sigma_{\alpha'}$ are both
smooth, then, in the Lie groupoid $\twist$, their difference $\sigma_{\alpha}
\sigma_{\alpha'}\inv$ must be smooth where defined; the S stands for
``smooth." Similarly, in the situation of \PropC, if $\sigma_{\alpha_{1}},
\sigma_{\alpha_{2}}$ and $\sigma_{\alpha}$ are smooth, then $(\gamma_{1},
\gamma_{2}) \mapsto \sigma_{\alpha}(\gamma_{1}\gamma_{2})$ and
$(\gamma_{1},\gamma_{2}) \mapsto
\sigma_{\alpha_{1}}(\gamma_{1})\sigma_{\alpha_{2}}(\gamma_{2})$ are smooth
maps; the M stands for ``multiplication." Finally, if $\sigma_{\alpha}$ and
$\sigma_{\alpha'}$ are smooth, then in the situation of \PropD, $\gamma
\mapsto \sigma_{\alpha}(\gamma)$ and $\gamma\mapsto
\sigma_{\alpha'}(\gamma\inv)\inv$ define smooth maps and hence, using
smoothness of the multiplication in $\twist$, so is their product; the I
stands for ``inverses." As discussed in Remark~\ref{rmk:inversion
auto-smooth}, smoothness of inversion follows from smoothness of the other
operations, so we typically do not need to verify~\PropD independently; but
to make the connection between properties of Lie groupoids and elements of
associated $C^*$-algebras later, it is useful to name the condition
nevertheless.

We shall see later not only that Conditions~\PropE, \PropB~and~\PropC are
necessary, but also that they are sufficient in the sense that any such
family of sections determines a unique smooth structure on $\twist$, with
respect to which it is a Lie twist  over $G$ in the sense of
Definition~\ref{dfn:Lie twist}.
\end{remark}

Condition~\PropC\ looks a little unwieldy because the domain
$\big\{(\gamma_{1}, \gamma_{2}) \in U_{\alpha_{1}} \fpsr U_{\alpha_{2}} :
\gamma_{1}\gamma_{2} \in U_{\alpha}\big\}$ of the map in question is not
easily expressed in terms of the sets $U_{\alpha}$; the point is that the
inverse image in $G\comp$ of a given subset of $G$ under the multiplication
map is hard to picture. Importantly, though, since multiplication is
continuous on the fibred product $U_{\alpha_{1}} \fpsr U_{\alpha_{2}}$, the
domain of the map is an open subset of $U_{\alpha_{1}} \fpsr U_{\alpha_{2}}$
and hence of $G\comp$ (see the proof of Lemma~\ref{lem:PropC matching}
below). Typically we will be interested in the situation where
$\{U_{\alpha}\}_{\alpha \in \mathfrak{A}}$ is a base for the topology on $G$,
and in this case the condition can be rephrased more nicely in terms of the
$U_{\alpha}$ as follows.

\begin{lemma}\label{lem:PropC matching}
Suppose that $ \ses $ is a topological twist over a Lie groupoid $G$ and that
$\{U_{\alpha}\}_{\alpha\in\mathfrak{A}}$ is a collection of open subsets of
$G$. If a given family $\{\sigma_{\alpha} \colon U_{\alpha} \to \pi\inv
(U_{\alpha})\}_{\alpha \in \mathfrak{A}}$
    of continuous sections satisfies \PropC, then it also satisfies:
\begin{itemize}
    \item[\namedlabel{property:M infty B}{\normalfont(M$^{\infty}_B$)}] For
        all $\alpha, \alpha_{1}, \alpha_{2} \in \mathfrak{A}$, the map
        $U_{\alpha_{1}} \fpsr U_{\alpha_{2}} \to \mbb{T} \times G\z$,
        $(\gamma_{1},\gamma_{2})\mapsto \iota\inv \bigl(
        \sigma_{\alpha_{1}} (\gamma_{1})\sigma_{\alpha_{2}}
        (\gamma_{2})\sigma_{\alpha}(\gamma_{1}\gamma_{2})\inv \bigr)$, is
        smooth.
\end{itemize}
If $\{U_{\alpha}\}_{\alpha \in \mathfrak{A}}$ is a base for the topology of
$G$ and the $\sigma_{\alpha}$ satisfy \PropB~and~\PropCC, then they
satisfy~\PropC.
\end{lemma}
\begin{proof}
If $\alpha_{1}, \alpha_{2}, \alpha$ satisfy $U_{\alpha_{1}}
U_{\alpha_{2}}\subseteq U_{\alpha}$, then $\big\{(\gamma_{1}, \gamma_{2}) \in
U_{\alpha_{1}} \fpsr U_{\alpha_{2}} \colon \gamma_{1}\gamma_{2} \in
U_{\alpha}\big\} = U_{\alpha_{1}} \fpsr U_{\alpha_{2}}$. So \PropC\ implies
\PropCC.

Now suppose that $\{U_{\alpha}\}_{\alpha\in\mathfrak{A}}$ is a base for the
topology of $G$ and that the $\sigma_{\alpha}$ satisfy \PropB\ and \PropCC.
Fix $\alpha, \alpha_{1}, \alpha_{2}$ such that $U_{\alpha_{1}} U_{\alpha_{2}}
\cap U_{\alpha} \not= \emptyset$. Fix a point $(\gamma_{1}, \gamma_{2}) \in
U_{\alpha_{1}} \fpsr U_{\alpha_{2}}$ such that $\gamma_{1}\gamma_{2} \in
U_{\alpha}$. Since the multiplication map $M \colon G\comp \to G$ is
continuous, the set $M\inv(U_{\alpha})$ is open. The set $U_{\alpha_{1}}
\fpsr U_{\alpha_{2}}$ is open by definition of the topology on $G\comp$, so
the intersection
\[
    M\inv(U_{\alpha})\cap (U_{\alpha_{1}} \fpsr U_{\alpha_{2}}) = \big\{(\gamma_{1}, \gamma_{2}) \in U_{\alpha_{1}} \fpsr U_{\alpha_{2}} : \gamma_{1}\gamma_{2} \in U_{\alpha}\big\}
\]
is open. Since the $U_{\alpha}$ are a base for the topology, there exist
$\beta_{1}, \beta_{2} \in \mathfrak{A}$ such that $U_{\beta_{i}}\subseteq
U_{\alpha_{i}}$ and $(\gamma_{1}, \gamma_{2}) \in U_{\beta_{1}} \fpsr
U_{\beta_{2}} \subseteq M\inv(U_{\alpha}) \cap (U_{\alpha_{1}} \fpsr
U_{\alpha_{2}})$. In particular, $U_{\beta_{1}}  U_{\beta_{2}} \subseteq
U_{\alpha}$. By \PropB, for $i = 1,2$, the map $\omega_i \colon \eta \mapsto
\iota\inv(\sigma_{\alpha_i}(\eta)\sigma_{\beta_i}(\eta)\inv )$ is smooth from
$U_{\beta_i}$ to $\mbb{T}\times G\z $. By definition, this map satisfies
$\omega_i(\eta) \cdot \sigma_{\beta_i}(\eta) = \sigma_{\alpha_i}(\eta)$ for
all $\eta \in U_{\alpha'_i}$. By \PropCC, the map $(\eta_{1},\eta_{2})\mapsto
    \iota\inv \bigl(
        \sigma_{\beta_{1}} (\eta_{1})\sigma_{\beta_{2}} (\eta_{2})
        \sigma_{\alpha}(\eta_{1}\eta_{2})\inv
    \bigr)$
is smooth from $U_{\beta_{1}} \fpsr U_{\beta_{2}}$ to $\mbb{T} \times G\z$.
On this same domain, using centrality of the $\mbb{T}$-action, we see that
\begin{align*}
\iota\inv \bigl(
            \sigma_{\alpha_{1}} (\eta_{1})\sigma_{\alpha_{2}} (\eta_{2})
            \sigma_{\alpha}(\eta_{1}\eta_{2})\inv
        \bigr)
    &=  \iota\inv \bigl(
            (\omega_{1}(\eta_{1})\cdot\sigma_{\beta_{1}} (\eta_{1}))\,
            (\omega_{2}(\eta_{2})\cdot \sigma_{\beta_{2}} (\eta_{2}))\,
            \sigma_{\alpha}(\eta_{1}\eta_{2})\inv
        \bigr)\\
    &= \big(\omega_{1}(\eta_{1})\omega_{2}(\eta_{2}), r(\eta_{1})\big)\,
        \iota\inv \bigl(
            \sigma_{\beta_{1}} (\eta_{1})\sigma_{\beta_{2}} (\eta_{2})
            \sigma_{\alpha}(\eta_{1}\eta_{2})\inv
        \bigr).
\end{align*}
Since multiplication in $\mbb{T} \times G\z$ is smooth, it follows that
$(\eta_{1},\eta_{2}) \mapsto
    \iota\inv\bigl(
        \sigma_{\alpha_{1}} (\eta_{1})\sigma_{\alpha_{2}} (\eta_{2})
        \sigma_{\alpha}(\eta_{1}\eta_{2})\inv
    \bigr)$
is smooth on the open neighbourhood $U_{\beta_{1}} \fpsr U_{\beta_{2}}$ of
$(\gamma_{1}, \gamma_{2})$ in $M\inv(U_{\alpha}) \cap U_{\alpha_{1}} \fpsr
U_{\alpha_{2}}$. Since the point $(\gamma_{1}, \gamma_{2}) \in
M\inv(U_{\alpha}) \cap (U_{\alpha_{1}} \fpsr U_{\alpha_{2}})$ was arbitrary,
the result follows.
\end{proof}

\begin{lemma}[Motivation]\label{lem:motivation}
    Suppose that $ \sestriple $ is a Lie twist over a Lie groupoid $G$. Then
    there exists a base $\mathcal{B} = \{U_{\alpha} \}_{\alpha\in\mathfrak{A}}$ for the topology on $G$ and a family  $\{\sigma_{\alpha} \colon U_{\alpha} \to \pi\inv (U_{\alpha})\}_{\alpha \in \mathfrak{A}}$ of smooth sections that satisfies~\PropE, \PropB, \PropC~and~\PropD. If $G$ is \etale, then $\mathcal{B}$ can be chosen to consist of open bisections.
\end{lemma}
\begin{proof}
    By  Condition~\ref{item:LT:principal bundle}, $\PB{\twist}{\pi}{G}$ is a smooth
    principal bundle. Therefore, for every $\gamma\in G$, there exists a neighbourhood $V_{\gamma}$ and a
    smooth map $\psi_{\gamma}\colon V_{\gamma}\to \twist$ such that $\pi\circ \psi_{\gamma}=\mathrm{id}_{V_{\gamma}}$.
    Fix a base $\mathcal{B}_{0}$ for the topology on $G$ (to prove the final statement when $G$ is \etale, take $\mathcal{B}_{0}$ to be the collection of all open bisections of $G$). Let
    \[
    \mathfrak{A} \coloneqq \{
        (W, \gamma) : W \in \mathcal{B}_{0}\text{ and } W \subseteq V_\gamma
    \}
    \subseteq  \mathcal{B}_{0}\times G,
    \]
    and for each $(W,  \gamma) \in \mathfrak{A}$ let $U_{(W,  \gamma)} \coloneqq W$.
    If $U$ is any open subset of $G$ and $ \gamma\in U$, then since $\mathcal{B}_{0}$ is a base, there exists  $W\in \mathcal{B}_{0}$ with $ \gamma\in W\subseteq V_{ \gamma}\cap U$, so $W = U_{(W, \gamma)} $ is an element of $\mathcal{B}\coloneqq \{U_{\alpha}\}_{\alpha\in\mathfrak{A}}$. This proves that $\mathcal{B}$ is also a base for the topology of $G$.

    By definition of $\mathcal{B}$, we can choose a function $\Gamma \colon \mathfrak{A} \to G$ such that $U_{\alpha} \subseteq V_{\Gamma (\alpha)}$ for all $\alpha \in \mathfrak{A}$.
    For each $\alpha \in \mathfrak{A}$, define $\sigma_{\alpha}\coloneqq \psi_{\Gamma (\alpha)}|_{U_{\alpha}}$;
    by construction, $\sigma_{\alpha}$ is a smooth section of $\pi$.
    Since $G$ and $\twist$ are Lie groupoids, multiplication and inversion on $G$ and $\twist$ are smooth. Since each $\psi_{ \gamma}$ is smooth and each $U_{\alpha}$ is open, and using Lemma~\ref{lem:PropC matching}, it follows that for all $\alpha,\alpha', \alpha_i \in \mathfrak{A}$ with $U_{\alpha_{1}} U_{\alpha_{2}}\subseteq U_{\alpha_{3}}$, the following maps are smooth:
    \begin{align*}
         U_{\alpha}\cap U_{\alpha'} &\to \twist, & \eta &\mapsto \sigma_{\alpha}(\eta)\sigma_{\alpha'}(\eta)\inv,
         \\
         U_{\alpha_{1}} \fpsr U_{\alpha_{2}} &\to \twist, & (\eta_{1},\eta_{2})&\mapsto
          \sigma_{\alpha_{1}}(\eta_{1})\sigma_{\alpha_{2}}(\eta_{2})\sigma_{\alpha_{3}}(\eta_{1}\eta_{2})\inv
         ,
         \text{ and }
         \\
          U_{\alpha}\cap U_{\alpha'}\inv
         &\to \twist, & \eta&\mapsto
          \sigma_{\alpha}(\eta) \sigma_{\alpha'}(\eta\inv)
         .
    \end{align*}
    Since $\iota\colon \mbb{T} \times G\z \to \pi\inv (G\z)$ is a diffeomorphism
    (Condition~\ref{item:LT:iota diffeo}), composing any
    of the above smooth maps with $\iota\inv\colon \pi\inv (G\z)\to \mbb{T} \times G\z$ yields a smooth map, giving \PropB, \PropC, and~\PropD. Also, again since $\iota\colon \mbb{T} \times G\z \to \pi\inv (G\z)$ is a diffeomorphism, for each $\alpha \in \mathfrak{A}$ such that $U_{\alpha} \cap G\z \not= \emptyset$, the map $\iota\inv \circ \sigma_{\alpha}|_{U_{\alpha} \cap G\z}$ is smooth, and so composing with the smooth projection map $\mbb{T} \times G\z \to \mbb{T}$ gives a smooth map $k_{\alpha}$ as required by \PropE.
\end{proof}

In Theorem~\ref{thm:E is mfd}, we prove that, given a {\em topological} twist
$\ses$ over a Lie groupoid $G$ and a suitably large collection of sections
that satisfies~\PropB, then we can lift the smooth structure of $G$ to one on
$\twist$. In order for us to be able to replace the collection with a
slightly different family of sections, we require the following lemma.

\begin{lemma}\label{lem:refinement}
    Suppose that $ \sestriple $ is a topological twist over a Lie groupoid $G$, that
$\{(U_{\alpha},\varphi_{\alpha})\}_{\alpha\in\mathfrak{A}}$ is an atlas of
$G$, and that $\{\sigma_{\alpha} \colon U_{\alpha} \to \pi\inv
(U_{\alpha})\}_{\alpha\in\mathfrak{A}}$ is a family of continuous sections.
Suppose further that  we are given an open {\em refinement}
$\{V_\beta\}_{\beta\in\mathfrak{B}}$ of $\{U_{\alpha}\}_{\alpha}$, i.e., an
open cover of $G$ with a map $F\colon \mathfrak{B}\to\mathfrak{A}$ such that
for all $\beta\in\mathfrak{B}$, we have $V_\beta\subseteq U_{F(\beta)}$.
    For each $\beta\in\mathfrak{B}$, let $\kappa_\beta \coloneqq  \sigma_{F(\beta)}|_{V_\beta} \colon V_\beta \to \pi\inv (V_\beta)$.
    \begin{enumerate}[label=\textup{(\arabic*)}]
        \item\label{item:refinments inherit PropBCD} If
            $\{(U_{\alpha},\sigma_{\alpha})\}_{\alpha}$ satisfies \PropE\
            (respectively \PropB, \PropC, \PropD), then so does
            $\{(V_\beta,\kappa_{\beta})\}_{\beta}$.
        \item\label{item:refinements inherit smoothness} Suppose that
            $\twist$ is a manifold and that   $\sigma_{\alpha}$ is smooth
            for each $\alpha$. Then  $\kappa_{\beta}$ is smooth for each
            $\beta$.
        \item\label{item:refinements preserve smoothness} Suppose that
            $\{(U_{\alpha},\sigma_{\alpha})\}_{\alpha}$ satisfies \PropB.
            Suppose further that $\twist$ is a manifold and that, with
            respect to that smooth structure, its multiplication is smooth,
            the map $\iota\colon \mbb{T} \times G\z \to \pi\inv (G\z)$ is a
            diffeomorphism, and $\kappa_{\beta}$ is smooth for each
            $\beta$. Then $\sigma_{\alpha}$ is smooth for each $\alpha$.
    \end{enumerate}
\end{lemma}
\begin{proof}
    \ref{item:refinments inherit PropBCD} Fix $\beta, \beta', \beta_i \in \mathfrak{B}$ with $V_{\beta_{1}} V_{\beta_{2}} \subseteq V_{\beta_3}$ and let $\alpha = F(\beta)$, $\alpha' = F(\beta')$ and $\alpha_i = F(\beta_i)$. Then the maps
    \begin{align*}
         V_\beta \cap G\z &\to \mbb{T}, & x &\mapsto k_{\alpha}(x),
         \\
         V_{\beta}\cap V_{\beta'} &\to \twist, & h &\mapsto \kappa_{\beta}(h)\kappa_{\beta'}(h)\inv,
         \\
         V_{\beta_{1}} \fpsr V_{\beta_{2}} &\to \twist, & (h_{1},h_{2})&\mapsto
            \kappa_{\beta_{1}}(h_{1})\kappa_{\beta_{2}}(h_{2})
            \kappa_{\beta_{3}}(h_{1}h_{2})\inv,
         \text{ and }
         \\
         V_{\beta}\cap V_{\beta'}\inv &\to \twist, & h&
         \mapsto
           \kappa_{\beta}(h) \kappa_{\beta'}(h\inv)
    \end{align*}
    are the restrictions of the corresponding maps $U_{\alpha}\cap G\z\to\mbb{T}$, $U_{\alpha}\cap U_{\alpha'} \to \twist$, $U_{\alpha_{1}} \fpsr U_{\alpha_{2}} \to \twist$, and $U_{\alpha}\cap U_{\alpha'}\inv \to \twist$ determined by $\{\sigma_{\alpha}\}_{\alpha}$, which are all smooth by assumption. Since multiplication in $G$ is open (Exercise~1.2.6 and its solution (p.~353) of \cite{Wil2019}) and since restriction of a smooth map to an open set is again smooth, we obtain~\ref{item:refinments inherit PropBCD}.

    \ref{item:refinements inherit smoothness} Take $\beta\in \mathfrak{B}$ and let $\alpha=F(\beta)$. Then $\sigma_{\alpha}\colon U_{\alpha}\to \pi\inv (U_{\alpha})$ is smooth by assumption. Since $V_\beta$ is an open subset of $U_{\alpha}$, it follows that $\sigma_{\alpha}|_{V_\beta}\colon V_\beta\to \pi\inv (U_{\alpha})$ is likewise smooth. As $\pi\inv (V_\beta)$ is an embedded
    submanifold of $\pi\inv (U_{\alpha})$ and as the image of $\sigma_{\alpha}|_{V_\beta}$ is contained in $\pi\inv (V_\beta)$, this implies that $\kappa_\beta =  \sigma_{\alpha}|_{V_\beta} \colon V_\beta \to \pi\inv (V_\beta)$ is smooth.

    \ref{item:refinements preserve smoothness} Pick $\alpha\in \mathfrak{A}$. To see that $\sigma_{\alpha}$ is smooth, fix $\gamma\in U_{\alpha}$; it suffices to check that $\gamma$ has a neighbourhood $U$ such that $\sigma_{\alpha}|_{U}$ is smooth.
    Since $\{V_\beta\}_{\beta}$ is a cover of $G$, there exists $\beta\in\mathfrak{B}$ such that $\gamma\in U_{\alpha}\cap V_\beta$. Let $U \coloneqq U_{\alpha} \cap V_\beta$. By definition of $F$, we have $\gamma \in U \subseteq U_{\alpha}\cap U_{F(\beta)}$. By \PropB\ applied to $U_{\alpha}, U_{F(\beta)}$ and by definition of $\kappa_{\beta}$, the map
    $U\to \mbb{T} \times G\z$ given by $ \eta\mapsto \iota\inv (\sigma_{\alpha}(\eta)\kappa_{\beta}(\eta)\inv )$ is smooth. Since $\iota$ is a diffeomorphism onto an open subset of $\twist$, it follows that the map $U\to \twist$ given by $ \eta\mapsto \sigma_{\alpha}(\eta)\kappa_{\beta}(\eta)\inv$ is smooth. Now, $\kappa_{\beta}$ and  multiplication on $\twist$ are smooth by assumption. Hence the map
    $\sigma_{\alpha}|_{U}\colon U\to \twist$ given by $ \eta\mapsto (\sigma_{\alpha}(\eta)\kappa_{\beta}(\eta)\inv)\kappa_{\beta}(\eta)$
    is also smooth. Since $\gamma$ was arbitrary, this proves that $\sigma_{\alpha}$ is smooth.
\end{proof}

\begin{lemma}\label{lem:phase change}
Suppose that $ \sestriple $ is a topological twist over a Lie groupoid $G$,
that $\twist$ is a manifold, and that the $\mbb{T}$-action on $\twist$ is
smooth. Suppose that $\{\sigma_{\alpha} \colon U_{\alpha} \to \pi\inv
(U_{\alpha})\}_{\alpha\in\mathfrak{A}}$ is a family of continuous sections
and that $\mathfrak{B}$ is a family of triples
\[
    \mathfrak{B} \subseteq \big\{(\alpha, U, \omega) : \alpha \in \mathfrak{A}, U\subseteq U_{\alpha}\text{ is open, and } \omega \in C^{\infty}(U, \mbb{T})\big\}
\]
such that $(\alpha, U_{\alpha}, 1_{U_{\alpha}}) \in \mathfrak{B}$ for all
$\alpha \in \mathfrak{A}$. Define $F \colon \mathfrak{B} \to \mathfrak{A}$ by
$F(\alpha,U,\omega) = \alpha$.  For each $\beta = (\alpha, U, \omega) \in
\mathfrak{B}$, let $V_\beta \coloneqq U$, let $\omega_\beta \coloneqq
\omega$, and define $\kappa_\beta \colon V_\beta \to \twist$ by
$\kappa_\beta(\gamma) = \omega(\gamma) \cdot \sigma_{\alpha}(\gamma)$.
    \begin{enumerate}[label=\textup{(\arabic*)}]
        \item\label{item:phase changes inherit PropBCD} If
            $\{(U_{\alpha},\sigma_{\alpha})\}_{\alpha}$ satisfies \PropE\
            (respectively \PropB,  \PropC, \PropD), then so does
            $\{(V_\beta,\kappa_{\beta})\}_{\beta}$.
        \item\label{item:phase changes inherit smoothness} For $\alpha \in
            \mathfrak{A}$, the section $\sigma_{\alpha}$ is smooth if and
            only if the sections $\{\kappa_\beta : F(\beta) = \alpha\}$ are
            all smooth.
    \end{enumerate}
\end{lemma}
\begin{proof}
\ref{item:phase changes inherit PropBCD} For \PropE, fix $\beta = (\alpha, U,
\omega)\in \mathfrak{B}$. Since $\{(U_{\alpha},\sigma_{\alpha})\}_{\alpha}$
satisfies \PropE, there exists a smooth function $k_{\alpha} \colon
U_{\alpha} \cap G\z \to \mbb{T}$ such that $\sigma_{\alpha}(x) =
\iota(k_{\alpha}(x), x)$ for each $x \in U_{\alpha} \cap G\z$. Since
$k_{\alpha}$ is smooth, so is its restriction $l_\beta$ to the open set
$V_\beta \cap G\z$. Since $\omega_\beta$ is smooth, so is its restriction to
$V_\beta \cap G\z$. Thus the pointwise prodct $\omega_\beta l_\beta$ is
smooth. We have $\kappa_\beta(x) = \iota(\omega_\beta(x) l_\beta(x), x)$ by
definition of $\kappa_\beta$.

For \PropB, fix $\beta,\beta' \in \mathfrak{B}$, and let $\alpha = F(\beta)$
and $\alpha' = F(\beta')$. For $\gamma \in V_\beta \cap V_{\beta'} \subseteq
U_{F(\beta)} \cap U_{F(\beta')}$, we have
\[
\iota\inv(\kappa_\beta(\gamma)\kappa_{\beta'}(\gamma)\inv )
    = \iota\inv\Bigl(
        \omega_\beta(\gamma)\cdot\sigma_{\alpha}(\gamma)\,
        \bigl(\omega_{\beta'}\cdot\sigma_{\alpha'}(\gamma)\bigr)\inv
    \Bigr)
    = (\omega_\beta(\gamma)\overline{\omega_{\beta'}(\gamma)}, r(\gamma))\,
    \iota\inv \big(\sigma_{\alpha}(\gamma)\sigma_{\alpha'}(\gamma)\inv \big)
\]
because the $\mbb{T}$-action is central and $\iota$ intertwines the action of
$\mbb{T}$ on $\mbb{T} \times G\z$ by multiplication in the second coordinate
with the $\mbb{T}$-action on $\twist$. Since $\omega_\beta$ and
$\omega_{\beta'}$ are smooth, so is their product, and since
$\{(U_{\alpha},\sigma_{\alpha})\}_{\alpha}$ satisfies \PropB, the map $\gamma
\mapsto \sigma_{\alpha}(\gamma)\sigma_{\alpha'}(\gamma)\inv$ is smooth on the
open subset $V_{\beta} \cap V_{\beta'}$ of $U_{\alpha} \cap U_{\alpha'}$.
Since multiplication in $\mbb{T} \times G\z$ is smooth, it follows that
$\gamma \mapsto \iota\inv(\kappa_\beta(\gamma)\kappa_{\beta'}(\gamma)\inv )$
is smooth.

For \PropC, fix $\beta_i \in \mathfrak{B}$ such that
$V_{\beta_{1}}V_{\beta_{2}} \cap V_{\beta_3} \not= \emptyset$, and let
$\alpha_i \coloneqq F(\beta_i)$ for each $i$. Fix $\gamma_i \in V_{\beta_i}$
such that $\gamma_{1}\gamma_{2} = \gamma_3$ as above. Since the
$\mbb{T}$-action is central, and since $\iota$ intertwines the action of
$\mbb{T}$ on $G\z \times \mbb{T}$ (by multiplication in the second
coordinate) with the $\mbb{T}$ action on $\twist$, we have
\[
    \iota\inv \big(\kappa_{\beta_{1}}(\gamma_{1})\kappa_{\beta_{2}}(\gamma_{2})\kappa_{\beta_3} (\gamma_3)\inv \big)
    = \big(\omega_{\beta_{1}}(\gamma_{1})\omega_{\beta_{2}}(\gamma_{2})\overline{\omega_{\beta_3} (\gamma_{1}\gamma_{2})}, s(\gamma_{2})\big)\,
        \iota\inv \big(\sigma_{\alpha_{1}}(\gamma_{1})\sigma_{\alpha_{2}}(\gamma_{2})\sigma_{\alpha_3}(\gamma_3)\inv \big).
\]
Since multiplication in $G$ is smooth and $\omega_{\beta_3}$ is smooth, the
map $(\gamma_{1}, \gamma_{2}) \mapsto
\overline{\omega_{\beta_3}(\gamma_{1}\gamma_{2})}$ is smooth on the open
subset $V_{\beta_{1}} \fpsr V_{\beta_{2}}$ of $U_{\alpha_{1}} \fpsr
U_{\alpha_{2}}$. Since $\omega_{\beta_{1}}$ and $\omega_{\beta_{2}}$ are also
smooth, it follows that $(\gamma_{1}, \gamma_{2})
\mapsto\omega_{\beta_{1}}(\gamma_{1})\omega_{\beta_{2}}(\gamma_{2})
\overline{\omega_{\beta_3}(\gamma_{1}\gamma_{2})}$ is smooth. Similarly, that
$\{(U_{\alpha},\sigma_{\alpha})\}_{\alpha}$ satisfies \PropC\ ensures that
the map $(\gamma_{1},\gamma_{2}) \mapsto
    \iota\inv \big(
        \sigma_{\alpha_{1}}(\gamma_{1})
        \sigma_{\alpha_{2}}(\gamma_{2})
        \sigma_{\alpha_3}(\gamma_3)\inv
\big)$ is smooth on the open subset of $V_{\beta_{1}} \fpsr V_{\beta_{2}}$
where it is defined, and so \PropC\ again follows from smoothness of
multiplication in $\mbb{T} \times G\z$.

For \PropD, fix $\beta,\beta' \in \mathfrak{B}$ such that $V_\beta\inv \cap
V_{\beta'} \not= \emptyset$, and put $\alpha = F(\beta)$ and $\alpha' =
F(\beta')$. Then $V_{\beta} \cap V_{\beta'}\inv  \subseteq U_{\alpha} \cap
U_{\alpha'}\inv$ so that the latter is nonempty. For $\gamma \in V_{\beta}
\cap V_{\beta'}\inv$,
\[
    \iota\inv(\kappa_{\beta}(\gamma)\kappa_{\beta'}(\gamma\inv ))
        = \big(\omega_{\beta}(\gamma)\omega_{\beta'}(\gamma\inv )\big)\,
            \iota\inv \big(\sigma_{\alpha}(\gamma)\sigma_{\alpha'}(\gamma\inv ), s(\gamma)\big).
\]
Since $\omega_{\beta}$ and $\omega_{\beta'}$ are smooth and since inversion
is smooth in $G$, the map  $\gamma \mapsto
\omega_{\beta}(\gamma)\omega_{\beta'}(\gamma\inv )$ is smooth. Since
$\{(U_{\alpha},\sigma_{\alpha})\}_{\alpha}$ satisfies \PropD, the map $\gamma
\mapsto
    \iota\inv(\sigma_{\alpha}(\gamma)\sigma_{\alpha'}(\gamma\inv ))$
is smooth. So \PropD\ once again follows from smoothness of multiplication in
$\mbb{T} \times G\z$.

\ref{item:phase changes inherit smoothness} Since $\mathfrak{B}$ contains
each triple $(\alpha, U_{\alpha}, 1_{U_{\alpha}})$, the ``if'' implication is
trivial. For the ``only if'' implication, suppose that $\sigma_{\alpha}$ is
smooth, and fix $\beta \in \mathfrak{B}$ with $F(\beta) = \alpha$. Then by
definition, we have $V_\beta \subseteq U_{\alpha}$ and on $V_\beta$ we have
$\kappa_\beta(\gamma) = \omega_\beta(\gamma) \cdot \sigma_{\alpha}(\gamma)$.
Since $\omega_\beta$ and $\sigma_{\alpha}$ are smooth and the
$\mbb{T}$-action on $\twist$ is smooth, it follows that $\kappa_\beta$ is
smooth.
\end{proof}

\begin{thm}\label{thm:E is mfd}
Suppose that $ \ses $ is a  topological twist over a Lie groupoid $G$, that
$\{(U_{\alpha},\varphi_{\alpha})\}_{\alpha\in\mathfrak{A}}$ is an atlas of
$G$, and that $\{\sigma_{\alpha} \colon U_{\alpha} \to \pi\inv
(U_{\alpha})\}_{\alpha\in\mathfrak{A}}$ is a family of continuous sections.
If $\{\sigma_{\alpha}\}_{\alpha}$ satisfies~\PropB, then
\begin{enumerate}[label=\textup{(\arabic*)}]
    \item\label{item:E is mfd:principal bdl} there is a unique smooth
        structure on the topological space $\twist$ with respect to which
$\PB{\twist}{\pi}{G}$ is a smooth principal $\mbb{T}$-bundle and all the
$\sigma_{\alpha}$ are smooth;
    \item\label{item:E is mfd:pi inv Gz is submfd} $\pi\inv(G\z)$ is an
        embedded submanifold of $\twist$; and
    \item\label{item:E is mfd:smoothness of bundle maps} if $\PB{S}{p}{G}$
        is another smooth principal $\mbb{T}$-bundle  and if $f\colon
        \twist\to S$ is a continuous bundle map that is
        $\mbb{T}$-equivariant, then $f$ is smooth if only if
        $f\circ\sigma_{\alpha}\colon U_{\alpha}\to S$ is smooth for each
        $\alpha$.
\end{enumerate}
\end{thm}
\begin{proof}
By Lemma~\ref{lem:sections => triv's}, each $\sigma_{\alpha}$ defines a
topological trivialisation, namely the homeomorphism $\mbb{T} \times
U_{\alpha}\to \pi\inv (U_{\alpha})$ given by $(z, \gamma) \mapsto z\cdot
\sigma_{\alpha}(\gamma)=\iota(z, r(\gamma))\sigma_{\alpha}(\gamma)$. We want
to define an atlas for~$\twist$ using these local trivialisation maps. To be
precise, we will define charts that map onto (open subsets of)
$\mbb{R}^n\times\mbb{T}$, and we will leave it to the reader to modify these
charts to map onto (open subsets of) $\mbb{R}^{n+1}$.

For each $U_{\alpha}$,  define $W_{\alpha}  \coloneqq \pi\inv (U_{\alpha})$
and $\Phi_{\alpha} \colon W_{\alpha} \to \mbb{R}^n \times \mbb{T}$ by $z\cdot
\sigma_{\alpha}(\gamma) \mapsto (z, \varphi_{\alpha}(\gamma))$. To see that
the collection $\{(W_{\alpha}, \Phi_{\alpha})\}_{\alpha\in \mathfrak{A}}$ is
an atlas for a smooth structure on $\twist$ in the sense of
Definition~\ref{dfn:atlas}, note first that the $W_{\alpha}$ cover $\twist$
since the $U_{\alpha}$ cover $G$. Moreover, each $\Phi_{\alpha}$ is a
homeomorphism onto its image by construction, so we already have
Conditions~\ref{item:atlas:cover} and~\ref{item:atlas:homeo} of
Definition~\ref{dfn:atlas}. It remains to check~\ref{item:atlas:smooth}. So
suppose that $\alpha' \in \mathfrak{A}$ satisfies $U_{\alpha}\cap
U_{\alpha'}\neq \emptyset$. We must show that
\[
    \Phi_{\alpha'}\circ \Phi_{\alpha}\inv\colon\quad
    \Phi_{\alpha}(U_{\alpha}\cap U_{\alpha'}) \to  \Phi_{\alpha'}(U_{\alpha}\cap U_{\alpha'})
\]
is a diffeomorphism between open subsets of $\mbb{R}^n\times \mbb{T}$. By
\PropB, there exists a smooth map $f\coloneqq f_{\alpha,\alpha'}\colon
U_{\alpha}\cap U_{\alpha'}\to \mbb{T}$ such that the map
\[
    F\colon U_{\alpha}\cap U_{\alpha'}\to \mbb{T} \times G\z,
    \quad \gamma  \mapsto \iota\inv \bigl(\sigma_{\alpha}( \gamma  )
            \sigma_{\alpha'}( \gamma  )\inv\bigr),
\]
can be written as
\[
    F(\gamma )=(f( \gamma  ), r( \gamma  )).
\]
For $(z, \varphi_{\alpha} ( \gamma  )) \in \Phi_{\alpha}(U_{\alpha}\cap
U_{\alpha'})$, we compute
\begin{align*}
    \Phi_{\alpha}\inv
    (z, \varphi_{\alpha} ( \gamma  ))
    &=
    \iota(z, r( \gamma  ))\sigma_{\alpha}( \gamma  )
    =
    \iota(z, r( \gamma  ))
        \sigma_{\alpha}( \gamma  )
            \sigma_{\alpha'}( \gamma  )\inv\sigma_{\alpha'}( \gamma  )
    \\
    &=
   \iota(z, r( \gamma  ))
        \iota(F( \gamma  ))\sigma_{\alpha'}( \gamma  )
    =
    \iota(zf(\gamma ), r( \gamma  ))
    \sigma_{\alpha'}( \gamma  ).
    \intertext{Hence}
    (\Phi_{\alpha'}\circ\Phi_{\alpha}\inv)
    (z, \varphi_{\alpha} ( \gamma  ))
    &=
    (zf( \gamma  ), \varphi_{\alpha'}( \gamma  )).
\end{align*}
Writing $\varphi_{\alpha} ( \gamma  )=\vecrm{x}$, we obtain
\begin{align*}
    \bigl(\Phi_{\alpha'}\circ \Phi_{\alpha}\inv\bigr)(z, \vecrm{x})
    =
    \big(z\big(f \circ\varphi_{\alpha}\inv\big)(\vecrm{x}), \big(\varphi_{\alpha'}\circ\varphi_{\alpha}\inv\big)(\vecrm{x})\big).
\end{align*}
Since $\varphi_{\alpha'}\circ\varphi_{\alpha}\inv$,
$f\circ\varphi_{\alpha}\inv$, and multiplication on $\mbb{T}$ are smooth, we
conclude that $\Phi_{\alpha'}\circ \Phi_{\alpha}\inv$ is smooth. It is now
clear that the sections we started with are all smooth: The map
$\Phi_{\alpha}\circ\sigma_{\alpha}\circ \varphi_{\alpha}\inv\colon
\mbb{R}^n\to \mbb{T} \times \mbb{R}^{n}$ is given by $\vecrm{x}\to (1,
\vecrm{x})$ and hence smooth, which means that $\sigma_{\alpha}\colon
U_{\alpha}\to \twist$ is smooth.

It will be helpful to know what other smooth sections this manifold structure
on $\twist$ allows. By Lemma~\ref{lem:phase change}, Part~\ref{item:phase
changes inherit PropBCD}, the collection
\begin{equation}\label{eq:Theta}
     \Theta\coloneqq \{\omega\cdot (\sigma_{\alpha}|_U) : \alpha\in\mathfrak{A}, U\subseteq U_{\alpha} \text{ open, } \omega\in C^{\infty} (U, \mbb{T})\}
\end{equation}
also satisfies \PropB. Applying the procedure of the first two paragraphs of
the proof to $\Theta$, we obtain another atlas for a smooth structure on
$\twist$ with respect to which  each element of $\Theta$ is smooth, including
each $\sigma_{\alpha}$. By construction, this new atlas {\em contains all the
charts $\Phi_{\alpha}$} from the previously constructed atlas, and therefore
generates the same smooth structure as the one that we constructed from only
$\{\sigma_{\alpha}\}_{\alpha}$. To sum up, each element of $\Theta$ is smooth
with respect to the smooth structure with which we have equipped $\twist$.

To see that $\PB{\twist}{\pi}{G}$ is a smooth principal $\mbb{T}$-bundle, we
will apply Lemma~\ref{lem:smooth principal bundle}. The $\mbb{T}$-action
$\mu\colon \mbb{T}\times\twist\to\twist$ is smooth, since, in charts, it is
just given by multiplication in $\mbb{T}$. To be more precise, given any
$(w,e)\in \mbb{T}\times \twist$, choose any $\alpha\in \mathfrak{A}$ with
$\pi(e)\in U_{\alpha}$. Then
\[
    \Phi_{\alpha} \circ \mu \circ (\mathrm{id}_{\mbb{T}}\times \Phi_{\alpha})\inv \colon\; \mbb{T}\times (\mbb{T} \times \mbb{R}^n)
    \to \mbb{T} \times \mbb{R}^n,
    \quad
    (w,(z, \vecrm{x}))
    \mapsto
    (wz, \vecrm{x}),
\]
is a smooth map, which proves that $\mu$ is smooth around $(w,e)$. Next,
$\pi$ is smooth, since
\[
    \varphi_{\alpha} \circ \pi \circ \Phi_{\alpha}\inv\colon\;
    \mbb{T} \times \mbb{R}^n
    \to \mbb{R}^n,
    \quad
    (z, \vecrm{x})\mapsto \vecrm{x},
\]
is clearly smooth. Using these local charts, we also see that $\pi$ is a
submersion, since $\varphi_{\alpha} \circ \pi \circ \Phi_{\alpha}\inv$, being
a linear map, is its own derivative and clearly surjective. We have checked
that all the assumptions of Lemma~\ref{lem:smooth principal bundle} hold, and
we conclude that $\PB{\twist}{\pi}{G}$ is indeed a smooth principal
$\mbb{T}$-bundle. This proves \ref{item:E is mfd:principal bdl} save
uniqueness, which we will prove at the end.

Part~\ref{item:E is mfd:pi inv Gz is submfd} now follows from
Remark~\ref{rmk:submersion implies transverse} and
Theorem~\ref{thm:Generalized Preimage Thm}\ref{item:Generalized Preimage
Thm}, since $G\z$ is an embedded submanifold of the Lie groupoid $G$ and
since $\pi$ is a submersion.

We  can now invoke Lemma~\ref{lem:sections => triv's} to make our lives
easier. For {\em any} section $\sigma\colon U\to \pi\inv(U)$ of $\pi$ that is
smooth with respect to the smooth structure that we just constructed on
$\twist$, the map
\begin{equation}\label{eq:psi sigma}
    \psi_\sigma\colon \pi\inv(U) \to \mbb{T}\times U,\quad z\cdot \sigma(\gamma ) \mapsto (z,\gamma ),
\end{equation}
is a smooth local trivialisation of $\twist$. Thus, if $(U,\varphi)$ is a
smooth chart of $G$, then $\psi_\sigma$ determines a smooth chart
\begin{equation}\label{eq:Phi sigma varphi}
    \Phi_{\sigma,\varphi}\colon \pi\inv(U) \to \mbb{T}\times \mbb{R}^n,\quad z\cdot \sigma(\gamma ) \mapsto (z,\varphi(\gamma )),
\end{equation}
of $\twist$. Note that \(
    \Phi_{\alpha}=\Phi_{\sigma_{\alpha},\varphi_{\alpha}}.
\)

For \ref{item:E is mfd:smoothness of bundle maps}, suppose that
$\PB{S}{p}{G}$ is a smooth principal $\mbb{T}$-bundle and $f\colon \twist\to
S$ is  a continuous $\mbb{T}$-equivariant bundle map. For the forward
implication, assume that $f$ is smooth and fix $\alpha\in\mathfrak{A}$ and
$\gamma \in U_{\alpha}$. The set $\pi\inv(\{\gamma \}) \cap \Phi_{\alpha}\inv
(\{1\} \times \mbb{R}^n)$ contains a unique element, namely $e\coloneqq
\sigma_{\alpha} (\gamma )$. Smoothness and $\mbb{T}$-equivariance of $f$
imply that, for every chart $(V,\Phi^S)$ around $f(e)$ in $S$, the map
\[
   \Phi^S\circ f\circ \Phi_{\alpha}\inv\colon \Phi_{\alpha} (W_{\alpha} \cap f\inv (V)) \to \Phi^S(V), \quad
   (z, \varphi_{\alpha} (\eta)) \mapsto
        \Phi^S(f(z\cdot \sigma_{\alpha} (\eta))) = \Phi^S(z\cdot f(\sigma_{\alpha} (\eta))),
\]
is smooth as a map between neighbourhoods in Euclidean space. In other words,
\[
\Phi^S\circ f\circ \Phi_{\alpha}\inv\colon \quad
    (z, \vecrm{x}) \mapsto \Phi^S(z\cdot f(\sigma_{\alpha} (\varphi_{\alpha}\inv (\vecrm{x}))))
\]
is smooth. Since $\mbb{R}^n\to \mbb{T} \times \mbb{R}^n, \vecrm{x}\mapsto (1,
\vecrm{x})$, is smooth, it follows that the map
\[
    \vecrm{x} \mapsto
        (\Phi^S \circ [f \circ \sigma_{\alpha}] \circ \varphi_{\alpha}\inv) (\vecrm{x})
\]
is smooth, defined on the open subset
$U'\coloneqq\mathrm{pr}_{\mbb{R}^n}(\Phi_{\alpha} (W_{\alpha} \cap f\inv
(V)))$ of $\mbb{R}^n$. Note that the open set $U\coloneqq
\varphi_{\alpha}\inv ( U')$ in $G$ contains $\gamma $: since $f(e)\in V$ by
choice of $V$ and since $\pi(e)=\gamma \in U_{\alpha}$, we have $e\in
f\inv(V) \cap \pi\inv(U_{\alpha}) = f\inv(V) \cap W_{\alpha}$. Since
$e=\sigma_{\alpha} (\gamma )$, the definition of $\Phi_{\alpha}$ shows that
$\Phi_{\alpha} (e) =(1, \varphi_{\alpha}(\gamma ))$ and thus
$\varphi_{\alpha}(\gamma )\in U'$; that is, $\gamma \in U$ as claimed.
Moreover, since $\Phi^S$ was chosen as a map around $f(e)=(f\circ
\sigma_{\alpha})(\gamma )$, the proof that $\Phi^S \circ [f \circ
\sigma_{\alpha}] \circ \varphi_{\alpha}\inv$ is smooth implies that $f\circ
\sigma_{\alpha}$ is smooth around $\gamma $. As $\gamma $ was arbitrary, it
follows that $f \circ \sigma_{\alpha}$ is smooth on all of $U_{\alpha}$.

For the backwards implication of  \ref{item:E is mfd:smoothness of bundle
maps}, assume that $f\circ \sigma_{\alpha}$ is smooth for every
$\alpha\in\mathfrak{A}$. In order to prove that $f$ is smooth, fix $e\in
\twist$ and $\alpha\in\mathfrak{A}$ such that $\gamma \coloneqq \pi(e)\in
U_{\alpha}$. Then there exists a unique $z_{0}\in \mbb{T}$ such that
$e=z_{0}\cdot \sigma_{\alpha}(\gamma )$. Since the $\mbb{T}$-action on $S$
and the function $f\circ \sigma_{\alpha}$ are smooth, the map $\tau\colon
\mbb{T}\times U_{\alpha}\to S$, $(z,\eta)\mapsto z\cdot f(\sigma_{\alpha}
(\eta))$, is smooth. In particular, for any neighbourhood $U$ of $\gamma $ in
$U_{\alpha}$ and any smooth chart $(V,\Phi^S)$ around $e=z\cdot
f(\sigma_{\alpha}(\gamma ))$ in $S$ with $z\cdot f(\sigma_{\alpha}
(U))\subseteq V$, the map
\[
    \Phi^S\circ \tau \circ(\mathrm{id}_{\mbb{T}}\times \varphi_{\alpha})\inv\colon \quad
        \mbb{T}\times \varphi_{\alpha}(U) \to \Phi^S(V)
\]
is smooth. Now, $\mbb{T}$-equivariance of $f$ and the definition of
$\Phi_{\alpha}$ show that for any $(z,\eta)\in \mbb{T}\times U_{\alpha}$,
\[
    \tau(z,\eta)
        = f(z\cdot \sigma_{\alpha}  (\eta))
        = (f\circ \Phi_{\alpha}\inv) (z, \varphi_{\alpha} (\eta)).
\]
Hence for any $(z,\vecrm{x})\in \mbb{T}\times \varphi_{\alpha} (U)$, we have
\[
    \bigl[\Phi^S\circ \tau \circ(\mathrm{id}_{\mbb{T}}\times \varphi_{\alpha})\inv\bigr] (z,\vecrm{x})
        = (\Phi^S\circ f\circ \Phi_{\alpha}\inv) (z,\vecrm{x}).
\]
We have shown that the map
\[
    \Phi^S\circ f\circ \Phi_{\alpha}\inv\colon \quad
        \mbb{T} \times \varphi_{\alpha} (U) \to \Phi^S(V)
\]
is smooth. As $\gamma \in U$, we have $(z_{0}, \varphi_{\alpha}(\gamma )) \in
\mbb{T} \times \varphi_{\alpha} (U)$. Thus
\[
    e= z_{0}\cdot \sigma_{\alpha} (\gamma ) \in \Phi_{\alpha}(\mbb{T} \times \varphi_{\alpha} (U))
\]
and $f$ is smooth around $e\in \twist$. Since $e$ was arbitrary, this proves
that $f$ is smooth.

We can now prove that the smooth structure in Part~\ref{item:E is
mfd:principal bdl} is unique. Let $\tilde{\twist}$ be a copy of $\twist$
carrying a smooth structure such that
\begin{enumerate}[label={($\tilde{\twist}$\arabic*)}]
    \item\label{item:mcA:principal bdl} $\PB{\tilde{\twist}}{\pi}{G}$ is a
        smooth principal $\mbb{T}$-bundle, and
    \item\label{item:mcA:Theta smooth} for each $\alpha\in\mathfrak{A}$,
        the map $\tilde{\sigma}_{\alpha}\colon U_{\alpha}\to
        \tilde{\twist}$, $\gamma \mapsto \sigma_{\alpha}(\gamma )$, is
        smooth.
\end{enumerate}
We show that $f \colon \twist\to \tilde{\twist}$, defined as the identity map
on the underlying topological space, is a diffeomorphism, so the smooth
structure on $\tilde\twist$ coincides with the structure we gave $\twist$ by
hand. Clearly, $f$ is a continuous, $\mbb{T}$-equivariant bundle map. By
assumption \ref{item:mcA:Theta smooth},  each $\tilde{\sigma}_{\alpha} = f
\circ \sigma_{\alpha}$ is smooth. Because of Theorem~\ref{thm:E is
mfd}\ref{item:E is mfd:smoothness of bundle maps} combined with
Assumption~\ref{item:mcA:principal bdl}, this implies that $f$ is smooth. An
application of Lemma~\ref{lem:principal bundles are rigid} now shows that $f$
is a diffeomorphism.
\end{proof}

\begin{corollary}\label{p:PropCD}
Suppose that $ \ses $ is a Lie twist. Suppose that
$\{(U_{\alpha},\varphi_{\alpha})\}_{\alpha\in\mathfrak{A}}$ is an atlas of
$G$ and that  $\{\sigma_{\alpha} \colon U_{\alpha} \to \pi\inv
(U_{\alpha})\}_{\alpha\in\mathfrak{A}}$ is a family of {\em continuous}
sections satisfying~\PropB. Let $\tilde{\twist}$ be the topological space
$\twist$ equipped with the smooth structure induced by
$\{\sigma_{\alpha}\}_{\alpha}$ as in Theorem~\ref{thm:E is mfd}. Then
$\twist=\tilde{\twist}$ if and only if the all maps $\sigma_{\alpha}\colon
U_{\alpha}\to \twist$ are smooth.
\end{corollary}
\begin{proof}
Since $\twist$ is a Lie twist, $\PB{\twist}{\pi}{G}$ is a smooth principal
$\mbb{T}$-bundle by Lemma~\ref{lem:free stuff about Lie twists},
\ref{item:LT:principal bundle}. With respect to the same $\mbb{T}$-action and
the same projection map, $\PB{\tilde{\twist}}{\pi}{G}$ is also a smooth
principal $\mbb{T}$-bundle. In particular, $f \colon \tilde{\twist}\to
\twist$, defined as the identity map on the underlying topological space, is
a $\mbb{T}$-equivariant bundle map. Since $\twist$ and $\tilde{\twist}$ have
the same topology by assumption, $f$ is a homeomorphism. Thus, by
Lemma~\ref{lem:principal bundles are rigid}, we have that
$\twist=\tilde{\twist}$ if and only if $f$ is smooth. By Theorem~\ref{thm:E
is mfd}, Part~\ref{item:E is mfd:smoothness of bundle maps}, this is
equivalent to $f\circ \sigma_{\alpha}\colon U_{\alpha}\to \twist$ being
smooth for all $\alpha$, as claimed.
\end{proof}

While we do not have an immediate use for the following corollary, it seems
worth recording a checkable condition under which two families of sections as
in Theorem~\ref{thm:E is mfd} determine the same smooth structure.

\begin{corollary}
Suppose that $ \ses $ is a  topological twist over an \etale\ Lie groupoid.
Suppose that $\{U_{\alpha}\}_{\alpha \in \mathfrak{A}}$ and
$\{V_\beta\}_{\beta \in \mathfrak{B}}$ are bases of open bisections for $G$
and that  $\{\sigma_{\alpha} \colon U_{\alpha} \to \pi\inv
(U_{\alpha})\}_{\alpha \in \mathfrak{A}}$ and $\{\tau_\beta \colon V_\beta
\to \pi\inv (V_\beta)\}_{\beta \in \mathfrak{B}}$ are families of continuous
sections that both satisfy~\PropB. Suppose that there is a common refinement
$\mathcal{R} = \{W_\delta \colon \delta \in \mathfrak{D}\}$ of
$\{U_{\alpha}\}_{\alpha}$ and $\{V_\beta\}_\beta$ (with refinement maps
$\mathfrak{a} \colon \mathfrak{D} \to \mathfrak{A}$ and $\mathfrak{b} \colon
\mathfrak{D} \to \mathfrak{B}$) that is a cover of $G$, such that for each
$\delta \in \mathfrak{D}$ the  map $\iota\inv
(\sigma_{\mathfrak{a}(\delta)}(\gamma)\tau_{\mathfrak{b}(\delta)}(\gamma)\inv)$
is smooth on~$W_\delta$. Then the unique smooth structures on $\twist$,
obtained from Theorem~\ref{thm:E is mfd}, for which the $\sigma_{\alpha}$ and
the $\tau_\beta$ are smooth coincide.
\end{corollary}
\begin{proof}
By Lemma~\ref{lem:refinement}, the system $\{(W_\delta,
\sigma_{\mathfrak{a}(\delta)} )\}_{\delta}$ yields the same smooth structure
as $\{(U_{\alpha}, \sigma_{\alpha})\}_{\alpha}$ and similarly $\{(W_\delta,
\tau_{\mathfrak{b}(\delta)} )\}_{\delta}$ yields the same smooth structure as
$\{(V_\beta, \tau_\beta)\}_{\beta}$. For each $\delta \in \mathcal{D}$, let
$\omega_\delta \colon W_\delta \to \mbb{T}$ be the function such that
$\iota\inv
(\sigma_{\mathfrak{a}(\delta)}(\gamma)\tau_{\mathfrak{b}(\delta)}(\gamma)\inv
) = \omega_\delta(\gamma)$ for all $\gamma \in W_\delta$. So $\omega_\delta$
is smooth by assumption. Now Lemma~\ref{lem:phase change} applied to the
system $\{(W_{\delta}, \tau_{\mathfrak{b}(\delta)})\}_{\delta}$ and with the
indexing set
\[
    \{(\delta, W_\delta, 1_{W_\delta}), (\delta, W_\delta, \omega_\delta) : \delta \in \mathfrak{D}\}
\]
shows that the system system $\{(W_{\delta}, \eta_{\delta,i}) : (\delta,i)
\in \mathfrak{D} \times \{0,1\}\}$ given by $\eta_{\delta,0} =
\tau_{\mathfrak{b}(\delta)}$ and $\eta_{\delta, 1} =
\sigma_{\mathfrak{a}(\delta)}$ determines the same smooth structure as
$\{(W_\delta, \tau_{\mathfrak{b}(\delta)})\}_{\delta}$. Lemma~\ref{lem:phase
change}, this time applied to the system $\{(W_\delta,
\sigma_{\mathfrak{a}(\delta)})\}_{\delta}$ and with the indexing set
\[
    \{(\delta, W_\delta, 1_{W_\delta}), (\delta, W_\delta, \overline{\omega_\delta}) : \delta \in \mathfrak{D}\},
\]
shows that the same system $\{(W_{\delta},\eta_{\delta,i})\}_{(\delta,i)}$
gives the same smooth structure as $\{(W_\delta,
\sigma_{\mathfrak{a}(\delta)})\}_{\delta}$, and we are done.
\end{proof}

\begin{remark}
The previous result is not an ``if and only if'' statement: to conclude from
coincidence of the smooth structure on $\twist$ coming from the
$\sigma_{\alpha}$ and the $\tau_\beta$ that functions of the form $\iota\inv
(\sigma_{\mathfrak{a}(\delta)}(\gamma)\tau_{\mathfrak{b}(\delta)}(\gamma)\inv)$
are smooth requires that multiplication and inversion in $\twist$ are smooth
maps; as we will see in Theorem~\ref{thm:E is Lie}, this is the case
precisely if either of the families $\{\sigma_{\alpha}\}_{\alpha}$ or
$\{\tau_\beta\}_{\beta}$ satisfies \PropE~and~\PropC. The correct general
``if and only if'' statement is that two families $\sigma_{\alpha}$ and
$\tau_\beta$ give the same smooth structure if and only if there are smooth
functions $\omega_{\alpha,\beta} \colon U_{\alpha} \cap V_\beta \to \mbb{T}$
such that $\omega_{\alpha,\beta}(\gamma) \cdot \sigma_{\alpha}(\gamma) =
\tau_\beta(\gamma)$; but this an instance of a general statement about smooth
principle $\Gamma$-bundles (it could be deduced, for example, using
Lemma~\ref{lem:sections => triv's}) rather than about twists over Lie
groupoids.
\end{remark}

\begin{prop}\label{prp:E:PropB and PropE}
Suppose that $ \ses $ is a  topological twist over a  Lie groupoid $G$, that
$\{(U_{\alpha},\varphi_{\alpha})\}_{\alpha\in\mathfrak{A}}$ is an atlas of
$G$, and that $\{\sigma_{\alpha} \colon U_{\alpha} \to \pi\inv
(U_{\alpha})\}_{\alpha\in\mathfrak{A}}$ is a family of continuous sections
that satisfies  \PropB~and~\PropE. With respect to  the smooth structure on
$\twist$ from Theorem~\ref{thm:E is mfd}, the following hold.
\begin{enumerate}[label=\textup{(\arabic*)}]
    \item\label{item:E:constant 1} For each $x\in G\z$, there exists an
        open neighbourhood $U$ of $x$ in $G$ and a smooth section
        $\sigma\colon U\to\pi\inv(U)$ such that $\sigma(y)=\iota(1, y)$ for
        all $y\in U\cap G\z$.
    \item\label{item:E:Ez submfd} $\twist\z$ is an embedded submanifold of
        $\twist$.
    \item\label{item:E:iota diffeo} $\iota\colon
        G\z\times\mbb{T}\to\pi\inv(G\z)$ is a diffeomorphism.
    \item\label{item:E:r,s,pi submersions} The range and source maps
        $r,s\colon \twist\to \twist\z$ of $\twist$ and the projection map
        $\pi\colon \twist\to G$ are submersions.
\end{enumerate}
\end{prop}

\begin{proof}
\ref{item:E:constant 1} Fix $x_{0}\in G\z$ and $\alpha\in\mathfrak{A}$ such
that $x_{0}\in U_{\alpha}$, and let $k_{\alpha}\colon U_{\alpha}\cap G\z\to
\mathbb{T}$ be the map implicitly defined by
$\sigma_{\alpha}(x)=\iota(k_{\alpha}(x),x)$; by \PropE, this map is smooth.
Let $z = -k_{\alpha}(x_{0}) \in \mbb{T}$ (the point antipodal to
$k_{\alpha}(x_{0})$), and consider $W\coloneqq  k_{\alpha}\inv
(\mbb{T}\setminus \{z\})$. Since $W$ is open in $G\z$ and $G\z$ has the
subspace topology, there exists an open $U\subseteq G$ such that $W=U\cap
G\z$. Since $W\subseteq U_{\alpha}\cap G\z$, we can assume without loss of
generality that $U\subseteq U_{\alpha}$. Let $\psi\colon \mbb{T}\setminus
\{z\} \approx \mbb{R}$ be a smooth chart for this subset of~$\mbb{T}$. Since
$k_{\alpha}$ maps $W$ to $\mbb{T}\setminus \{z\}$, we may consider $\psi\circ
k_{\alpha}|_{W}\colon W\to \mbb{R}$.

Since $G\z$ is properly embedded in $G$ by \cite[Proposition~5.5]{Lee:Intro},
it follows that $W$ is properly embedded in $U$ (see Lemma~\ref{lem:proper
embedding when intersecting}). So \cite[Lemma~5.34(b)]{Lee:Intro} yields a
smooth function $\tilde{\omega}\colon U\to \mbb{R}$ such that
$\tilde{\omega}|_{W}\equiv \psi\circ k_{\alpha}|_{W}$; let $\omega\coloneqq
\psi\inv \circ \tilde{\omega}\colon U\to \mbb{T}\setminus\{z\}$. Since
$\overline{\omega}\colon \gamma \mapsto \overline{\omega(\gamma )}$ is an
element of $C^{\infty}(U,\mbb{T})$, we may consider the section
$\sigma\coloneqq \overline{\omega}\cdot (\sigma_{\alpha}|_{U})$, an element
of the set $\Theta$ as defined in Equation~\eqref{eq:Theta}. For $x\in U\cap
G\z= W$, it follows from $\tilde{\omega}|_{W}\equiv \psi\circ
k_{\alpha}|_{W}$ that
\begin{align}\label{eq:omega equiv k_alpha}
    \omega(x)
    &= (\psi\inv \circ \tilde{\omega})(x)
    = k_{\alpha}(x).
\end{align}
Hence
\begin{align*}
    \sigma(x)
    &= \overline{\omega(x)} \cdot \sigma_{\alpha}(x)
    = \overline{k_{\alpha}(x)}\cdot \iota(x,k_{\alpha}(x))
    = \iota(1, x)
\end{align*}
as claimed.

\ref{item:E:Ez submfd} Fix $y_{0}\in \twist\z$, and let $x_{0}\coloneqq
\pi(y_{0})\in G\z$. Let $m$ be the manifold dimension of $G\z$; note that,
since $G$ is not necessarily \etale, $m$ may be strictly smaller than the
manifold dimension $n$ of $G$. As $G\z$ is an embedded submanifold of $G$, we
can pick a chart  $(U,\varphi)$  of $G$ around $x_{0}$ such that
\begin{equation}\label{eq:varphi on Gz}
\varphi ( U\cap G\z)=\varphi(U)\cap (\mbb{R}^{m}\times\{0\}^{n-m}).
\end{equation}
By potentially shrinking $U$, we may assume by Part~\ref{item:E:constant 1}
that there exists a section $\sigma\colon U\to\pi\inv(U)$ such that
$\sigma(x)=\iota(1,x)$ for all $x\in U\cap G\z$. We claim that the smooth
chart $\Phi_{\sigma,\varphi}$ of $\twist$ maps $\pi\inv(U)\cap \twist\z $
onto $\{1\} \times \mbb{R}^{m}\times\{0\}^{n-m}$:

Given $y\in \pi\inv(U)\cap \twist\z$, we have $\pi(y)\in U\cap G\z$ and thus
$y=\iota(1,\pi(y))= 1\cdot \sigma(\pi(y))$ by choice of~$\sigma$. Thus, \(
    \Phi_{\sigma,\varphi} (y)
    = \bigl(
        1, \varphi (\pi(y))
    \bigr)
\) by definition of $\Phi_{\sigma,\varphi}$ (see \eqref{eq:Phi sigma varphi}
in the proof of Theorem~\ref{thm:E is mfd}). Then since $\pi(y)\in
\pi(\pi\inv(U)\cap\twist\z) = U\cap G\z$, it follows from the choice of
$\varphi$ (Equation~\eqref{eq:varphi on Gz}) that \(
    \Phi_{\sigma,\varphi} (y)
        \in \{1\} \times \mbb{R}^{m}\times\{0\}^{n-m}.
\)

Now suppose that $e\in \pi\inv(U)$ satisfies \(
    \Phi_{\sigma,\varphi}(e)
    = (1, \vecrm{x}, 0)
    \in \{1\} \times \mbb{R}^{m}\times\{0\}^{n-m}
\); we must argue that $e\in \twist\z$. As $\pi(e)\in U$, there exists a
unique $z\in\mbb{T}$ such that $e=z\cdot \sigma(\pi(e))$, so that
\[
(1, \vecrm{x},0)
    = \Phi_{\sigma,\varphi}(e)
    \overset{\eqref{eq:Phi sigma varphi}}{=} (z , \varphi(\pi(e))).
\]
By Equation~\eqref{eq:varphi on Gz}, $(\vecrm{x},0)=\varphi(\pi(e))$ implies
that $\pi(e)=x\in G\z$. The above computation further implies that $z=1$, so
that \(
    e
    = \sigma(x)
    = \iota(1, x)\in \twist\z,
\) as claimed. This finishes the proof that $\twist\z$ is an embedded
submanifold of $\twist$.

\ref{item:E:iota diffeo} We already know from the theory of topological
twists that $\iota$ is a homeomorphism. To see that $\iota$ is smooth, fix
$(z, x_{0}) \in \mbb{T} \times G\z$. By Part~\ref{item:E:constant 1}, there
is an open neighbourhood $U$ of $x_{0}$ in $G$ and a smooth section $\sigma
\colon U \to \pi\inv(U)$ such that $\sigma(x) = \iota(1, x)$ for all $x \in
U_{0}\coloneqq U \cap G\z$. It suffices to show that $\iota$ is smooth on the
open neighbourhood $\mbb{T} \times U_{0}$ of $(z,x_{0})$ in $\mbb{T} \times
G\z$. As explained earlier, the map $\psi_\sigma$ in \eqref{eq:psi sigma} in
the proof of Theorem~\ref{thm:E is mfd} is a diffeomorphism of $\pi\inv(U)$
onto $\mbb{T} \times U$. For $x \in U_{0}$, we have $\psi_{\sigma} \circ
\iota(z,x) = \psi_{\sigma}(z \cdot \iota(1,x)) = \psi_{\sigma}(z \cdot
\sigma(x)) = (z,x)$, so $\psi_{\sigma} \circ \iota$ is the identity on, and
in particular a diffeomorphism of, $\mbb{T} \times U_{0}$. Hence
$\iota|_{\mbb{T} \times U_{0}} = \psi_{\sigma}\inv \circ (\psi_{\sigma} \circ
\iota|_{\mbb{T} \times U_{0}})$ is smooth. Now, $\PB{\pi\inv(G\z)}{\pi}{G\z}$
is a principal $\mbb{T}$-bundle (see Remark~\ref{rmk:pi submersion, restr
bdl}) and, by construction of the $\mbb{T}$-action on $\twist$, $\iota$ is a
$\mbb{T}$-equivariant bundle map. It now follows from
Lemma~\ref{lem:principal bundles are rigid} that the  smooth homeomorphism
$\iota$ is a diffeomorphism onto $\pi\inv(G\z)$.

\ref{item:E:r,s,pi submersions} We already know that $\pi$ is a submersion
because $\PB{\twist}{\pi}{G}$ is a smooth principal bundle
(Theorem~\ref{thm:E is mfd}\ref{item:E is mfd:principal bdl}). We prove that
$r_{\twist}$ is a submersion; a similar argument shows that $s_{\twist}$ is a
submersion. To see that $r_{\twist}$ is smooth, note that it is given by
\[
    r_{\twist}(e) = \iota(1, r_{G}(\pi(e))).
\]
Since the maps $\pi\colon \twist\to G$, $r_{G}\colon G\to G\z$, $\iota\colon
\mbb{T} \times G\z \to \twist$, and $G\z\to \mbb{T} \times G\z, x\mapsto (1,
x)$, are all smooth, it follows that $r_{\twist}\colon \twist\to \twist$ is
smooth with image contained in $\twist\z$. Since we have shown in
Part~\ref{item:E:Ez submfd} that $\twist\z$ is an embedded submanifold of
$\twist$, it now follows from \cite[Corollary~5.30]{Lee:Intro} that
$r_{\twist}$ is also smooth as a map $\twist\to \twist\z$, as claimed.

To see that $r_{\twist}$ is a submersion, fix $e$ in $\twist$. We have
\begin{align*}
    dr_{\twist}|_{e} = d(r_{G}\circ \pi)|_{e} = dr_{G}|_{\pi(e)} \circ d\pi|_{e}.
\end{align*}
Hence if $\pi(e)\in U_{\alpha}$, then
\[
dr_{\twist}|_{e} \circ d(\sigma_{\alpha}|_{\pi(e)})
    = dr_{G}|_{\pi(e)} \circ d\pi|_{e} \circ d(\sigma_{\alpha}|_{\pi(e)})
    = d(r_{G}\circ \pi \circ \sigma_{\alpha})|_{\pi(e)}
    = dr_G|_{\pi(e)}.
\]
This is surjective because $G$ is a Lie groupoid, so we conclude that
$dr_{\twist}|_{e}$ is surjective as well.
\end{proof}

\begin{remark}\label{rmk:PropB-D when E is already Lie}
Condition~\PropE\ in Proposition~\ref{prp:E:PropB and PropE} is necessary. To
see why, observe that since the topological group $\mbb{T}$ admits only one
nontrivial continuous automorphism, namely $z \mapsto \overline{z}$, if $G\z$
is connected, then there are just two possible continuous injective
homomorphisms $\iota \colon \mbb{T} \times G\z \to \twist$ that satisfy $\pi
\circ \iota = \operatorname{id}_{G\z}$. Now suppose that $G = G\z$ is a
(connected) manifold viewed as a groupoid consisting entirely of units, and
that $\twist = \mbb{T} \times  G$ viewed as a (topologically) trivial group
bundle over $G$, with $\pi \colon \twist \to G$ the projection map. Then
$\iota \colon \mbb{T} \times G\z \to \twist$ is either $\iota(z, x) = (z, x)$
or $\iota(z, x) = (\overline{z}, x)$. Let $f \colon G \to \mbb{T}$ be a
continuous function that is not differentiable. Fix a maximal atlas
$\mathcal{A} = \{(U_{\alpha}, \varphi_{\alpha}) : \alpha \in \mathfrak{A}\}$
for $G$; so the $U_{\alpha}$ are a base for the topology on $G$. Then
$\mathcal{A}$ is a base of bisections of $G$. For each $\alpha \in
\mathfrak{A}$, define $\sigma_{\alpha} \colon U_{\alpha} \to \twist$ by
$\sigma_{\alpha}(x) \coloneqq (f(x), x)$. Then $\{\sigma_{\alpha}\}_{\alpha
\in \mathfrak{A}}$ It also satisfies \PropB, because for any $\alpha,\alpha'
\in \mathfrak{A}$ and any $x \in U_{\alpha} \cap U_{\alpha'}$, we have
$\sigma_{\alpha}(x)\sigma_{\alpha'}(x)\inv = (f(x), x)(f(x), x)\inv = (1, x)
= \iota(1, x)$. As in Theorem~\ref{thm:E is mfd}, there is a unique smooth
structure on $\twist$ that makes the maps $\sigma_{\alpha}$ smooth.
Specifically, we identify $\twist$ with $\mbb{T} \times G$ by the map $h
\colon \mbb{T} \times G \to \twist$ defined by $h(z, x) \coloneqq (zf(x),
x)$, and then charts $\psi_{\alpha} \colon U_{\alpha} \times \mbb{T} \to
\pi\inv(U_{\alpha})$ are given by $\psi_{\alpha}(z,x) = h(z,x) = (zf(x), x)$.
In particular, we have $\sigma_{\alpha}(x) = \iota(f(x), x)$ or
$\iota(\overline{f(x)}, x)$ (depending on the choice of $\iota$ above) for
each $\alpha \in \mathfrak{A}$ and $x \in U_{\alpha}$. That is, the map
$k_{\alpha}$ in \PropE\ is either $f$ or $\overline{f}$, so is not smooth.
Hence $\mathcal{A}$ does not satisfy \PropE. Multiplication and inversion in
$\twist$ are smooth with respect to the smooth structure induced by
$\mathcal{A}$, and each $\sigma_{\alpha} \colon U_{\alpha} \to \twist$ is by
definition smooth in this smooth structure. So since the composition of the
map $x \mapsto \iota(1, x)\inv \sigma_{\alpha}(x)$ with the projection from
$\mbb{T} \times G$ onto $\mbb{T}$ is either $f$ or $\overline{f}$, neither of
which is smooth, we deduce that $\iota|_{\{1\} \times G}$ is not smooth as a
map into $\twist$ with the smooth structure $\mathcal{A}$.

That is, in order for the smooth structure on $\twist$ induced by
$\mathcal{A}$ to be compatible through $\iota$ with the canonical
differential structure on $\mbb{T}\times G$, we {\em must} assume \PropE.

Note that this example also satisfies \PropC\ and \PropD\ because $x \mapsto
(x,x)$ and $(x,x) \mapsto x^2 = x$ are mutually inverse diffeomorphisms
between $G\comp$ and $G$, and $x \mapsto x\inv$ is the identity map on $G$.
So even given all of \PropB--\PropD, the condition~\PropE is a necessary
additional assumption.
\end{remark}

The following companion result to Theorem~\ref{thm:E is mfd} and
Proposition~\ref{prp:E:PropB and PropE} shows that Condition~\PropC~is the
additional ingredient required to ensure that the twist $\twist$ is a Lie
groupoid under the smooth structure it obtains from the theorem.

\begin{thm}\label{thm:E is Lie}
Suppose that $ \ses $ is a topological twist over a Lie groupoid $G$, that
$\{(U_{\alpha},\varphi_{\alpha})\}_{\alpha\in\mathfrak{A}}$ is an atlas of
$G$, and that  $\{\sigma_{\alpha} \colon U_{\alpha} \to \pi\inv
(U_{\alpha})\}_{\alpha\in\mathfrak{A}}$ is a family of continuous sections.
    If $\{\sigma_{\alpha}\}_{\alpha}$ satisfies \PropB, \PropE, and~\PropC, then
    $\twist$ is a Lie groupoid with respect to  the smooth structure from Theorem~\ref{thm:E is mfd}
    and
    $ \ses $ is a Lie twist.
\end{thm}
\begin{proof}
To check that multiplication is smooth, fix a point $(e_{1},e_{2}) \in
\twist\comp$. Choose $\alpha \in \mathfrak{A}$ such that $\pi(e_{1} e_{2})
\in U_{\alpha}$. Since  we have a base for the topology on $G$ and
multiplication is continuous, there exist $\alpha_{1}, \alpha_{2} \in
\mathfrak{A}$ such that $\pi(e_{i}) \in U_{i}\coloneqq U_{\alpha_{i}}$ and
$U_{1} U_{2} \subseteq U_{\alpha}$. Since $\pi\inv(U_{1}) \fpsr
\pi\inv(U_{2})$ is a neighbourhood of $(e_{1}, e_{2})$ in $\twist\comp$, it
suffices to show that multiplication is smooth on $\pi\inv(U_{1}) \fpsr
\pi\inv(U_{2})$.

As before, each $\beta\in\mathfrak{A}$ gives rise to a diffeomorphism
$\psi_{\beta}\colon\pi\inv(U_\beta) \to \mbb{T} \times U_\beta$ such that $z
\cdot \sigma_\beta(\gamma) \mapsto (z, \gamma)$; we write $\psi_{i}\coloneqq
\psi_{\alpha_{i}}$ for $i=1,2$. The map $\psi_{1} \times \psi_{2} \colon
\pi\inv(U_{1}) \times \pi\inv(U_{2}) \to \mbb{T} \times U_{1} \times \mbb{T}
\times U_{2}$ is a diffeomorphism. Let $\psi \colon \pi\inv(U_{1}) \times
\pi\inv(U_{2}) \to \mbb{T}^2 \times U_{1} \times U_{2}$  be the
diffeomorphism obtained by composing the map $(z_{1}, \gamma_{1}, z_{2},
\gamma_{2}) \mapsto (z_{1}, z_{2}, \gamma_{1}, \gamma_{2})$ with $\psi_{1}
\times \psi_{2}$. Then $\psi$ carries $\pi\inv(U_{1}) \fpsr \pi\inv(U_{2})$
to $\mbb{T}^2 \times (U_{1} \fpsr U_{2}$). Since $r, s\colon \twist \to
\twist\z$ are submersions (Proposition~\ref{prp:E:PropB and
PropE}\ref{item:E:r,s,pi submersions}), the subset $\pi\inv(U_{1}) \fpsr
\pi\inv(U_{2})$ is an embedded submanifold of $\pi\inv(U_{1}) \times
\pi\inv(U_{2})$ (Proposition~\ref{prop:transverse maps implies fibred product
is mfd}). Likewise $U_{1} \fpsr U_{2}$ is an embedded submanifold of $U_{1}
\times U_{2}$ and hence $\mbb{T}^2 \times (U_{1} \fpsr U_{2})$ is an embedded
submanifold of $\mbb{T}^2 \times U_{1} \times U_{2}$. By \cite[Theorem~5.27
and Corollary~5.30]{Lee:Intro}, $\psi$  therefore restricts to a
diffeomorphism $\psi \colon \pi\inv(U_{1}) \fpsr \pi\inv(U_{2}) \to \mbb{T}^2
\times (U_{1} \fpsr U_{2})$. The trivialisation map $\psi_{\alpha} \colon
\pi\inv(U_{\alpha}) \to \mbb{T} \times U_{\alpha}$ restricts to a
trivialisation, also denoted $\psi_{\alpha}$, of the open $\mbb{T}$-invariant
subset $\pi\inv(U_{1} U_{2})$. Let $M  \colon  \twist\comp \to \twist$ be the
multiplication map. Using that the $\mbb{T}$-action is central at the second
step, we calculate:
\[
    \psi_{\alpha}(M(\psi\inv(z_{1}, z_{2}, \gamma_{1}, \gamma_{2})))
        = \psi_{\alpha}(M(\psi_{1}\inv(z_{1}, \gamma_{1}), \psi_{2}\inv(z_{2}, \gamma_{2})))
        = \psi_{\alpha}\big(z_{1}z_{2} \cdot (\sigma_{\alpha_{1}}(\gamma_{1})\sigma_{\alpha_{2}}(\gamma_{2}))\big).
\]
By \PropEC, the function $g = g_{\alpha,\alpha_{1},\alpha_{2}}  \colon  U_{1}
\fpsr U_{2} \to \mbb{T}$ determined by
$\sigma_{\alpha_{1}}(\gamma_{1})\sigma_{\alpha_{2}}(\gamma_{2}) =
g(\gamma_{1},\gamma_{2})\cdot \sigma_{\alpha}(\gamma_{1}\gamma_{2})$ is
smooth. We have
\[
    \psi_{\alpha}\big(z_{1}z_{2} \cdot (\sigma_{\alpha_{1}}(\gamma_{1})\sigma_{\alpha_{2}}(\gamma_{2}))\big)
            = \psi_{\alpha}\big(z_{1}z_{2} \cdot(g (\gamma_{1},\gamma_{2})\cdot \sigma_{\alpha}(\gamma_{1}\gamma_{2}))\big)
            = (z_{1}z_{2}g(\gamma_{1},\gamma_{2}), \gamma_{1}\gamma_{2}).
\]
Since the (iterated) multiplication map $\mbb{T}^3 \to \mbb{T}$, the map $g$
and the multiplication map from $G\comp$ to $G$ are all smooth, we deduce
that multiplication in $\twist$ is smooth.

The twist $\twist$ is by assumption a topological groupoid. By
Proposition~\ref{prp:E:PropB and PropE}\ref{item:E:Ez submfd}, $\twist\z$ is
an embedded submanifold; this takes care of~\ref{item:Liegpd:mfds} and
\ref{item:Liegpd:inclusion map}. By Proposition~\ref{prp:E:PropB and
PropE}\ref{item:E:r,s,pi submersions}, $r$ and $s$ are submersions, so
\ref{item:Liegpd:s and r} holds. We have further proved above that the
multiplication of $\twist$ is smooth, so \ref{item:Liegpd:multiplication map}
holds. By \cite[Proposition 1.1.5]{Mackenzie:2005:book}, $\twist$ is a Lie
groupoid. In Part~\ref{item:E:iota diffeo} and the remainder of
Part~\ref{item:E:r,s,pi submersions} of Proposition~\ref{prp:E:PropB and
PropE}, we have further shown that \ref{item:LT:iota smooth} respectively
\ref{item:LT:pi submersion} of Definition~\ref{dfn:Lie twist} hold, so that
$\ses$ is indeed a Lie twist.
\end{proof}

\subsection{Lie twists over \etale\ groupoids}\label{ssec:Lie over etale}

In this subsection, we revisit Conditions \PropB--\PropD\ in the situation of
a twist $\twist$ over an {\em \etale} Lie groupoid $G$. The point is that
while Condition~\PropE\ is phrased purely in terms of the manifold structure
on $G\z$, the remaining conditions depend on the manifold structure of $G$ as
a whole. We know from Section~\ref{sec:etale smooth units} that for \etale\
groupoids $G$, the manifold structure of $G$ is completely determined by that
of $G\z$; and critically, this is the data that is encoded by a Cartan pair
$B\subseteq A$ of $C^{*}$-algebras together with a smooth subalgebra
$B^{\infty}$ of $B$. So our work in this section, specifically
Corollary~\ref{cor:etale twist from sections}, is the articulation point
between Conditions \PropE--\PropD\ and algebraic conditions on a Cartan pair
of $C^{*}$-algebras (see Section~\ref{sec:Cstar}).

\begin{defn}\label{dfn:theProps:etale case}
Suppose that $ \ses $ is a topological twist over a groupoid $G$, that $G\z$
is a manifold, that $\{U_{\alpha}\}_{\alpha\in\mathfrak{A}}$ is a collection
of open subsets of $G$, and that  $ \{\sigma_{\alpha} \colon U_{\alpha} \to
\pi\inv (U_{\alpha})\}_{ \alpha\in\mathfrak{A}}$ is a family of continuous
sections  of $\pi$. We define the following conditions for the family
$\{\sigma_{\alpha}\}_{\alpha\in\mathfrak{A}}$.
\begin{itemize}
    \item[\namedlabel{property:S
        infty,etale}{\normalfont(S$^{\infty}_{\z}$)}] For all
        $\alpha,\alpha' \in \mathfrak{A}$, the map $r(U_{\alpha} \cap
        U_{\alpha'}) \to \mbb{T} \times G\z$, $r(\gamma) \mapsto
        \iota\inv(\sigma_{\alpha}(\gamma)\sigma_{\alpha'}(\gamma)\inv)$ for
        $\gamma \in U_{\alpha} \cap U_{\alpha'}$, is smooth;
    \item[\namedlabel{property:M
        infty,etale}{\normalfont(M$^{\infty}_{\z}$)}] For all $\alpha,
        \alpha_{1}, \alpha_{2} \in \mathfrak{A}$, the map $r(U_{\alpha_{1}}
        U_{\alpha_{2}} \cap U_{\alpha}\inv) \to \mbb{T} \times G\z$,
        $r(\gamma_{1}) \mapsto
        \iota\inv\big(\sigma_{\alpha_{1}}(\gamma_{1})\sigma_{\alpha_{2}}(\gamma_{2})\sigma_{\alpha}(\gamma_{1}\gamma_{2})\inv\big)$
        for all $(\gamma_{1},\gamma_{2}) \in U_{\alpha_{1}} \fpsr
        U_{\alpha_{2}}$ with $\gamma_{1}\gamma_{2} \in U_{\alpha}$, is
        smooth; and
    \item[\namedlabel{property:I
        infty,etale}{\normalfont(I$^{\infty}_{\z}$)}] For all
        $\alpha,\alpha' \in \mathfrak{A}$, the map $r(U_{\alpha} \cap
        U_{\alpha'}\inv) \to \mbb{T} \times G\z$, $r(\gamma) \mapsto
        \iota\inv(\sigma_{\alpha}(\gamma)\sigma_{\alpha'}(\gamma\inv))$ for
        $\gamma \in U_{\alpha} \cap U_{\alpha'}\inv$, is smooth.
    \end{itemize}
\end{defn}
\begin{remark}
Condition~\PropE\ is already given in terms of functions defined on $G\z$. So
similarly to (U$^{\infty}_{\mathbb{T}}$), any reasonable definition of a
Condition~(U$^{\infty}_{\z}$) would coincide with \PropE.
\end{remark}

The above definition makes no reference to a smooth structure on $G$. The
smooth structure that is induced on an \etale\ $G$ from $G\z$ is built in a
way that the above conditions translate to Conditions~\PropB, \PropC, \PropD\
in Definition~\ref{dfn:theProps}. To be precise:

\begin{lemma}\label{lem:etalisation of Props}
Let $G$ be an \etale\ groupoid such that $G\z$ is a manifold and such that
$G$ acts smoothly on $G\z$. Suppose that $ \ses $ is a topological twist. Let
$\{U_{\alpha} \}_{\alpha\in\mathfrak{A}}$ be an open cover of $G$, consisting
of bisections, and for each $\alpha \in \mathfrak{A}$, let $\sigma_{\alpha}
\colon U_{\alpha} \to \pi\inv (U_{\alpha})$ be a continuous section. Suppose
that $\{\sigma_{\alpha}\}_{\alpha}$ satisfies~\PropE. With respect to the
Lie-groupoid structure on $G$ obtained from Proposition~\ref{prop:lcdiff iff
etale Lie gpd}, the following hold:
\begin{enumerate}[label=\textup{(\arabic*)}]
    \item $\{\sigma_{\alpha}\}_{\alpha}$ satisfies \PropB\ if and only if
        it satisfies \PropBz;
    \item $\{\sigma_{\alpha}\}_{\alpha}$ satisfies \PropC\ if and only if
        it satisfies \PropCz; and
    \item $\{\sigma_{\alpha}\}_{\alpha}$ satisfies \PropD\ if and only if
        it satisfies \PropDz.
\end{enumerate}
\end{lemma}

\begin{proof}
    By Lemma~\ref{lem:lcdiff => all bisections are smooth}, every bisection of $G$ is a smooth bisection. In particular, for all $\alpha\in\mathfrak{A}$,  $r|_{U_{\alpha}}\colon U_{\alpha} \to r(U_{\alpha})$ is a diffeomorphism. The result follows.
\end{proof}

We can now combine all of our previous results.
\begin{corollary}\label{cor:etale twist from sections}
Let $G$ be an \etale\ groupoid such that $G\z$ is a manifold and let $\ses$
be  a topological twist. Let $\{U_{\alpha} \}_{\alpha \in \mathfrak{A}}$ be a
base for the topology on $G$ consisting of bisections, and for each $\alpha
\in \mathfrak{A}$, let $\sigma_{\alpha} \colon U_{\alpha} \to \pi\inv
(U_{\alpha})$ be a continuous section. Suppose that
$\{\sigma_{\alpha}\}_{\alpha}$ satisfies~\PropE, \PropBz~and~\PropCz. If $G$
acts smoothly on $G\z$, then there are unique smooth structures on $\twist$
and $G$ with respect to which the $\sigma_{\alpha}$ are all smooth,
$\PB{\twist}{\pi}{G}$ is a smooth principal $\mbb{T}$-bundle, and $G$ is an
\etale\ Lie groupoid. With respect to these smooth structures, $\sestriple$
is a Lie twist.
\end{corollary}
\begin{proof}
By Proposition~\ref{prop:lcdiff iff etale Lie gpd}, since $G$ acts smoothly
on $G\z$, $G$ has a unique smooth structure with respect to which it is an
\etale\ Lie groupoid. Moreover, a chart around any $\gamma \in G$ is given by
$(V, \varphi \circ r|_V)$, where $V$ is a (sufficiently small) open bisection
containing $\gamma$ and $(r(V), \varphi)$ is a chart around $r(\gamma)$ in
$G\z$; see Proposition~\ref{prop:cR}. Using that $\{U_{\alpha}\}_{\alpha}$ is
a cover of open bisections, we may therefore let $\{(V_\beta,
\psi_\beta)\}_{\beta \in \mathfrak{B}}$ be an atlas of $G$ for which there
exists a refinement map $F \colon \mathfrak{B} \to \mathfrak{A}$ such that
$V_{\beta} \subseteq U_{F(\beta)}$ for each $\beta$.

Lemma~\ref{lem:etalisation of Props} states that the assumed properties on
$\sigma_{\alpha}$ translate to the three properties~\PropE, \PropB\
and~\PropC\ with respect to the Lie-groupoid structure on $G$. Thus, by
Lemma~\ref{lem:refinement}\ref{item:refinments inherit PropBCD}, the sections
$\kappa_\beta \coloneqq  \sigma_{F(\beta)}|_{V_\beta}$ also satisfy the three
properties~\PropE, \PropB\ and \PropC. Since the domains of these refined
sections coincide with the domains of the charts of an atlas for $G$, the
result now follows from Theorems \ref{thm:E is mfd}~and~\ref{thm:E is Lie}.
\end{proof}

\begin{remark}
If we {\em start} with a Lie twist $ \ses $ where $G$ is \'etale, then we can
choose a family $\{(U_{\alpha}, \sigma_{\alpha})\}_{\alpha}$ of smooth local
sections of $\pi$ supported on bisections; see Lemma~\ref{lem:motivation}.
These then satisfy the conditions of Lemma~\ref{lem:etalisation of Props},
and the uniqueness assertion in Corollary~\ref{cor:etale twist from sections}
implies that the resulting smooth structure is the same as the one we started
with.
\end{remark}

\section{Smooth Cartan triples}\label{sec:Cstar}

In this section we prove our main $C^{*}$-algebraic theorem. Given a Cartan
pair $B \subseteq A$ in which $\widehat{B}$ is a manifold, we describe
algebraic conditions---phrased purely in terms of the smooth subalgebra
$B^{\infty}$ of $B$, and $C^{*}$-algebraic data intrinsic to the pair $B
\subseteq A$---that are equivalent to the corresponding Weyl groupoid acting
smoothly on its unit space, and the sections of the Weyl twist corresponding
to the given normalisers satisfying \PropE--\PropD. 
We incorporate all of this data into our definition of a smooth Cartan triple
$(A, B, \mathscr{N})$. Our main theorem says, roughly speaking, that every
smooth Cartan triple arises from a Lie twist over an effective \etale\ Lie
groupoid, and conversely that every such twist gives rise to a smooth Cartan
triple. We briefly discuss, in two concluding remarks, how our result can be
combined with Connes' reconstruction theorem, and also the dependence of the
smooth structure on the Weyl twist coming from our theorem on the choice of a
family $\mathscr{N}$ of normalisers in a smooth Cartan triple.

We need a little background on twisted groupoid $C^*$-algebras. Given a twist
$\ses$, we write $C_{c}(G; \twist)$ for the space of compactly supported
functions $f \colon \twist \to \mbb{C}$ that are $\mbb{T}$-equivariant in the
sense that $f(z\cdot e) = zf(e)$ for all $e \in \twist$ and $z \in \mbb{C}$.
The twisted $C^*$-algebra $C^*_r(G; \twist)$ is the completion of $C_{c}(G;
\twist)$ in the norm determined by a natural family of regular
representations \cite{Kum:Diags, Sims:gpds}. For $f \in C_{c}(G\z)$ the
formula $\tilde{f}(\iota(z, x)) = zf(x)$ defines the unique $\tilde{f} \in
C_{c}(G;\twist)$ such that $\tilde{f}|_{\twist\z} = f$. The map $f \mapsto
\tilde{f}$ extends to an isomorphism of $C_{0}(G\z)$ onto an abelian
subalgebra of $C^*_r(G; \twist)$, and we use this isomorphism to regard
$C_{0}(G\z)$ as a subalgebra of $C^*_r(G; \twist)$. We can make this
identification precise as follows. The identity map on $C_{c}(G; \twist)
\subseteq C^*_r(G; \twist)$ extends to a norm-decreasing linear map from
$C^*_r(G; \twist)$ to the space $C_{0}(G; \twist)$ of $\mbb{T}$-equivariant
elements of $C_{0}(\twist)$, and we frequently use this map to identify
elements of $C^*_r(G;\twist)$ with functions in $C_{0}(G; \twist)$. For more
detail see \cite[Chapter~11]{Sims:gpds} and \cite[Theorem~1.1]{DWZ:jmap}.

We also need a little background about Renault's construction. A \emph{Cartan
subalgebra} $B$ of a $C^{*}$-algebra $A$ is a maximal abelian subalgebra such
that there is a faithful conditional expectation $\condExp \colon A \to B$
and such that the set $N(B) \coloneqq \{n \in A : nBn^{*} \cup n^{*}Bn
\subseteq B\}$ of normalisers of $B$ densely spans $A$
\cite[Definition~5.1]{Renault:Cartan}. (Renault's definition requires, in
addition, that $B$ contains an approximate unit for $A$, but Pitts proved
\cite[Theorem~2.6]{pitts2022normalizers} that this is a consequence of the
remaining conditions.) Let $\widehat{B}$ be the spectrum of $B$, so that we
may identify $B$ with $C_{0} (\widehat{B})$ via the Gelfand isomorphism.

We will explain how  $B\subseteq A$ gives rise to an effective groupoid
$G_{A, B}$ and a twist $\twist_{A, B}$ (cf.~\cite[1.6]{Kum:Diags} or
\cite[Proposition~4.7]{Renault:Cartan}). Recall that $\suppo(k)$ denotes the
\emph{open support}  $\{x \in \widehat{B}\,\vert\, k(x) \neq 0\}$ of $k \in
C_{0}(\widehat{B})$.

If $B \subseteq A$ is any abelian $C^*$-subalgebra of a $C^*$-algebra $A$
that contains an approximate unit for $A$, and if $n \in A$ satisfies $n^*Bn,
nBn^* \subseteq B$, then $n^{*}n, nn^{*} \in B$, and there exists a unique
*-isomorphism
\begin{equation}\label{eq:def-theta}
    \theta_{n} \colon C_{0} (\suppo  (n^{*}n)) \to C_{0} (\suppo  (nn^{*}))
    \quad
    \text{ that satisfies }
    \quad
    n f n^{*} = \theta_{n}(f) \, n n^{*}
\end{equation}
for all $ f \in C_{0}(\widehat{B})$
\cite[Lemma~2.3]{CRST:2021:Reconstruction}. This formula applied to $f =
(n^*n)^p$ shows that $\theta_n((n^*n)^p) = (nn^*)^p$ for all $p \ge 0$. Since
$n (n^*n)^p = (nn^*)^p n$ for all $p$, it follows that
\begin{equation}\label{eq:push through n}
    n h = \theta_n(h) n\quad\text{ for all $h \in C^*(n^*n) = C_{0}(\supp^0(n^*n))$;}
\end{equation}
see also \cite[Lemma 4.2]{DGNRW:Cartan}. Recall that, if $n,m\in N(B)$, then
$\theta_{n} \circ \theta_m = \theta_{nm}$ and $\theta_{n^{*}} =
\theta_{n}\inv$  (cf.~\cite[Corollary~1.7]{Kum:Diags}). We point out that, in
the literature, the focus is usually on the homeomorphism
\begin{equation}\label{eq:def-theta-hat}
    \widehat{\theta}_{n}\colon \suppo  (n^{*}n) \to \suppo  (nn^{*})
        \quad
        \text{ determined by }
        \quad
        \theta_{n} (f) = f\circ \widehat{\theta}_{n}\inv.
\end{equation}

Now suppose that $B\subseteq A$ is a Cartan pair in the sense of
\cite{Renault:Cartan}. The {\em Weyl groupoid} $G_{A, B}$ is the quotient of
\[
    \left\{(n, x) : n\in N(B), x\in \suppo  (n^{*}n) \right\}
\]
by the equivalence relation
\[\begin{split}
       & (n, x) \sim (m, y) \iff \\
       &x=y \text{ and }
        \theta_{n}\vert_{C_{0}(U)} = \theta_m\vert_{C_{0}(U)}
        \text{ for an open neighbourhood $U\subseteq\widehat{B}$ of $x$.}
\end{split}\]
This groupoid has unit space $\{[k,x] : k \in B, k(x)\not= 0\}$ homeomorphic
to $\widehat{B}$ via $[k,x] \mapsto x$, and groupoid operations $r([n,x]) =
\widehat{\theta}_{n}(x)$, $s([n,x]) = x$, $[n,\widehat{\theta}_m(x)][m, x] =
[nm,x]$ and $[n,x]\inv = [n^{*}, \widehat{\theta}_{n}(x)]$. The sets $\{[n,x]
: (n^{*}n)(x) \not= 0\}$ indexed by $n \in N(B)$ constitute a base of open
bisections for the topology on~$G_{A, B}$.

Let $\twist_{A, B}$ denote the quotient of the same set $\{(n,x) : n \in
N(B), x \in \suppo(n^{*}n)\}$ by the more stringent equivalence relation
$\approx$ given by
\[
    (n,x) \approx (m,y) \quad\iff\quad (n,x)\sim(m,y)\text{ and } \condExp(n^{*}m)(x) > 0.
\]
We write $\llbracket n, x\rrbracket$ for the equivalence class of $(n, x)$
under $\approx$. Then the quotient
\[
    \twist_{A, B} = \left\{(n, x) : n\in N(B), x\in \suppo(n^{*}n) \right\}/{\approx}
\]
is a groupoid with structure maps identical to the ones for $G_{A, B}$ above,
with $[\,, \,]$ replaced with $\llbracket\,, \,\rrbracket$. The sets
$\{\llbracket zn,x\rrbracket : (n^{*}n)(x) \not= 0, z \in U\}$ indexed by
normalisers $n \in N(B)$ and open sets $U \subseteq \mbb{T}$ constitute a
base for the topology on $\twist_{A, B}$. There is a map $\pi \colon
\twist_{A, B} \to G_{A, B}$ given by $\pi(\llbracket n, x\rrbracket) = [n,
x]$, and there is an inclusion $i \colon \mbb{T} \times G_{A, B}\z \to
\twist_{A, B}$ given by $i((z, x)) = [k, x]$ for any $k \in B$ such that
$k(x) = z$. The sequence
\[
    \mbb{T} \times G_{A, B}\z  \overset{\iota}{\to} \twist_{A, B} \overset{\pi}{\to} G_{A, B}
\]
is a twist, called the \emph{Weyl twist} of $B\subseteq A$.

Renault proves in \cite{Renault:Cartan} that
there is an isomorphism $a \mapsto \hat{a}$ from $A$ to $C^*_r(G_{A, B};
\twist_{A, B})$ such that, identifying elements of the reduced $C^*$-algebra
$C^*_r(G_{A, B}; \twist_{A, B})$ with  functions in $C_{0}(\twist_{A, B})$ as
above, $\hat{a}$ is given by the formula
\begin{equation}\label{eq:evaluation map}
\hat{a}(\llbracket n, x\rrbracket) = \frac{P(n^*a)(x)}{\sqrt{n^*n(x)}},
\end{equation}
and this isomorphism restricts to the Gelfand isomorphism $B \mapsto
C_{0}(\widehat{B}) = C_{0}(G_{A,B}\z)$; in particular, the latter is a Cartan
subalgebra of $C^{*}_{r}(G_{A, B};\twist_{A, B})$. Under this identification,
\cite[Proposition~4.8.]{Renault:Cartan} (see also the proof of
\cite[Proposition~11.1.4]{Sims:gpds}) shows that every $\mbb{T}$-equivariant
function on $\twist_{A, B}$ whose support in $\twist_{A,B}$ is the preimage
of a bisection in $G_{A, B}$ is a normaliser of $C_{0}(G_{A, B}\z) \cong B$
in $C^{*}_{r}(G_{A, B}; E_{A, B}) \cong A$, and every normaliser $n\in  N(B)$
has this form.

The following lemma describes how to recognise when the \'etale groupoid
underlying a twist can be made into a Lie groupoid consistent with the given
smooth structure on its unit space solely in terms of $C^*$-algebraic
information. We are most interested in this when $G$ is effective, or
equivalently $C_{0}(G\z)$ is a masa in $C^{*}_{r}(G;\twist)$\footnote{If $G$
is not effective, then it admits an open bisection $U \subseteq G \setminus
G\z$ consisting of isotropy, and any nonzero $f \in C_{c}(U)$ belongs to $B'
\setminus B$, so $B$ is not a masa. If $G$ is effective, then, for example
\cite[Corollary~5.3]{CRST:2021:Reconstruction} shows that $B$ is maximal
abelian.}, but it is not difficult to prove without this additional
hypothesis.

\begin{ntn}
Throughout this section if $M$ is a manifold, we will write $C^\infty_0(M)
\coloneqq C_0(M) \cap C^\infty(M)$ and $C_c^\infty(M) \coloneqq C_c(M) \cap
C^\infty(M) $.
    In particular, if $M$ is a manifold and $U \subseteq M$ is an open subset, then $C^\infty_0(U)$ is the space of smooth $C_0$-functions on $U$ regarded as a manifold in its own right---it is not, in general, identifiable with $\{f \in C^\infty_0(M) : \suppo(f) \subseteq U\}$ because the extension of a function $f \in C^\infty_0(U)$ by zero off $U$ need not be smooth on $M$. However, $C^\infty_c(U)$ \emph{can} be identified with $\{f \in C^\infty_c(M) : \supp(f) \subseteq U\}$, and we will frequently make this identification.%
    \footnote{%
        Defining $C^\infty_0(M) $ to be
        \(
        \{f\in C^{\infty}(M) : f\text{ and all of its derivatives vanish at infinity}\}
        \)
        may seem more natural, but it was pointed out to us by Sven Raum and Jonathan Taylor that this algebra is only defined if the tangent bundle of the manifold $M$ carries additional structure.
    }
\end{ntn}

\begin{lemma}\label{lem:C*-interpretation of lcdiff}
Suppose that $G$ is an \etale\ groupoid, $\ses$ is a topological twist, and
$G\z$ is a manifold. Let $A=C^{*}_{r}(G;\twist)$, and let $B  = C_{0}(G\z)
\subseteq A$. Then the following are equivalent:
\begin{enumerate}[label=\textup{(\arabic*)}]
    \item\label{item:C*-interpretation:G acts smoothly} $G$ is an \etale\
        Lie groupoid;
    \item\label{item:C*-interpretation:all normalisers smooth} for every
        normaliser $n \in N(B)$, the isomorphism $\theta_{n} \colon
        C_{0}(\suppo (n^{*}n)) \to C_{0}(\suppo (nn^{*}))$
        of~\eqref{eq:def-theta} restricts to a *-isomorphism from
        $C_{0}^{\infty}(\suppo (n^{*}n))$ to $C_{0}^{\infty}(\suppo
        (nn^{*}))$; and
    \item\label{item:C*-interpretation:enough normalisers smooth} there is
        a family $\mathscr{N} \subseteq N(B)$ of normalisers of $B$ that
        densely span $A$ such that for each $n \in \mathscr{N}$, the
        isomorphism $\theta_{n} \colon C_{0}(\suppo (n^{*}n)) \to
        C_{0}(\suppo (nn^{*}))$ of~\eqref{eq:def-theta} restricts to a
        *-isomorphism from $C_{0}^{\infty}(\suppo (n^{*}n))$ to
        $C_{0}^{\infty}(\suppo (nn^{*}))$.
\end{enumerate}
\end{lemma}
\begin{proof}
First observe that, by the proof of \cite[Lemma
6.6]{CRST:2021:Reconstruction} (which goes through for the twisted case),  if
$n$ is a normaliser of $B$, regarded as a function on $\twist$, and if
$n(e)\not= 0$, then $r(e) = \widehat{\theta}_n(s(e))$. In particular,
$U_n\coloneqq \pi(\suppo(n))$ is covered by open bisections $\{B_{n,i}\}_{i}$
with the property that $\widehat{\theta}_n$ agrees with $r \circ
(s|_{B_{n,i}})\inv$ on $s(B_{n,i})$.

Now, for \mbox{\ref{item:C*-interpretation:G acts
smoothly}\;$\implies$\;\ref{item:C*-interpretation:all normalisers smooth}},
suppose that $G$ is an \'etale Lie groupoid. By  Proposition~\ref{prop:lcdiff
iff etale Lie gpd} and Lemma~\ref{lem:lcdiff holds for every bisection}, for
every open bisection $U$ of $G$, the map $s|_U \circ (r|_{U})\inv$ is a
diffeomorphism. So~\ref{item:C*-interpretation:all normalisers smooth}
follows from the observation of the first paragraph.

\mbox{\ref{item:C*-interpretation:all normalisers
smooth}\;$\implies$\;\ref{item:C*-interpretation:enough normalisers smooth}}
is trivial: every $n \in C_{c}(G; \twist)$ such that $\pi(\suppo(n))$ is a
bisection is a normaliser of $B$. A partition-of-unity argument shows that
such elements span $C_{c}(G;\twist)$. So they densely span $C^{*}_{r}(G; E)$.

For \mbox{\ref{item:C*-interpretation:enough normalisers
smooth}\;$\implies$\;\ref{item:C*-interpretation:G acts smoothly}}, note that
since the elements of $\mathscr{N}$ densely span $A$, their supports cover
$G$. So the observation of the first paragraph shows that $G$ is covered by
open bisections that act smoothly on $G\z$; that is, $G$ acts smoothly on
$G\z$. So Proposition~\ref{prop:lcdiff iff etale Lie gpd} shows that $G$ is
an \'etale Lie groupoid.
\end{proof}

It will be convenient to be able to move back and forth between sections of a
Lie twist over a Lie groupoid and $\mbb{T}$-equivariant functions on the
twist. The technical result Lemma~\ref{lem:f<->sigma} below shows that we can
do so.

\begin{defn}\label{dfn:phase}
For $X$ a topological space and $f\in C_{0}(X)$, we define $\Ph(f) \in
C_b(\suppo (f))$ to be the unique function such that $f = \Ph(f)|f|$ on
$\suppo (f)$. We write $\Ph(f)(x)=\Ph(f(x))$ for $x\in \suppo (f)$.
\end{defn}

\begin{lemma}\label{lem:f<->sigma}
Suppose $ \ses $ is a topological twist over an \etale\ groupoid $G$. Fix an
open bisection $U\subseteq G$ and a continuous section $\sigma\colon U \to
\twist$  of $\pi$.
\begin{enumerate}[label=\textup{(\arabic*)}]
    \item\label{item:t_sigma:cts} There exists a unique continuous
        $\mbb{T}$-equivariant function $t_{\sigma}\colon
        \pi\inv(U)\to\mbb{T}$ such that $t_{\sigma}(e)\cdot
        \sigma(\pi(e))= e$ for all $e\in \pi\inv(U)$.
    \item\label{item:T-equivariant fct gives pi-section} Any $n\in C_{0}
        (G;\twist)$ with $\suppo(n)=\pi\inv(U)$ gives rise to a continuous
        section $\sigma_{n}\colon U\to \pi\inv(U)$ of $\pi$ given by
        \begin{equation}\label{eq:sigma_n formula}
            \sigma_n (\pi(e)) = \overline{\Ph(n(e))}\cdot e
            .
        \end{equation}
   In this case, the continuous function from Part~\ref{item:t_sigma:cts}
   is given by $t_{\sigma_{n}}=\Ph(n)$.

    \item\label{item:pi-section gives T-equivariant fcts} The  map
        \[
            C_{0}\big(r(U), [0,\infty)\big) \to \{n\in C_{0} (G;\twist) : \suppo(n) \subseteq \pi\inv(U), \sigma|_{\suppo(n)} = \sigma_n\}
        \]
        that sends $\varphi \in C_{0}(r(U),[0,\infty))$ to the
        $\mbb{T}$-equivariant function
        \[
            n_{\varphi}\colon \qquad e \mapsto
            \begin{cases}
                \varphi(r(e))\, t_{\sigma}(e) & \text{if } e\in \pi\inv(U),
                \\
                0& \text{otherwise},
            \end{cases}
        \]
        is  bijective. Its inverse carries $n$ to the function
         \[
            \varphi_{n}\colon \qquad x \mapsto
                \bigl|\bigl(n\circ \sigma\circ r|_{U}\inv\bigr)(x)\bigr|
            .
        \]
        Moreover, $\suppo(n_{\varphi})=\pi\inv(U\cap
        r\inv(\suppo(\varphi)))$.
\end{enumerate}

Now suppose $ \ses $ is a Lie twist, $G$ is an \etale\ Lie groupoid, and
$\sigma$ is smooth. Then the following hold.
\begin{enumerate}[label=\textup{(\arabic*)}, resume]
    \item\label{item:t_sigma:smooth} The function $t_{\sigma}$ in
        Part~\ref{item:t_sigma:cts} is smooth.
    \item\label{item:T-equivariant fct gives pi-section:smooth} If $n$ as
        in Part~\ref{item:T-equivariant fct gives pi-section} is smooth,
        then the corresponding section $\sigma_{n}$ is likewise smooth.
    \item\label{item:pi-section gives T-equivariant fcts:smooth} The map
        $\varphi \mapsto n_\varphi$ from~\ref{item:pi-section gives
        T-equivariant fcts} restricts to a bijection
        \[
            C^{\infty}_{c}\big(r(U), [0,\infty)\big) \to \{n\in C^{\infty}_{c} (G;\twist) : \supp(n) \subseteq \pi\inv(U), \sigma|_{\suppo(n)} = \sigma_n\};
        \]
        that is, given a smooth compactly supported $\varphi$, the function
        $n_{\varphi}$ is smooth and compactly supported, and given a smooth
        compactly supported $n$, the function $\varphi_{n}$ is smooth and
        compactly supported.
\end{enumerate}
\end{lemma}

\begin{proof}
\ref{item:t_sigma:cts} (respectively \ref{item:t_sigma:smooth}):
 As explained in \cite[Lemma 2.4(c)]{Armstrong:2022:Uniqueness}, the map $t_{\sigma}$ is uniquely determined by the equality $\iota(t_{\sigma}(e), r(e))=e\sigma(\pi(e))\inv$ for all $e\in \pi\inv(U)$. It is thus continuous (smooth) since $\sigma$ and  inversion and multiplication in $\twist$ are continuous (smooth), and  since $\iota$ is a homeomorphism  (respectively diffeomorphism; see Condition~\ref{item:LT:iota diffeo}).

\ref{item:T-equivariant fct gives pi-section} (respectively
\ref{item:T-equivariant fct gives pi-section:smooth}): To see that
$\sigma_{n}$ is well-defined, note first that $n(e)\neq 0$ if $\pi(e)\in U$,
so it makes sense to consider $\Ph(n(e))$. Moreover, if $\pi(e)=\pi(e')$,
then there exists $z\in \mbb{T}$ with $e'=z\cdot e$, so that
$\mbb{T}$-equivariance of $n$ implies
\[
 \frac{\overline{ n (e')}}{| n (e')|}\cdot e'
 =
  \frac{\overline{ n (z\cdot e)}}{| n (z\cdot e)|}\cdot (z\cdot e)
  =
  \frac{\overline{z}\ \overline{ n (e)}}{|z \ n(e)|}\cdot (z\cdot e)
  =
 \frac{\overline{ n (e)}}{| n (e)|}\cdot e,
\]
which proves that $\sigma_{n}(\pi(e))$ does not depend on the chosen
representative $e$ of $\pi(e)$.

The map $\pi\inv(U)\to \twist, e\mapsto \overline{\Ph(n(e))}\cdot e$, is
continuous (smooth) since $n$, inversion of complex numbers, conjugation of
non-zero complex numbers, and the $\mathbb{T}$-action on $\twist$ are all
continuous (respectively smooth; see Condition~\ref{item:LT:smooth action}).
Since $\pi$ is a topological principal $\mathbb{T}$-bundle (respectively a
smooth such bundle; see Condition~\ref{item:LT:principal bundle}), we know
that the bundle $\pi\inv(U)$ is locally homeomorphic (diffeomorphic) to the
topological (smooth) product bundle $\mathbb{T}\times U$. Since the map
$\gamma\mapsto (1,\gamma)$ is a  continuous (smooth) map from $U$ to
$\mathbb{T}\times U$, it follows that $\sigma_{n}$ is  continuous (smooth) as
a concatenation of  continuous (smooth) maps.

Lastly, for $e\in\pi\inv(U)$, we have
\begin{align*}
    t_{\sigma_{n}}(e)\cdot r(e)
    &=
    e\sigma(\pi(e))\inv
    &&\text{by choice of } t_{\sigma_{n}}
    \\
    &=
    e \bigl(\overline{\Ph (n(e))} \cdot e\bigr)\inv
    &&\text{by definition of }\sigma_n
    \\
    &=
    \Ph(n(e))\cdot r(e),
\end{align*}
so $t_{\sigma_{n}}=\Ph(n)$.

\ref{item:pi-section gives T-equivariant fcts}
 Fix  $\varphi\in C_{0} (r(U), [0,\infty))$.
 We want to show that $n_{\varphi}$ is continuous, so suppose $\{e_{\lambda}\}_{\lambda}$ is a net in $\twist$ that converges to $e$. If $n_\varphi(e)\neq 0$, then eventually $e_\lambda\in \pi\inv(U)$, and so continuity of the map $e\mapsto \varphi(r(e)) t_\sigma (e)$ on $\pi\inv(U)$ implies $n_{\varphi}(e_\lambda)\to n_\varphi(e)$. If $n_\varphi(e)= 0$, then $r(e)\notin
 \suppo(\varphi)$, so $0=\varphi(r(e))=\lim_{\lambda} \varphi(r(e_{\lambda}))$. Hence $|\varphi(r(e_{\lambda})) t_\sigma (e_{\lambda})|=\varphi(r(e_{\lambda}))\to 0$, which again implies $n_{\varphi}(e_\lambda)\to 0=n_\varphi(e)$. This proves that $n_{\varphi}$ is continuous.

 If $K_{0}\subseteq r(U)$ is compact such that $\varphi|_{r(U)\setminus K_{0}} \leq \epsilon$, then since $G$ is \etale, the subset $K_{1}\coloneqq r\inv (K_{0})\cap U$ of $U\subseteq G$ is homeomorphic to $K_{0}$ and hence also compact. Because $\sigma\colon U\to \pi\inv(U)$ is a continuous section, Lemma~\ref{lem:sections => triv's} implies that $\mathbb{T}\times U$ is homeomorphic to $\pi\inv(U)$ via $(z,\gamma)\mapsto z\cdot \sigma(\gamma)$. This homeomorphism maps  the compact set $\mathbb{T}\times K_{1}$ to $K\coloneqq\pi\inv(K_{1})$, so $K$ is compact. Now pick any $e=z\cdot \sigma(\gamma)\in \pi\inv(U)$. If $r(e)\in K_{0}$, then $\gamma\in r\inv (K_{0})\cap U=K_{1}$, which implies $e\in K$. In other words,
 \begin{align*}
     \| n_{\varphi}|_{\twist\setminus K} \|_{\infty}
     =
     \| n_{\varphi}|_{ \pi\inv(U)\setminus K } \|_{\infty}
     =
     \| \varphi|_{ r(U)\setminus K_{0} } \|_{\infty}
     \leq \epsilon.
 \end{align*}
 We conclude that $n_{\varphi}\in C_{0}(\twist)$.

 From the equality
\[
\sigma(\pi(z\cdot e))\inv (z\cdot e)
=
z\cdot \sigma(\pi(e))\inv e
\]
it follows that $t_{\sigma}(z\cdot e)=z t_{\sigma}(e)$, so $n_{\varphi}\in
C_{0}(G;\twist)$. Since $n_{\varphi}(e)\neq 0$ if and only if  $r(e)\in
\suppo(\varphi)$, we have $\suppo(n_{\varphi})\subseteq \pi\inv(U)$, or
equivalently, $\pi(\suppo(n_{\varphi})) \subseteq U$. Since $\varphi$ is
positive-valued and since $t_{\sigma}$ is $\mbb{T}$-valued, for $e \in
\suppo(n_\varphi)$, we have
\[
\sigma_{n} (\pi(e))
    \overset{\eqref{eq:sigma_n formula}}{=} \frac{\overline{\varphi(r(e)) t_{\sigma}(e)}}{|\varphi(r(e)) t_{\sigma}(e)|} \cdot e
    = \overline{t_{\sigma}}(e)\cdot e
    \overset{\ref{item:t_sigma:cts}}{=} \sigma(\pi(e)),
\]
so indeed $\sigma_{n} = \sigma|_{\suppo(n)}$. All in all, we have shown that
$n_{\varphi}$ is of the claimed form.

Conversely, fix $n\in C_{0}(G;\twist)$ with open support contained in
$\pi\inv(U)$ and $\sigma|_{\suppo(n)} = \sigma_{n}$. Since $G$ is \etale,
$r\colon U\to r(U)$ is a homeomorphism. As a composition of continuous maps,
$\varphi_{n}$ is therefore continuous. By construction, $\varphi_{n}$ is
defined on $r(U)$ and takes positive values.

To see that $\varphi_n$ vanishes at infinity, fix $\epsilon>0$. Since $n$
vanishes at infinity, there exists $K\subseteq \pi\inv(U)$ compact such that
$n|_{\twist\setminus K}$ is bounded by $\epsilon$. To show that $\varphi_{n}$
is bounded by $\epsilon$ outside of the compact set $K_{0}\coloneqq
r_{\twist}(K)=r(\pi(K))$, pick $x\in r(U)\setminus K_{0}$, and let
$\gamma\coloneqq r|_{U}\inv (x)$. By choice of $x$, it is not in
$r_{\twist}(\mbb{T}\cdot K)$, so $\sigma(\gamma)\notin \mathbb{T}\cdot K$. In
other words, $\sigma(\gamma)\in \twist\setminus (\mathbb{T}\cdot K)$. Since
this set is contained in $ \mbb{T}\cdot (\twist \setminus K)$, there exists
$z\in\mbb{T}$ and $e\in \twist\setminus K$ with $\sigma(\gamma)=z\cdot e$.
Since $n$ is equivariant, we conclude
\begin{align*}
    \varphi_{n}(x)
    &=
    \bigl|n(\sigma(\gamma))\bigr|
    =
    \bigl|n(z\cdot e)\bigr|
    =
    \bigl|n(e)\bigr|,
\end{align*}
which is bounded by $\epsilon$ since $e\notin K$. As $x\in r(U)\setminus
K_{0}$ was arbitrary, this proves the claim.

It remains to show that the two maps are inverse to one another. Clearly,
$\varphi_{n_{\psi}}=\psi$. To see that $n_{\varphi_{m}}=m$, fix
$e\in\pi\inv(U)$; note that $\gamma\coloneqq \pi(e)$ is the unique element of
$U$ with range $r_{\twist}(e)$. Therefore,
\begin{align*}
    n_{\varphi_{m}} (e)
    &
    =
    \varphi_{m}(r(e))\, t_{\sigma}(e)
    \\
    &=
    |m(\sigma(\gamma))|\, t_{\sigma}(e)
    &&\text{since $r|_{U}\inv(r(e))=\gamma$}
    \\
    &=
    |m(e)|\, t_{\sigma}(e)
    &&\text{since $m$ is $\mbb{T}$-equivariant and $\pi(\sigma(\gamma))=\gamma=\pi(e)$.}
\end{align*}
If $m(e) = 0$ then both sides of this equation are zero; and if $m(e) \not=
0$, then since $\sigma|_{\suppo(m)} = \sigma_m$,
applying Part~\ref{item:t_sigma:cts} at the step labelled~$(*)$, we have
\begin{align*}
    n_{\varphi_{m}} (e)
    &
    =
    |m(e)|\, t_{\sigma}(e)
    \overset{(*)}{=}
    |m(e)|\, \Ph(m(e))
    =
    m(e).
\end{align*}
Thus, $n_{\varphi_{m}}=m$.

\ref{item:pi-section gives T-equivariant fcts:smooth}
    If $M$ is a manifold and $U \subseteq M$ is an open set, then for any $f \in C^{\infty}_{c}(U)$ the extension $\tilde{f} \in C_{c}(M)$ of $f$ such that $\tilde{f}|_{M \setminus U} = 0$ is smooth: since $f \equiv 0$ on the open set $U \setminus \supp(f)$, all of its derivatives are also zero, so $\tilde{f}$ has continuous derivatives of all orders. So it suffices to show that $n$ is smooth on $\suppo(n)$ if and only if $\varphi_{n}$ is smooth on $r(\suppo(n))$.
    For the ``only if" implication, replace every instance of ``continuous'' with ``smooth'' in the proof that $\varphi_{n}$ is continuous, using Lemma~\ref{lem:lcdiff => all bisections are smooth} to see that $r|_{U}$ is a diffeomorphism. For the ``if" implication, suppose that $\varphi$ is smooth on its open support. Then, as above, the extension $\tilde{\varphi}$ of $\varphi$ to $r(U)$ by 0 is smooth too. Since $\sigma$ is smooth, the map $e\mapsto \tilde{\varphi}(r(e)) t_\sigma (e)$ is smooth on $\pi\inv(U)$ by Part~\ref{item:t_sigma:smooth}, completing the proof.
\end{proof}

\begin{lemma}\label{lem:C*-interpretation of Props}
Suppose that $G$ is an \etale\ groupoid and $\ses$ is a topological twist.
Suppose further that $G\z$ is a manifold and that $G$ acts smoothly on $G\z$;
equip $G$ with the manifold structure constructed in
Proposition~\ref{prop:lcdiff iff etale Lie gpd}. Let $A=C^{*}_{r}(G;\twist)$,
$B  = C_{0}(G\z) \subseteq A$, and $\condExp\colon A\to B$ the conditional
expectation of \cite[Proposition~4.3]{Renault:Cartan}. Suppose that
$\mathscr{N} \subseteq A$ is a family of $\mbb{T}$-equivariant functions on
$\twist$ such that, for each $n\in \mathscr{N}$, the set $U_{n}\coloneqq
\pi(\suppo(n))$ is a bisection of $G$. For each $n\in\mathscr{N}$, let
$\sigma_{n}\colon U_{n}\to \twist$ be as in Lemma~\ref{lem:f<->sigma}. Then
the following dictionary holds.
\begin{enumerate}[label=\textup{(\arabic*)}]
    \item\label{item:dict:PropE} $\{\sigma_{n}\}_{n\in\mathscr{N}}$
        satisfies \PropE\ if and only if for all $n\in\mathscr{N}$, the
        function $\Ph(\condExp(n))$ is smooth.
    \item\label{item:dict:PropB} $\{\sigma_{n}\}_{n\in\mathscr{N}}$
        satisfies \PropB\ if and only if for all $m,k\in\mathscr{N}$, the
        function $\Ph(\condExp(mk^{*}))$ is smooth.
    \item\label{item:dict:PropC} $\{\sigma_{n}\}_{n\in\mathscr{N}}$
        satisfies \PropC\ if and only if for all $n,m,k\in\mathscr{N}$, the
        function $\Ph(\condExp(nmk^{*}) )$ is smooth.
    \item\label{item:dict:PropD} $\{\sigma_{n}\}_{n\in\mathscr{N}}$
        satisfies \PropD\ if and only if for all $n,m\in\mathscr{N}$, the
        function $\Ph(\condExp(nm))$ is smooth.
\end{enumerate}
\end{lemma}

\begin{proof}
For $n\in\mathscr{N}$, recall the definition of $\sigma_{n}$:
\begin{equation*}
    \sigma_{n} (\pi(e)) = \frac{\overline{n(e)}}{|n(e)|}\cdot e
    =
    \overline{\Ph(n(e))}\cdot e.
\end{equation*}

\ref{item:dict:PropE} Since $\condExp(n)=n|_{\twist\z}$, we have for $x\in
U_{n} \cap G\z= \pi\inv(U_{n})\cap \twist\z$ that
\[
    \iota(\overline{\Ph(\condExp(n))(x)},x)
    =
    \overline{\Ph(n)(x)}\cdot x
    =
    \sigma_{n} (x),
\]
showing that \ref{item:dict:PropE} holds if and only if all maps
$\Ph(\condExp(n))$ are smooth.

For the other conditions, fix $n,m,k\in\mathscr{N}$ and define the following
functions implicitly:
\begin{align*}
    f_{m,k}&\colon U_{m}\cap U_{k}
    \to \mbb{T},
    &&
    \sigma_{m}(\gamma)\sigma_{k}(\gamma)\inv
       = \iota(f_{m,k}(\gamma), r(\gamma)),
       \\
    g_{n,m,k}&\colon
    \big\{(\gamma_{1}, \gamma_{2}) \in U_{ n } \fpsr U_{ m } :
        \gamma_{1}\gamma_{2} \in U_{ k }\big\}
        \to \mbb{T} ,
    &&
    \sigma_{ n } (\gamma_{1})\sigma_{ m } (\gamma_{2}) \sigma_{ k }(\gamma_{1}\gamma_{2})\inv
        = \iota\bigl(
                g_{n,m,k} (\gamma_{1},\gamma_{2}), r(\gamma_{1})
            \bigr) ,
            \\
        h_{n,m}&\colon U_{ n }\cap  U_{ m }\inv
    \to \mbb{T}  ,
    &&
    \sigma_{ n }(\gamma)\sigma_{ m }(\gamma\inv)
       = \iota(h_{n,m}(\gamma), r(\gamma)).
\end{align*}
As explained in Remark~\ref{rmk:T-conditions}, the following dictionary
holds:
\begin{align*}
    \{\sigma_{n}\}_{n}&\text{ satisfies \PropB}
    &&\iff&&
    \forall m,k\in\mathscr{N}, &f_{m,k} \text{ is smooth},
    \\
    \{\sigma_{n}\}_{n}&\text{ satisfies \PropC}
    &&\iff&&
    \forall n,m,k\in\mathscr{N}, &g_{n,m,k} \text{ is smooth},
    \\
    \{\sigma_{n}\}_{n}&\text{ satisfies \PropD}
    &&\iff&&
    \forall n,m\in\mathscr{N}, &h_{n,m} \text{ is smooth}.
\end{align*}

\ref{item:dict:PropB} We compute for $\gamma=\pi(e)\in U_m\cap U_k$,
\begin{align*}
    \sigma_{m}(\gamma)\sigma_{k}(\gamma)\inv
    &=
    \bigl(\overline{\Ph(m(e))}\cdot e\bigr)\,
    \bigl(\overline{\Ph(k(e))}\cdot e\bigr)\inv
    =
    \Ph\bigl(\overline{m(e)}\ k(e)\bigr) \cdot r_{\twist}(e).
\end{align*}
Note that $m (e)\, \overline{k(e)}$ is exactly the value of the function $m
k^{*}$ at $r_{\twist}(e)=r_{G}(\gamma)$, so we have shown that
\[
    \overline{f_{m,k}} = \Ph(\condExp(mk^{*})) \circ r_{G}
        \quad \text{ on the domain of } f_{m,k}.
\]
Since $U_{m}\cap U_{k}$ is a smooth bisection by Lemma~\ref{lem:lcdiff holds
for every bisection}, we conclude that $\Ph(\condExp(mk^{*}))$ is smooth if
and only if $f_{m,k}$ is smooth, as claimed.

\ref{item:dict:PropC} We compute for
$(\gamma_{1},\gamma_{2})=(\pi(e_{1}),\pi(e_{2}))$ in the domain of
$g_{n,m,k}$:
\begin{align*}
    \sigma_{ n } (\gamma_{1})\sigma_{ m } (\gamma_{2}) \sigma_{ k }(\gamma_{1}\gamma_{2})\inv
    &=
    \bigl(
    \overline{\Ph(n(e_{1}))}\cdot e_{1}
    \bigr)
    \bigl(
    \overline{\Ph(m(e_{2}))}\cdot e_{2}
    \bigr)
    \bigl(
    \overline{\Ph(k(e_{1}e_{2}))}\cdot e_{1}e_{2}
    \inv\bigr)
    \\
    &=
    \Ph\bigl(
        \overline{n(e_{1})\ m(e_{2})}
        \
        k(e_{1}e_{2})
    \bigr)
    \cdot r_{\twist}(e_{1}e_{2})\end{align*}
A quick computation reveals that $n(e_{1})m(e_{2}) \overline{    k(e_{1}e_{2})}$ is the value of $nmk^*$ at  $r_{\twist}(e_{1})=r_{G}(\gamma_{1})$, so we have shown that
\[
    \overline{g_{n,m,k}} =
        \Ph(\condExp( n   m  k^{*})) \circ r_{G}
        \quad
        \text{ on the domain of } g_{n,m,k}.
\]
This proves that $\Ph(\condExp( n   m  k^{*}))$ is smooth if and only if $g_{n,m,k}$ is smooth.

\ref{item:dict:PropD}
For $\gamma=\pi(e)$ in $U_n\cap U_m\inv$, we have
\[
    \sigma_{ n }(\gamma)\sigma_{ m }(\gamma\inv)
    =
    \bigl(
        \overline{\Ph(n(e))}\cdot e
    \bigr)
    \bigl(
        \overline{\Ph(m(e\inv))}\cdot e\inv
    \bigr)
    =
    \Ph
    \bigl(
        \overline{n(e)} m(e\inv)
    \bigr)
    \cdot
    r_{\twist}(e).
\]
As before, a quick computation reveals that $n(e) \overline{m(e\inv)}$ is the value of $nm$ at $r_{\twist}(e)=r(\gamma)$, so we have shown that
\[
    \overline{h_{n,m}} = \Ph(\condExp( n   m )) \circ r_{G}
        \quad \text{ on the domain of }h_{n,m}.
\]
Again, we conclude that $\Ph(\condExp( n   m ))$ is smooth if and only if $h_{n,m}$ is smooth.
\end{proof}

\begin{remark}\label{rmk:I* is superfluous, part 1}
We have included Statement~\ref{item:dict:PropD} of Lemma~\ref{lem:C*-interpretation of Props} mainly for completeness of our characterisation of the relationship between algebraic properties of $A, B$ and the geometric properties of sections of $\twist$.  For the purposes of our main results,  Statement~\ref{item:dict:PropD} is redundant. Specifically, suppose that $\mathscr{N}$ is a family as in Lemma~\ref{lem:C*-interpretation of Props} with the additional property that the sets $U_n$ cover $G$. If $\Ph(\condExp(n))$,  $\Ph(\condExp(mk^*))$ and $\Ph(\condExp(nmk^*))$ are smooth for all $n,m,k \in \mathscr{N}$, then Lemma~\ref{lem:C*-interpretation of Props} shows that the corresponding sections $\{\sigma_n\}_{n \in \mathscr{N}}$ satisfy~\PropE, \PropB~and~\PropC. So we could deduce from Theorem~\ref{thm:E is Lie}
and Theorem~\ref{thm:E is mfd}\ref{item:E is mfd:principal bdl} that $\ses$ is a Lie twist and that the sections $\sigma_n$ are all smooth.
As explained in Remark~\ref{rmk:PropB-D when E is already Lie}, it then it follows that they also satisfy~\PropD, and therefore that
$\Ph(\condExp(nm))$ is smooth for all $n,m \in \mathscr{N}$.

This is a round-about way to deduce that the  $\Ph(\condExp(nm))$ are smooth from the other three conditions, and it is natural to ask for a direct argument. The idea is as follows: fix $n,m \in \mathscr{N}$ and a point $x$ in $\suppo(\condExp(nm))$. Since the $U_k$ cover $G$ we can find $k \in \mathscr{N}$ such that $x \in \suppo(k)$ and therefore $x \in \suppo(k)\cap G\z = \suppo(\condExp(k))$. Now both $\Ph(\condExp(nmk^*))$ and $\Ph(\condExp(k))$ are smooth, so $\Ph(\condExp(nmk^*))\Ph(k)$ is smooth. Since $\condExp$ corresponds to restriction of functions to the copy of $\{1\} \times G\z$ in the Weyl twist, there is a neighbourhood of $x$ where
$\condExp(nmk^*)$ agrees with $\condExp(nm)\condExp(k)^*=\condExp(nm)\overline{\condExp(k)}$. So on this neighbourhood, $\Ph(\condExp(nmk^*))\Ph(\condExp(k))$ and $\Ph(\condExp(nm))\overline{\Ph(\condExp(k))}\Ph(\condExp(k)) = \Ph(\condExp(nm))$ agree. So $\Ph(\condExp(nm))$ is also smooth at $x$.
\end{remark}

\begin{defn}\label{def:Iinftyc(U)}
Let $B$ be a commutative $C^{*}$-algebra, and let $B^{\infty}$ be a subalgebra of~$B$. Let $U \subseteq \widehat{B}$ be an open set.
Let $I_{c}(U) \cong C_{c}(U)$ be the algebraic ideal of $B$ consisting of elements whose Gelfand transforms have compact support contained in $U$, and let $I^{\infty}_{c}(U) \coloneqq B^{\infty} \cap I_{c}(U)$. We say that $f \in C_b(U)$
\emph{multiplies $I^{\infty}_{c}(U)$}
if $f\cdot I^{\infty}_{c}(U) \coloneqq\{fg : g \in I^{\infty}_{c}(U)\} \subseteq I^{\infty}_U$.
\end{defn}

The above definition depends implicitly on the choice of $B^{\infty}$. It is motivated by the following simple lemma which shows that a function $f$ multiplies $C_{c}^{\infty}(U)$ if and only if it is smooth.

\begin{lemma}\label{lem:local multipliers smooth}
Let $M$ be a manifold, let $B \coloneqq  C_{0}(M)$ and let  $B^{\infty} \coloneqq C^{\infty}(M)\cap C_{0}(M)\subseteq B$.  Fix an open subset $U \subseteq M$ and a function $f \in C_b(U)$. Then $f$ is smooth if and only if $f$ multiplies $C_{c}^{\infty} (U)$.
\end{lemma}
\begin{proof}
If $f$ is smooth, then for any $g \in C^{\infty}(U)$, the product $fg$ is smooth. In particular, if $g\in I^{\infty}_{c}(U)$, then $fg$ is smooth. It belongs to $C_{c}(U)$ because its support is contained in that of $g$. So $f$  multiplies $C_{c}^{\infty} (U)$. Conversely, suppose that $f$   multiplies $C_{c}^{\infty} (U)$. Fix a point $x \in U$; we must show that $f$ is smooth at $x$. Since $M$ is a normal topological space, there exists a precompact open set $V$ such that $x \in V \subseteq \overline{V} \subseteq U$, and then by the extension lemma \cite[Lemma~2.26]{Lee:Intro} for smooth functions, there is a smooth function $g \colon M \to \mbb{R}$ such that $g|_{\overline{V}} \equiv 1$ and $\supp(g) \subseteq U$.  In particular, $g \in I^{\infty}_{c}(U)$. Since $f$ multiplies $C_{c}^{\infty} (U)$, we have $fg \in B^{\infty}$. Since $fg \equiv f$ on $V$, we deduce that $f$ is smooth on the neighbourhood $V$ of $x$, and we are done.
\end{proof}

If $B$ is a commutative $C^{*}$-algebra with $\widehat{B}$ a manifold, we  will always choose $B^{\infty}\subseteq B$ to be the image of
\(
C^{\infty}(\widehat{B})\cap C_{0}(\widehat{B})
\)
under the Gelfand transform, and we write $B_{c}^{\infty}$ for the image of
$C_{c}^{\infty} (\widehat{B})
\coloneqq
C_{c} (\widehat{B})\cap C^{\infty} (\widehat{B})$.
If $B$ is a Cartan subalgebra of $A$, then for $n\in N(B)$, we have $\suppo(n^*n)\subseteq G\z$. We write  $B^{\infty}_{c}(n) \coloneqq I^{\infty}_{c}(\suppo(n^*n))$ as defined in Definition~\ref{def:Iinftyc(U)}.

\begin{defn}\label{dfn:smooth Cartan triple}
Suppose that $B\subseteq A$ is a Cartan pair with conditional expectation $\condExp\colon A\to B$, and that $\widehat{B}$ is a manifold.  Let $\mathscr{N}$ be a family of normalisers of $B$ that densely spans $A$. We say that $(A, B, \mathscr{N})$ is a \emph{smooth Cartan triple} if
$B_{c}^{\infty} \subseteq \mathscr{N}\cap B \subseteq B^{\infty}$, and the family $\mathscr{N}$ satisfies the following. %
\begin{enumerate}[leftmargin=1.5cm, label=\textup{(\arabic*)}]
    \item[\namedlabel{property:N,C*}{\normalfont(N$^*$)}]
         For each $n\in\mathscr{N}$,
            we have $nn^*, n^*n \in B^{\infty}$ and $n$ normalises $B^{\infty}$ in the sense that $nB^{\infty} n^{*} \subseteq B^{\infty}$ and $n^{*} B^{\infty} n \subseteq B^{\infty}$.
    \item[\namedlabel{property:U,C*}{\normalfont(U$^*$)}]
        For all $n\in\mathscr{N}$, $\Ph (\condExp(n))$ multiplies $B^{\infty}_{c}(n)$.
    \item[\namedlabel{property:S,C*}{\normalfont(S$^*$)}]
        For all $m,k\in\mathscr{N}$,
        $\Ph (\condExp(mk^{*}))$ multiplies $B^{\infty}_{c}(mk^*)$.
    \item[\namedlabel{property:M,C*}{\normalfont(M$^*$)}]
        For all $n,m,k\in\mathscr{N}$, $\Ph (\condExp(nmk^{*}))$ multiplies $B^{\infty}_{c}(nmk^*)$.
\end{enumerate}
\end{defn}

As seen in Lemma~\ref{lem:C*-interpretation of Props} (using Lemma~\ref{lem:local multipliers smooth}), Conditions~\ref{property:U,C*}, \ref{property:S,C*}, and \ref{property:M,C*} in the definition of a smooth Cartan triple  correspond to Conditions~\ref{property:U infty}, \ref{property:S infty}, and \ref{property:M infty}, respectively, in Definition~\ref{dfn:theProps}. There is a natural fifth condition that we could ask $\mathscr{N}$ to satisfy, namely
\begin{enumerate}
\item[\namedlabel{property:I,C*}{\normalfont(I$^*$)}]
        For all $n,m\in\mathscr{N}$, $\Ph (\condExp(nm))$ multiplies $B^{\infty}_{c}(nm)$.
\end{enumerate}
But as explained in Remark~\ref{rmk:I* is superfluous, part 1}, it follows from the others. Lastly, Condition~\ref{property:N,C*} is motivated by Corollary~\ref{cor:smooth thetas by conjugation} below, which we deduce from the following technical lemma.

\begin{lemma}\label{lem:thetas of smooth normalisers}
Let $B \subseteq A$ be a Cartan pair
and suppose that $\widehat{B}$ is a manifold. For a normaliser $n$ of $B$, the following conditions \ref{item:theta-hat smooth}~and~\ref{item:theta smooth} are equivalent.
\begin{enumerate}[label=\textup{(\arabic*)}]
    \item\label{item:theta-hat smooth}
        The homeomorphism $\widehat{\theta}_{n} \colon \suppo (n^{*}n) \to \suppo (nn^{*})$ of~\eqref{eq:def-theta-hat} has a smooth inverse.
    \item\label{item:theta smooth}
        The isomorphism $\theta_{n} \colon C_{0}(\suppo (n^{*}n)) \to C_{0}(\suppo (nn^{*}))$ of~\eqref{eq:def-theta} maps $B^{\infty}_{c}(n)$ into $B^{\infty}_{c}(n^*)$.
\end{enumerate}
The  condition
\begin{enumerate}[label=\textup{(\arabic*)}, resume]
    \item\label{item:n smooth}
        $n
        B^{\infty}_{c}(n)
        n^{*} \subseteq B^{\infty}$
\end{enumerate}
implies both \ref{item:theta-hat smooth}~and~\ref{item:theta smooth}, and also that $nn^{*}$ is smooth on its open support. If $nn^{*}$ is assumed to be smooth on its open support,  then~\ref{item:theta-hat smooth}, \ref{item:theta smooth}~and~\ref{item:n smooth} are equivalent.
\end{lemma}
\begin{proof}
The implication \ref{item:theta-hat smooth}$\implies$\ref{item:theta smooth} is trivial: the composition of smooth functions is smooth,
and the composition of a compactly supported function with a homeomorphism has compact support.
For \ref{item:theta smooth}$\implies$\ref{item:theta-hat smooth}, fix $y\in \suppo(nn^{*})$, and let  $(V,\psi)$ be a smooth chart around $x=\widehat{\theta}_{n}\inv (y)$
such that $\overline{V}$ is a compact subset of $\suppo(n^*n)$.
Write $\psi=(\psi_{1},\ldots,\psi_{k})$ where $k=\dim(\widehat{B})$ and $\psi_{i}\colon V\to \mbb{R}$. Fix an open neighbourhood $V'$ of $x$ with $\overline{V'}\subseteq V$ and, using the smooth Urysohn Lemma \cite[Proposition~2.25]{Lee:Intro}, choose
a smooth function $h\colon \widehat{B}\to [0,1]$ that is constant $1$ on $V'$ and vanishes off $V\cap \suppo(n^{*}n)$. Note that $h$ is compactly supported, so $h\psi_{i}\in B^{\infty}_{c}(n)$ and by \ref{item:theta smooth}, the map $\theta_{n}(h\psi_{i})\colon \suppo (nn^{*})\to\mbb{R}$ is smooth. By definition of $\theta_{n}$, this implies that $(h\psi_{i}) \circ \widehat{\theta}_{n}\inv \colon \widehat{\theta}_{n}(V') \to \mbb{R}$ is smooth. By choice of $h$, we have $(h\psi_{i}) \circ \widehat{\theta}_{n}\inv=\psi_{i} \circ \widehat{\theta}_{n}\inv$ on $\widehat{\theta}_{n}(V')$, so we have proved that each of the coordinates of $\psi\circ \widehat{\theta}_{n}\inv$ is smooth around $y$, meaning that $\widehat{\theta}_{n}\inv$ is smooth.

To see that \ref{item:n smooth} implies that $nn^*$ is smooth on its open support, fix $x\in\suppo(nn^*)$. Choose
precompact open sets $V, V'$ such that $\widehat{\theta}_n\inv(x) \in V \subseteq \overline{V} \subseteq V' \subseteq \overline{V'} \subseteq \suppo(n^*n)$. By the smooth Urysohn Lemma \cite[Proposition~2.25]{Lee:Intro} there exists a smooth function $h\colon \widehat{B}\to[0,1]$
that is identically~$1$ on $\overline{V}$ and vanishes off $V'$. Since the support of $h$ is contained in the compact set $\overline{V'}\subseteq \suppo(n^*n)$, $h$ is an element of $B_{c}^{\infty} (n)$, and so our assumption implies that $nhn^{*} \in B^{\infty}_{c}(n^*)$.
Equation~\eqref{eq:def-theta} shows that $nhn^{*} = \theta_n(h) nn^{*}$. We have $\theta_n(h) = h\circ\widehat{\theta}_n\inv$, which is identically~$1$ on $W\coloneqq\widehat{\theta}_n(V)$. Hence the smooth function $\theta_n(h) nn^{*}$ agrees with $nn^{*}$ on the neighbourhood~$W$ of~$x$.  In particular, $nn^{*}$ is smooth around~$x$.

For \ref{item:n smooth}$\implies$\ref{item:theta smooth}, fix $f \in B^{\infty}_{c}(n)$ and $x \in \suppo(nfn^{*}) =\suppo(\theta_n(f)) \subseteq \suppo(nn^*)$. It suffices to show that $\theta_n(f)$ is smooth  around~$x$, so let $h\in B^{\infty}$ and $W\subseteq G\z$ be as above. We have $nfn^{*} \in B^{\infty}_{c}(n^*)$ by assumption and $nfn^{*} = \theta_n(f)nn^{*}$ by Equation~\eqref{eq:def-theta}. Since $1/(nn^{*})$ is smooth on~$W$, the function $\theta_n(f)|_W = (\theta_n(f) nn^{*})/(nn^{*})$ is smooth on $W$. In particular, $\theta_n(f)$ is smooth around $x \in W$.

Lastly, assume that $nn^*$ is smooth on its open support; it remains to show that \ref{item:theta smooth}$\implies$\ref{item:n smooth}. Fix $f \in
B^{\infty}_{c}(n)$, so that $\theta_{n}(f)\in B^{\infty}_{c}(n^*)$ by assumption. As a product of a smooth compactly supported function and a smooth function, the element $\theta_{n}(f) nn^*=nfn^*$ of $B$ is smooth and compactly supported. In particular, it is in $B^{\infty}$. Since $f$ was arbitrary, the claim follows.
\end{proof}

\begin{corollary}\label{cor:smooth thetas by conjugation}
Let $B \subseteq A$ be a Cartan pair and suppose that $\widehat{B}$ is a manifold. Suppose that $n$ is a normaliser  of $B$ such that $nn^*,n^*n\in B^{\infty}$. Then the following are equivalent.
\begin{enumerate}[label=\textup{(\roman*)}]
    \item\label{item:theta-hat smooth:two sided both}
        $\widehat{\theta}_{n}$ is a diffeomorphism.
    \item\label{item:theta smooth:two sided both}
        $\theta_{n}$ restricts to a $*$-isomorphism from
        $B^{\infty}_{c}(n)$ to $B_{c}^{\infty}(n^{*})$.
    \item\label{item:n smooth:two sided}
    $n B^{\infty}_{c}(n) n^*$ and $n^* B^{\infty}_{c}(n^*)n$ are contained in $B^{\infty}$.
\end{enumerate}
If $nB^{\infty} n^* \cup n^*B^{\infty} n \subseteq B^{\infty}$, then \ref{item:theta-hat smooth:two sided both}--\ref{item:n smooth:two sided} all hold.
\end{corollary}
\begin{proof}
If $nn^*\in B^{\infty}$, then in particular $nn^*$ is smooth on its open support. Therefore, the equivalence of~\ref{item:theta-hat smooth:two sided both},  \ref{item:theta smooth:two sided both}~and~\ref{item:n smooth:two sided} is exactly the equivalence in Lemma~\ref{lem:thetas of smooth normalisers} of~\ref{item:theta-hat smooth}, \ref{item:theta smooth}~and~\ref{item:n smooth}  for both $n$ and $n^*$, under the extra assumption that $n^*n,nn^*\in B^{\infty}$.

For the final statement, observe that $B^{\infty}_{c}(n)$ and $B^{\infty}_{c}(n^*)$ are subsets of $B^{\infty}$, so if $n B^{\infty} n^*$ and $n^* B^{\infty} n$ are contained in $B^{\infty}$, then in particular $n B^{\infty}_{c}(n) n^*$ and $n^* B^{\infty}_{c}(n^*)n$ are contained in $B^{\infty}$.
\end{proof}

\begin{remark}
It is not clear to us whether the condition that $nB^{\infty} n^* \cup n^*B^{\infty} n \subseteq B^{\infty}$ implies that $ nn^*,n^*n \in B^\infty$; certainly the usual approximate-identity argument will not work since $B^\infty$ is not norm-closed.
\end{remark}

The situation is even nicer for normalisers with compact support. We state the relevant result in terms of sections of a concrete twist because the equivalent formulation in terms of normalisers in an abstract Cartan pair is a little awkward. The formulation of this result and its subsequent use were suggested to us by Tristan Bice; we thank him for the suggestion.

\begin{lemma}\label{lem:n w cpct supp}
    Suppose $\ses$ is a Lie twist over an \etale\ Lie groupoid, and suppose $n\in C^{\infty}_{c}(G;\twist)$ is such that the closure of $\pi(\suppo(n))$ is contained in a precompact bisection.
    Then $nC^{\infty} (G\z) = C^{\infty} (G\z) n\subseteq C_{c}^{\infty}(E)$.
\end{lemma}
\begin{proof}
    If we can show that
    $n C^{\infty} (G\z) \subseteq C^{\infty} (G\z)  n$, then the same argument applied to $n^*$ shows $n^* C^{\infty} (G\z) \subseteq C^{\infty} (G\z)  n^*$ and thus
    \[
        C^{\infty} (G\z)  n
        =
        (n^* C^{\infty} (G\z) )^*
        \subseteq
        (C^{\infty} (G\z)  n^*)^*
        =
        nC^{\infty} (G\z)
        \subseteq C^{\infty} (G\z)  n,
    \]
    giving equality throughout.

    Fix
    $f\in C^{\infty} (G\z)$. It suffices to find $g\in C^{\infty} (G\z)$
    such that $nf=gn$.
    Let $V$ be a precompact open bisection of $G$ such that
    \[
        U\coloneqq \pi(\suppo(n))
        \subseteq \overline{U}
        \subseteq V.
    \]
    On $V$, the map $f\circ s\circ r|_{V}\inv$ is smooth.
    By the smooth Urysohn Lemma, there exists a $\psi \in C^{\infty}(G\z, [0,1])$ that is identically~1 on $r(U)$ and vanishes outside of $r(V)$. Then
    \[
        g \coloneqq (f\circ s\circ r|_{V}\inv) \, \psi
    \]
    is a globally defined smooth function with support contained in the compact set $r(\overline{V})$, so
    $g\in C_{c}^{\infty} (G\z)$.
    Note that $\psi(r(e)) n(e)=n(e)$ for all $e\in \twist$, since $r(\suppo(n))\subseteq r(U)$ where $\psi$ is constant 1. For $e\in \twist$,
    \begin{align*}
        (gn) (e)
        &
        =
        g(r(e))n(e)
        =
        \begin{cases}
            0 & \text{if } e\notin \suppo(n),
            \\
            (f\circ s\circ r|_{V}\inv)(r(e)) n(e)
            & \text{if } e\in \suppo(n).
        \end{cases}
    \end{align*}
    If $e\in \suppo(n)\subseteq \pi\inv(V)$, then $\pi(e)$ is the unique element of $V$ such that $r(\pi(e))=r(e)$, so that
    \[
    (f\circ s\circ r|_{V}\inv)(r(e))
    =
    (f\circ s)(\pi(e))
    =
    f(s(e)).
    \]
    Hence $(gn)(e)=(nf)(e)$, so
    $nf=gn\in C_{c}^{\infty} (G\z) n$.
\end{proof}

\begin{remark}
In the situation of Lemma~\ref{lem:n w cpct supp}, since the closure of $\pi(\suppo(n^*n))$ is contained in a compact open bisection, the smooth Urysohn Lemma \cite[Proposition~2.25]{Lee:Intro} gives functions $f,g \in C^\infty_c(G\z)$ that are identically~1 on $\supp(nn^*)$ and $\supp(n^*n)$ respectively. So $n = fn = ng$, and we deduce that $nC^{\infty} (G\z) = n gC^\infty(G\z) \subseteq n C^\infty_c(G\z) \subseteq nC^\infty_0(G\z) \subseteq n C^\infty(G\z)$, giving equality throughout, and similarly $C_c^\infty(G\z)n = C_0^\infty(G\z)n = C^{\infty} (G\z) n$. So we can replace $C^\infty(G\z)$ with either $C_c^\infty(G\z)$ or $C_0^\infty(G\z)$ in the statement of Lemma~\ref{lem:n w cpct supp}.
\end{remark}

\begin{remark}\label{rmk:new N,C*}
Lemma~\ref{lem:n w cpct supp} is related to the final condition in Corollary~\ref{cor:smooth thetas by conjugation} and consequently to Assumption~\ref{property:N,C*} in the definition of a smooth Cartan triple (Definition~\ref{dfn:smooth Cartan triple}): If $\ses$ is a Lie twist with $B\subseteq A$ the associated Cartan pair, if $B^{\infty} = C^{\infty}(G\z)\cap B$, and if $n$ is as in Lemma~\ref{lem:n w cpct supp}, then
\[
    n B^{\infty} n^*
    \cup
    n^* B^{\infty} n
    = (B^{\infty} n) n^* \cup n^*(n B^{\infty}) \subseteq  B^{\infty} B^{\infty} \subseteq B^{\infty}.
\]
In other words, any collection $\mathscr{N}$ of normalizers $n\in C_{c}^{\infty} (G;\twist)$  of $C_{0} (G\z)$ for which $\overline{\pi\inv(\suppo(n))}$ is compact satisfies Assumption~\ref{property:N,C*}.
\end{remark}

\begin{example}
It may seem at first glance that, in the setting of Lemma~\ref{lem:thetas of smooth normalisers},
if $n B^{\infty}_{c}(n) n^* \subseteq B^{\infty}$ (and in particular $nn^{*}$ is smooth on its open support), then $n^{*}n$ should be smooth on its open support too. This is not the case. Define $h \colon \mbb{T} \to \mbb{T}$ by $h(e^{\pi i t}) = e^{\pi i t^3}$ for $-1 \le t \le 1$. Then $h$ is a homeomorphism and is smooth, but it is not a diffeomorphism since the cubed-root function is not differentiable at 0. Let $G$ be the groupoid equal as a topological space to $\{0,1\}^2 \times \mbb{T}$ with structure maps
\[
    r((i,j),x) = ((i,i), h^{i-j}(x))\quad s((i,j),x) = ((j,j), x)\quad\text{and}\quad
    ((i,j), h^{j-k}(x))((j,k), x) = ((i,k), x).
\]
Then $C^{*}_{r}(G) \cong M_{2}(C(\mbb{T}))$ and $C_{0}(G\z) \cong C(\mbb{T}) \oplus C(\mbb{T})$ embedded as diagonal matrices. So $\widehat{B} = \mbb{T} \sqcup \mbb{T}$ is a manifold, and $B^{\infty} = C^{\infty}(\mbb{T}) \oplus C^{\infty}(\mbb{T})$. Let $n_{0} \in C^{*}_{r}(G)$ be the indicator function of $\{(0,1)\} \times \mbb{T}$. One can compute that
\begin{equation}\label{eq:conj by n_{0}}
    (n_{0} \varphi n_{0}^*)((i,j),x)
    =
    \begin{cases}
        0&\text{ if } ij\neq 0,
        \\
        \varphi((1,1),h(x))&\text{ if } i=j=0.
    \end{cases}
\end{equation}
    and
\begin{equation}\label{eq:conj by n_{0}^*}
    (n_{0}^* \varphi n_{0})((i,j),x)
    =
    \begin{cases}
        0&\text{ if } ij\neq 1,
        \\
        \varphi((0,0),h\inv(x))&\text{ if } i=j=1.
    \end{cases}
\end{equation}
Thus, $n_{0}n_{0}^*$ is the indicator function of $\{(0,0)\}\times \mathbb{T}\subseteq G\z$, while $n_{0}^*n_{0}$ is the indicator function of $\{(1,1)\}\times \mathbb{T}\subseteq G\z$. It further follows from \eqref{eq:def-theta-hat} that $\widehat{\theta}_{n_{0}}\inv = h$ is smooth, but $\widehat{\theta}_{n_{0}} = h\inv$ is not.

Fix $f_{0} \in C^{\infty}(\mbb{T})^+$ such that $f_{0} \circ \widehat{\theta}_{n_{0}} \not\in C^{\infty}(\mbb{T})$. Let $f=f_{0}\oplus 0$, the copy of $f_{0}$ supported on $\{(0,0)\} \times \mbb{T} \subseteq G\z$, and define $n \coloneqq \sqrt{f} n_{0}$. Then for $g = g_{0} \oplus g_{1} \in B^{\infty}$, we have
\[
n g n^{*} = \sqrt{f} n_{0} g
n^{*}_{0} \sqrt{f} \overset{\eqref{eq:conj by n_{0}}}{=} \sqrt{f} (g_{1} \circ
h
) 1_{(0,0)\times\mbb{T}} \sqrt{f} = [f_{0} (g_{1} \circ h)]\oplus 0 \in B^{\infty}.
\]
So $nB^{\infty} n^* \subseteq B^{\infty}$, and in particular $nn^* \in B^{\infty}$.
However, \[
n^{*}n = n_{0}^{*} f n_{0} \overset{\eqref{eq:conj by n_{0}^*}}{=}
0\oplus [f_{0}\circ h\inv]
\notin B^{\infty}.
\]
\end{example}

\pagebreak[3]
\begin{thm}\label{thm:reconstruction}
\begin{enumerate}[label=\textup{(\arabic*)}]
\item\label{item:thm:from twist to triple}
    Suppose  $ \ses $ is  a Lie twist over an effective \etale\ Lie groupoid.
    Then there exists a collection $\{n_{\alpha}\}_{\alpha\in\mathfrak{A}}$ of elements of $C_{c}^{\infty} (G;\twist)$ such that
    \begin{enumerate}[label=\textup{(\roman*)}]
        \item\label{item:thm:U_alpha bisection} the closure of each $U_{\alpha}\coloneqq \pi(\suppo(n_{\alpha}))$ is contained in a precompact bisection,
        \item\label{item:thm:U_alpha base} the collection $\{U_{\alpha}\}_{\alpha}$ is a base for the topology on $G$, and
        \item\label{item:thm:sigmas} the collection $\{\sigma_{n_{\alpha}}\}_{\alpha}$ of sections of $\pi$ satisfies \PropE, \PropB\ and \PropC.
    \end{enumerate}
    For any collection $\{n_\alpha\}_{\alpha \in \mathfrak{A}} \subseteq C_{c}^{\infty} (G;\twist)$ satisfying \ref{item:thm:U_alpha bisection}--\ref{item:thm:sigmas}, the set \[
    \mathscr{N}=\{f n_{\alpha} : \alpha \in \mathfrak{A}\text{ and } f \in
        C_{c}^{\infty}(G\z)
    \} \cup
        C_{c}^{\infty}(G\z)
        \subseteq
        C_{c}^{\infty} (G;\twist)
    \]
    makes $(C^{*}_{r}(G; \twist), C_{0}(G\z), \mathscr{N})$ a smooth Cartan triple.
\item\label{item:thm:from triple to twist}
    Suppose that $(A, B, \mathscr{N})$ is a smooth Cartan triple and $ \ses $ is the Weyl twist associated to $B\subseteq A$.
    Give $G\z$ the smooth structure that makes the dual homeomorphism $G\z \to \widehat{B}$ a diffeomorphism.
    Then there is a unique smooth structure on $G$ extending that on $G\z$ under which $G$
    is a Lie groupoid, and there is a unique smooth structure on $\twist$ with respect to which $\mathscr{N}\subseteq C^{\infty} (E)$
    and $\PB{\twist}{\pi}{G}$ is a smooth principal $\mbb{T}$-bundle.
    With
    respect to these smooth structures, $\sestriple $ is a Lie twist. Moreover, for each $n \in \mathscr{N}$ and each $h \in B^{\infty}_{c}(n)$, the function  $\widehat{nh} \in C^*_r(G; \twist)$ defined by~\eqref{eq:evaluation map} belongs to $C^{\infty}_{c}(G; \twist)$.
\end{enumerate}
\end{thm}

\begin{remark}
    One might wonder what happens if one applies the two parts of Theorem~\ref{thm:reconstruction} successively.

    First suppose that $\ses$ is a Lie twist and that $\mathscr{N}$ is any family as constructed in~\ref{item:thm:from twist to triple}. We apply~\ref{item:thm:from triple to twist} to the triple $(A,B,\mathscr{N})$ to get an ({\em a priori} new) smooth structure on the (topological) twist~$\twist$; write~$\tilde{\twist}$ for~$\twist$ with this new smooth structure. By~\ref{item:thm:from triple to twist}, it is unique such that $\mathscr{N} \subseteq C^\infty (\tilde{\twist})$ and such that $\PB{\tilde{\twist}}{\pi}{G}$ is a smooth principal $\mbb{T}$-bundle. But by our assumption on $\mathscr{N}$ in~\ref{item:thm:from twist to triple}, we also have  $\mathscr{N} \subseteq C^\infty (\twist)$, and we know from \ref{item:LT:principal bundle} that $\PB{\twist}{\pi}{G}$ is a smooth principal $\mathbb{T}$-bundle; thus, $\twist=\tilde{\twist}$. In other words, applying~\ref{item:thm:from triple to twist} after~\ref{item:thm:from twist to triple} yields exactly the Lie twist we started with.

    However, if we start with a Cartan triple $(A,B,\mathscr{M})$ as in~\ref{item:thm:from triple to twist} and apply~\ref{item:thm:from twist to triple} to the resulting Lie twist $\ses$, then the new Cartan triples $(A,B,\mathscr{N})$ obtained by applying~\ref{item:thm:from twist to triple} typically satisfy neither $\mathscr{N}\subseteq \mathscr{M}$ nor $\mathscr{N}\supseteq \mathscr{M}$: we might not have $\mathscr{N}\subseteq \mathscr{M}$ because $\{\pi(\suppo(n)):n\in\mathscr{N}\}$ is a base for the topology of $G$ by \ref{item:thm:U_alpha base}, while $\{\pi(\suppo(m)):m\in\mathscr{M}\}$ is only assumed to be a cover; and we might not have $\mathscr{N}\supseteq \mathscr{M}$ because the elements of $\mathscr{N}$ were constructed to be compactly supported, while those in $\mathscr{M}$ are allowed not to be. But since $\mathscr{M}\subseteq C^\infty (\twist)$ by construction of the smooth structure on $\twist$ and since we chose $\mathscr{N}\subseteq C^\infty (\twist)$ in~\ref{item:thm:from twist to triple}, it follows from the first paragraph that the triples $(A,B,\mathscr{M})$ and $(A,B,\mathscr{N})$ encode the same smooth structure on $\twist$.

 Moreover, given a smooth Cartan triple $(A, B, \mathscr{M})$, if we define
    \[
        \mathscr{M}_c \coloneqq
        \{
            fm : m\in \mathscr{M}, f\in C_c^\infty (\suppo(m^*m))
        \},
    \]
    then $(A, B, \mathscr{M}_c)$ and $(A, B, \mathscr{M} \cup \mathscr{M}_c)$ are also smooth Cartan triples. The uniqueness assertion in~\ref{item:thm:from triple to twist} implies that $(A, B, \mathscr{M} \cup \mathscr{M}_c)$ induces the same smooth structure on $E$ as each of $(A, B, \mathscr{M})$ and $(A, B, \mathscr{M}_c)$. The family $\{n_\alpha\}_{\alpha \in \mathfrak{A}} = \mathscr{M}_c$ is contained in $C^\infty_c(G; \twist)$ and satisfies \ref{item:thm:U_alpha bisection}--\ref{item:thm:sigmas}. Applying~\ref{item:thm:from twist to triple} to this family, the resulting collection $\mathscr{N}$ is precisely $\mathscr{M}_c$. So the smooth Cartan triples $(A, B, \mathscr{M})$ in which elements of $\mathscr{M}$ have compact support contained in an open bisction are the ones for which the constructions in parts \ref{item:thm:from twist to triple}~and~\ref{item:thm:from triple to twist} are mutually inverse, and $\mathscr{M} \mapsto \mathscr{M}_c$ is a canonical method for replacing a given triple $(A, B, \mathscr{M})$ with a triple of this form.
\end{remark}

\begin{proof}[Proof of Theorem~\ref{thm:reconstruction}]
\ref{item:thm:from twist to triple}
To see that there exists a family $\{n_{\alpha}\}_{\alpha}$ with the desired property, first apply Lemma~\ref{lem:motivation} to obtain a base $\{U_{\alpha} \}_{\alpha\in \mathfrak{A}}$ of bisections for the topology on $G$ and a family  $\{\sigma_{\alpha} \colon U_{\alpha} \to \pi\inv (U_{\alpha})\}_{\alpha\in \mathfrak{A}}$ of smooth sections of $\pi$ that satisfies~\PropE, \PropB~and~\PropC. By Lemma~\ref{lem:refinement}, Parts~\ref{item:refinments inherit PropBCD} and~\ref{item:refinements inherit smoothness}, we may
assume without loss of generality that the closure of each $U_{\alpha}$ is contained in a precompact open bisection $V_{\alpha}$. For each $\alpha$, fix a smooth function $\varphi_{\alpha}\colon r(V_{\alpha})\to [0,\infty)$ with open support equal to $r(U_\alpha)$. By Lemma~\ref{lem:f<->sigma}\ref{item:pi-section gives T-equivariant fcts:smooth}, the element $n_{\alpha}\coloneqq n_{\varphi_{\alpha}}\in C_{c}^{\infty}(G;\twist)$ satisfies $\suppo(n_{\alpha})=\pi\inv (U_{\alpha})$ and $\sigma_{n_{\alpha}} = \sigma_{\alpha}$. So $\{n_\alpha\}_{\alpha \in \mathfrak{A}}$ is contained in $C^{\infty}_{c}(G; \twist)$ and satisfies~\ref{item:thm:U_alpha bisection}, \ref{item:thm:U_alpha base}~and~\ref{item:thm:sigmas}.

Now let $\{n_{\alpha}\}_{\alpha\in\mathfrak{A}}$ be \emph{any} subset of $C^{\infty}_{c}(G; \twist)$ satisfying~\ref{item:thm:U_alpha bisection}, \ref{item:thm:U_alpha base}~and~\ref{item:thm:sigmas}. Let
$\mathscr{N} = \{f n_{\alpha} : \alpha \in \mathfrak{A} \text{ and }f \in
        C_{c}^{\infty}(G\z)
    \} \cup
        C_{c}^{\infty}(G\z)
$. We must show that $(C^{*}_{r}(G; \twist), C_{0}(G\z), \mathscr{N})$ is a smooth Cartan triple. To see that $\mathscr{N}$ densely spans $A$, fix $\alpha \in \mathfrak{A}$. For each open set $U$ such that $U \subseteq \overline{U} \subseteq U_\alpha$, the final statement of Lemma~\ref{lem:f<->sigma} and the smooth Urysohn lemma yields an element $f_{0} \in C^{\infty}_{c}(r(U_\alpha), [0,\infty))$ such that $|f_{0} n_\alpha(e)| = 1$ for all $e \in \pi\inv (U)$. Let $\sigma \colon U \to \pi\inv (U)$ be the section defined by Lemma~\ref{lem:f<->sigma}\ref{item:T-equivariant fct gives pi-section} applied to $f_{0} n_\alpha$.  Then $\varphi \colon (z, \gamma) \mapsto z \cdot \sigma(\gamma)$ is a homeomorphism of $\mathbb{T} \times U$ onto $\pi\inv (U)$, and for $f \in C_{c}(r(U))$, we have $(ff_{0}n_\alpha)(\varphi(z,\gamma)) = z f(r(\gamma))$. Since the $C^*$-norm on $C^*_{r}(G; \twist)$ agrees with the supremum norm on functions supported on  preimages under $\pi$ of bisections, we deduce that $f \mapsto f f_{0} n_\alpha$ is an isometric linear isomorphism of $C_{c}(r(U))$ onto $C_{c}(U; \pi\inv (U))$ that restricts to an isomorphism between the subspaces of smooth functions. Hence  $C_{c}^{\infty}(U_\alpha; \pi\inv (U_\alpha)) \subseteq \{f n_\alpha : f \in  C_{c}^{\infty}(G\z)\} \subseteq \mathscr{N}$. A standard partition of unity argument shows that $C_{c}(G; \twist)$ and hence also $C^*_{r}(G; \twist)$ is spanned by the subspaces $C_{c}^{\infty}(U_\alpha; \pi\inv (U_\alpha))$
and thus by $\mathscr{N}$, as claimed.

As discussed at the start of the section, \cite[Proposition~4.8]{Renault:Cartan} implies that the normalisers of $C_{0}(G\z)$ in $ C^{*}_{r}(G; \twist)$ are precisely the elements that, when viewed as $\mathbb{T}$-equivariant functions on $\twist$, are supported on the preimage under $\pi$ of bisections. Since each $U_{\alpha}$ is a bisection, the collection $\mathscr{N}$ therefore consists of normalizers of $C_{0} (G\z)$. Moreover, since each $n_{\alpha}$ is smooth and compactly supported, we have $\mathscr{N}\cap C_{0}(G\z) = C_{c}^{\infty}(G\z)$, and in particular $C_{c}^{\infty}(G\z) \subseteq \mathscr{N}\cap C_{0}(G\z) \subseteq C^{\infty}(G\z)$. Consequently, $B^{\infty}_{c} \subseteq \mathscr{N}\cap B \subseteq B^{\infty}$ for $B = C_{0}(G\z)$ and $B^{\infty} = C^{\infty} (G\z)\cap B$ as needed.

Since $(n_{\alpha}n_{\alpha}^*)|_{r(V_{\alpha})} =  \varphi_{\alpha}$ for each $\alpha$, the compactly supported function $n_{\alpha}n_{\alpha}^*$ is smooth on $r(V_{\alpha})$ and hence an element of $B_{c}^{\infty} = C_{c}^{\infty}(G\z)$. We claim that $n_{\alpha}^* n_{\alpha}\in B_{c}^{\infty}$. Since $\twist$ is a Lie groupoid and $n_\alpha$ is smooth, so is $n_{\alpha}^*$. Since both $n_{\alpha}$ and $n^*_\alpha$ are compactly supported, Lemma~\ref{lem:f<->sigma}\ref{item:pi-section gives T-equivariant fcts:smooth} implies that the compactly supported function $\varphi_{n_\alpha^*} = n_{\alpha}^*n_{\alpha}$ is smooth on  $r(V_{\alpha}\inv)$. Hence $n_{\alpha}^*n_{\alpha}\in B_{c}^{\infty}$. Since $B_{c}^{\infty}$ is closed under involution and multiplication, we conclude that every $n\in \mathscr{N}$ satisfies $nn^*,n^*n\in B_{c}^{\infty}$.

Remark~\ref{rmk:new N,C*} implies that $fn_{\alpha}$ normalises $C^{\infty}_0(G\z)$. Thus $\mathscr{N}$ satisfies \ref{property:N,C*}.

Let $\condExp$ be the conditional expectation of the Cartan pair $C_{0}(G\z)\subseteq C^{*}_{r}(G; \twist)$. Since $\{\sigma_{n_\alpha}\}_{\alpha}$ satisfies~\PropE, \PropB~and~\PropC, for any three elements $n,m,k\in\{n_{\alpha}\}_{\alpha\in\mathfrak{A}}$, the functions $\Ph(\condExp(n))$, $\Ph(\condExp(mk^*))$, $\Ph(\condExp(nmk^*))$, and $\Ph(\condExp(nm))$ are smooth by Lemma~\ref{lem:C*-interpretation of Props} and Remark~\ref{rmk:I* is superfluous, part 1}. If $f\in B_{c}^{\infty}$, then $\Ph(\condExp(f)) = \Ph(f)$ is smooth; using what we have just showed and using $C_{0}(G\z)$-linearity of $\condExp$, this implies that  for $n,m,k\in\mathscr{N}$, all of the following are smooth as well:
\begin{align*}
    \Ph(\condExp(mf^*))&=\Ph(\condExp(m))\overline{\Ph(f)}
    \\
    \Ph(\condExp(fk^*))&=\Ph(f)
    \Ph(\condExp(k^*))=\Ph(f)\overline{\Ph(\condExp(k))}
    \\
    \Ph(\condExp(fmk^*))&=
    \Ph(f)\Ph(\condExp(mk^*))
    \\
    \Ph(\condExp(nmf^*))
    &=
    \Ph(\condExp(nm))\Ph(f^*)
\end{align*}
By Lemma~\ref{lem:n w cpct supp}, we have $nf=gn$ for some  $g\in C^{\infty} (G\z)$,
and so the function $$\Ph(\condExp(nfk^*)) = \Ph(g)\Ph(\condExp(nk^*))$$ is smooth. By Lemma~\ref{lem:local multipliers smooth}, it follows that $\mathscr{N}$ satisfies~\ref{property:U,C*},  \ref{property:S,C*}~and~\ref{property:M,C*}. So $(C^{*}_{r}(G; \twist), C_{0}(G\z), \mathscr{N})$ is a smooth Cartan triple.

\ref{item:thm:from triple to twist}
Because $\mathscr{N}$ satisfies Condition~\ref{property:N,C*}, any $n\in\mathscr{N}$ satisfies $n^{*}n \in B^{\infty}$  and normalises $B^{\infty}$. So Corollary~\ref{cor:smooth thetas by conjugation} \ref{item:n smooth:two sided}$\implies$\ref{item:theta-hat smooth:two sided both} shows that $\widehat{\theta}_{n}$ is a diffeomorphism. If $U_{n}\coloneqq \pi(\suppo(n))$, then $\widehat{\theta}_{n}=r\circ (s|_{U_{n}})\inv$. Since $\mathscr{N}$ densely spans $A$, the collection $\{U_{n}\}_{n\in\mathscr{N}}$ is a cover of $G$, and so Proposition~\ref{prop:lcdiff iff etale Lie gpd} \ref{item:TFAE:lcdiff}$\implies$\ref{item:TFAE:explicit Lie gpd} shows
that $G$ carries a unique \etale\ Lie groupoid structure  for which $G\z$ with its given manifold structure is an embedded submanifold.

Since $\mathscr{N}$ satisfies \ref{property:U,C*}, \ref{property:S,C*} and \ref{property:M,C*}, Lemma~\ref{lem:local multipliers smooth} shows that  for all $n,m,k\in\mathscr{N}$, the phases of the functions $\condExp(n)$, $\condExp(mk^{*} )$, and $\condExp(nmk^{*})$ are smooth functions. Hence, Lemma~\ref{lem:C*-interpretation of Props} shows that the family $\{\sigma_{n}\colon \suppo(n) \to \twist\}_{n\in\mathscr{N}}$ of sections of $\pi$ that is induced  by the given family $\mathscr{N}$ satisfies~\PropE, \PropB~and~\PropC\
from Definition~\ref{dfn:theProps} with respect to the given smooth structure on $G$.  By~\PropB, Theorem~\ref{thm:E is mfd} yields a unique smooth structure on $\twist$ with respect to which $\pi$ is a smooth principal $\mbb{T}$-bundle and each $\sigma_n$ is smooth. Properties~\PropE~and~\PropC\ combined with Theorem~\ref{thm:E is Lie} show that $\ses$ is a Lie twist.

To see that for each $m \in \mathscr{N}$ and $h \in B^{\infty}_{c}(m)$ the element $\widehat{mh}$ is smooth, fix such a pair $m, h$. The final assertion of Lemma~\ref{lem:f<->sigma}\ref{item:T-equivariant fct gives pi-section} combined with the formula for $n_\varphi$ shows that under the bijection  $\varphi \mapsto n_\varphi$ of Lemma~\ref{lem:f<->sigma}\ref{item:pi-section gives T-equivariant fcts:smooth} obtained from the section $\sigma_m$, we have $\widehat{m} = n_{\sqrt{mm^*}}$. Since $G\z$ is normal, we can find an open set $V$ such that $\supp(h) \subseteq V \subseteq \overline{V} \subseteq \suppo(m^*m)$, and then by the smooth Urysohn lemma \cite[Proposition~2.25]{Lee:Intro}, we can choose $h' \in C^{\infty}_{c}(G\z, [0,1])$ such that $h'$ is identically~1 on $\supp(h)$ and is identically~0 on $G\z \setminus V$; in particular $h' \in B^{\infty}_{c}(m)$. Identifying $C_{0}(G\z)$ with the corresponding subalgebra of $C^*_r(G; \twist)$ as usual, Equation~\eqref{eq:push through n} shows that $\widehat{mh'} = \widehat{m} h' = \theta_m(h')\widehat{m} = n_{\theta_m(h')\sqrt{mm^*}}$. Lemma~\ref{lem:thetas of smooth normalisers} shows that $\theta_m(h') \in C^{\infty}_{c}(G\z,[0,\infty))$, and since $\sqrt{mm^*} \in B^{\infty}$ is positive-valued, we deduce that $\theta_m(h')\sqrt{mm^*} \in  C^{\infty}_{c}(G\z,[0,\infty))$ as well. So Lemma~\ref{lem:f<->sigma}\ref{item:pi-section gives T-equivariant fcts:smooth} shows that $\widehat{mh'}$ is smooth. Since $h'$ is identically~1 on $\supp(h)$, we have $mh = mh'h$ and hence $\widehat{mh} = \widehat{mh'} h$. The function $\widehat{mh'} h$ is equal to the pointwise product of the smooth functions $\widehat{mh'}$ and $h \circ s$ on $\twist$, so it is smooth as claimed.
\end{proof}

\begin{remark}\label{rmk:use Connes}
Our main theorem is phrased so as to provide a correspondence between Lie twists over effective \etale\ Lie groupoids and triples  $(A, B, \mathscr{N}
)$ in which we are \emph{given} a manifold structure on $\widehat{B}$ for which $\mathscr{N}$ consists of smooth sections and contains the compactly supported smooth functions on $\widehat{B}$.
We can, of course, combine this with Connes' reconstruction theorem \cite[Theorem~1.1]{Connes:manifolds} (see also \cite{Rennie:manifolds, RV:reconstruction}): the algebra of smooth functions on a compact oriented smooth manifold is characterised by the functional-analytic data of a spectral triple. Putting Theorem~\ref{thm:reconstruction} together with Connes' result, we obtain a correspondence between Lie twists over \etale\ Lie groupoids whose unit spaces are compact oriented smooth manifolds, and purely functional-analytic tuples $(A, B, \mathscr{N}
, \mathcal{H}, D)$ such that, putting $B^{\infty} \coloneqq \mathscr{N} \cap B$,
\begin{itemize}[label=--]
    \item the triple $(B^{\infty}, \mathcal{H}, D)$ is a spectral triple satisfying Axioms (1)--(5) of \cite{Connes:gravity},
    \item $B$ is the $C^{*}$-completion of $B^{\infty}$, and
    \item the triple $(A, B, \mathscr{N}
)$  is a smooth Cartan triple.
\end{itemize}
However, since we are using Connes' formidable result here as a black box, we have chosen to emphasise that once the smooth structure on $\widehat{B}$ is given, the remaining smooth structure on $G$ and $\twist$ can be recovered from functional-analytic information.
\end{remark}

\begin{remark}\label{rmk:how many smooth structures}
Given a Cartan pair $B\subseteq A$ and  the algebra  $B^{\infty} = C^{\infty}(\widehat{B})\cap B$
of smooth functions for a manifold structure on $\widehat{B}$, it is an interesting question to what extent the smooth structure on the associated Weyl twist $\twist$ is unique. The point is that the smooth structure on $\twist$ is determined by a \emph{choice} of a family $\mathscr{N}$ extending $B^{\infty}$ and satisfying the conditions in Definition~\ref{dfn:smooth Cartan triple}. To see why this is non-unique, consider the situation where $G = X \rtimes \mbb{Z}/2\mbb{Z}$ is the transformation groupoid for a fixed-point-free order-2 diffeomorphism of a compact manifold $X$ and $\twist = \mbb{T} \times G$ is the trivial twist. Then continuous functions from $G$ to $\mbb{T}$ can be identified with sections of $\twist$ via $f \mapsto [\gamma \mapsto (f(\gamma), \gamma)]$. We can then obtain families $\mathscr{N}$ as in Definition~\ref{dfn:smooth Cartan triple} as follows: Let $n_{0} \colon G\z = X \times \{0\} \to \mbb{T}$ be the constant function $n_{0}(x, 0) = 1$; this is the identity element of $C^{*}_{r}(G; \twist )$ regarded as a normaliser of $C_{0}(G\z)$. Fix any \emph{continuous} function $f \colon X \to \mbb{T}$, and let $n_{1} \colon X \times \{1\} \to \mbb{T}$ be the function $n_{1}(x, 1) = f(x)$, again regarded as a normaliser. Now take $\mathscr{N} = \{h n_{0}, h n_{1} : h \in C^{\infty}(X)\}$. It is routine to check that these normalisers have the required properties because the continuous phase of $f$ cancels with its own conjugate wherever it appears in the conditions in Definition~\ref{dfn:smooth Cartan triple}. Since $n_{1}$ is, by definition, smooth with respect to the smooth structure on $\twist $ obtained from Theorem~\ref{thm:reconstruction} for $\mathscr{N}$, this smooth structure only agrees with the standard smooth structure on $\mbb{T} \times G$ if $f$ is itself smooth. So the smooth structure is not unique.

It should be possible to parameterise the possible choices of smooth structure in terms of cohomological data. Given any two families $\mathscr{M}$ and $\mathscr{N}$ as in Definition~\ref{dfn:smooth Cartan triple}, by passing to a common refinement and making appropriate use of Lemma~\ref{lem:refinement}, we can assume that the corresponding families of sections $\{\sigma_m : m \in \mathscr{M}\}$ and $\{\tau_{n} : n \in \mathscr{N}\}$ have the same supports. More precisely we may assume without loss of generality that $\mathscr{M} = \{m_{\alpha} : \alpha \in \mathfrak{A}\}$ and $\mathscr{N} = \{n_{\alpha} : \alpha \in \mathfrak{A}\}$ and that for each $\alpha$, the corresponding sections $\sigma_{\alpha} \coloneqq \sigma_{n_{\alpha}}$ and $\tau_{\alpha} \coloneqq \sigma_{m_{\alpha}}$ coming from Lemma~\ref{lem:f<->sigma} have the same support $U_{\alpha}$. For each $\alpha \in U$ we therefore obtain a continuous function $c_{\alpha} \colon U_{\alpha} \to \mbb{T}$ such that $\sigma_{\alpha}(\gamma)\tau_{\alpha}(\gamma)\inv = i(c_{\alpha}(\gamma), r(\gamma))$ for all $\gamma \in U_{\alpha}$. On double overlaps $U_{\alpha,\beta} \coloneqq U_{\alpha} \cap U_\beta$, the difference $c_{\alpha} \overline{c}_\beta$ is a smooth function by \PropB. So the $c_{\alpha}$ should determine a continuous \v{C}ech 0-cocycle on the \emph{space} $G$ with smooth transition functions. The smooth structures corresponding to $\mathscr{M}$ and $\mathscr{N}$ should coincide exactly when the functions $c_{\alpha}$ are themselves all smooth. So the possible distinct smooth structures should be parameterised by the group of continuous \v{C}ech cohomology classes of 0-cocycles with smooth transition functions modulo the subgroup of cohomology classes of smooth \v{C}ech 0-cocycles.

As a final note, however, we point out that there is, up to \emph{equivalence} only one possible smooth structure on $\twist$ that can arise from Theorem~\ref{thm:reconstruction}. Specifically, \cite[Proposition~I.13]{MW:Equiv} shows that the smooth structure on a given topological principal $\mbb{T}$-bundle over a manifold is unique \emph{up to equivalence} and so the smooth structures on $\twist$ obtained from different choices of $\mathscr{N}$ all coincide up to equivalence. We have not investigated whether the diffeomorphisms implementing these equivalences can be expected to be groupoid homomorphisms and/or morphisms of topological twists, but this seems related to the fact that the space $\mbb{T}$ admits many smooth structures, all of which are equivalent, but only one for which it is a Lie group.
\end{remark}

\appendix

\section{Differential Geometry}\label{sec:appendix} Since a significant proportion of the imagined audience of this paper---and all of its authorship---consists of $C^{*}$-algebraists who may not have at their mental fingertips all of the differential-geometric concepts and terminology used throughout, we have included this brief appendix collecting the key relevant information. Our primary reference is \cite{Lee:Intro}, and we recommend it to those, like us, trying to find the fundamentals of differential geometry that they need all in one well-organised place.

\subsection{Manifolds and submanifolds}
To understand and prove the (seemingly well-known) fact that the composable pairs in a Lie groupoid form an embedded submanifold (see the discussion following Definition~\ref{dfn:Lie gpd} on page~\pageref{page:G2 submanifold}), we need some definitions and  lemmas. Our notation follows that used in \cite{Conlon:DiffMa}.

\begin{defn}[{\cite[p.\ 2ff]{Lee:Intro}}]
Suppose that $M$ is a topological space. We say that $M$ is a {\em topological manifold of
dimension $m$}
if $M$ is Hausdorff, second-countable, and {\em locally Euclidean of dimension $m$}, meaning that each point of $M$ has a neighbourhood
that is homeomorphic to an open subset of $\mbb{R}^m$. A {\em topological chart} of $M$ is a pair $(U,\varphi)$, where $U$ is an open subset of $M$ and $\varphi\colon U\to \varphi(U)$ is a homeomorphism onto an open subset $\varphi(U)$ of $\mbb{R}^m$.
\end{defn}

\begin{defn}\label{dfn:atlas}
Suppose that $M$ is a topological manifold of dimension $m$. A {\em topological atlas}  of $M$ is a collection $\cA=\{(W_{\alpha}, \psi_{\alpha})\}_{\alpha\in \mathfrak{A}}$ where
\begin{enumerate}[label=(A\arabic*)]
    \item\label{item:atlas:cover}
        $\{W_{\alpha}\}_{\alpha\in\mathfrak{A}}$ is an open cover of $M$ and
    \item\label{item:atlas:homeo}
        each $\psi_{\alpha}\colon W_{\alpha} \to \mbb{R}^m$ is a topological chart.
\end{enumerate}
We call $\cA$ a {\em smooth atlas} (or just {\em atlas})  if we furthermore have
\begin{enumerate}[resume, label=(A\arabic*)]
    \item\label{item:atlas:smooth}
        for all $\alpha,\alpha'\in\mathfrak{A}$, either the set $W_{\alpha,\alpha'}\coloneqq W_{\alpha}\cap W_{\alpha'}$ is empty or the map
        \[
            \psi_{\alpha,\alpha'}\coloneqq  \psi_{\alpha}|_{W_{\alpha,\alpha'}} \circ (\psi_{\alpha'}|_{W_{\alpha,\alpha'}})\inv\colon\quad
            \psi_{\alpha'} \left( W_{\alpha,\alpha'}\right)
            \to
            \psi_{\alpha} \left( W_{\alpha,\alpha'}\right)
        \]
        is a diffeomorphism (between open subsets of $\mbb{R}^m$).
\end{enumerate}
In this case, we will call one of the topological charts  $\psi_{\alpha}\colon W_{\alpha} \to \mbb{R}^m$ a {\em smooth chart} (or just {\em chart}).

We say that two atlases are {\em compatible} if their union is another atlas.
An atlas is called {\em maximal} if it is not properly contained in any larger smooth atlas. A topological manifold $M$ is called a {\em smooth manifold} (or just {\em manifold}) if it comes with a chosen maximal smooth atlas $\cA$.
We will often refer to a choice of an atlas as a {\em smooth structure on $M$}. Two smooth structures $\cA,\cB$ on $M$ are called {\em equivalent} if there exists a diffeomorphism $f\colon (M,\cA)\to (M,\cB)$.
\end{defn}

\begin{remark}\label{rmk:equivalent vs compatible atlases}
Recall that every topological manifold $M$ that admits a smooth structure $\mathcal{A}$, admits uncountably many incompatible smooth structures $\mathcal{A}_{t}, t\in\mbb{R}$, that are built on top of the same, fixed topology on $M$ \cite[Problem 1-6]{Lee:Intro}. In other words, the identity map $\mathrm{id}_{M}$ is not a diffeomorphism between $(M,\mathcal{A}_{t})$ and $(M,\mathcal{A}_{s})$ for any two distinct real numbers $s,t$. However, there are many manifolds $M$ for which any two smooth structures are equivalent, i.e., the smooth structure is {\em unique up to diffeomorphism}.
An example of such a manifold is the circle $\mbb{T}$ (see \cite[Exercise~15-13]{Lee:Intro}).
\end{remark}

\begin{defn}[{\cite[p.\ 77]{Lee:Intro}}]
Given two (smooth) manifolds $M, N$, a smooth map $f\colon M \to N$ is
\begin{itemize}
    \item
        a {\em submersion} if its differential $\dd f_x$ is surjective for each $x\in M$ (i.e., $f$ has constant rank $\dim (N)$),
    \item
        an {\em immersion} if its differential $\dd f_x$ is injective for each $x\in M$ (i.e., $f$ has constant rank $\dim(M)$), and
    \item
        an {\em embedding} if it is an immersion and a homeomorphism onto its image.
\end{itemize}
\end{defn}

\begin{defn}[{\cite[p.\ 98ff]{Lee:Intro}}]
Suppose that $M$ is a smooth manifold. An {\em embedded submanifold} of $M$ is a subset $K\subseteq M$ that is a manifold in the subspace topology, endowed with a smooth structure with respect to which the inclusion map
$K\hookrightarrow M$ is a smooth embedding. Embedded submanifolds are also called {\em regular
submanifolds} by some authors.
\end{defn}

We will often think of embedded submanifolds in terms of the {\em local slice condition}:

\begin{lemma}[{\cite[Theorem~5.8]{Lee:Intro}}]
Suppose that $M$ is a manifold of dimension $m$. A subset $K$ is an embedded submanifold of dimension~$k$ if and only if it satisfies the {\em local $k$-slice condition}: for every $y\in K$, there exists a chart $(U,\varphi)$ around $y$ in $M$ such that $\varphi(U\cap K) = \varphi(U)\cap \mbb{R}^k$, where $\mbb{R}^k\subseteq \mbb{R}^m$ is (the image of) the standard inclusion.
\end{lemma}

We were unable to locate an explicit reference for the following standard fact, so we provide a brief proof.

\begin{lemma}\label{lem:proper embedding when intersecting}
If $S$ is an embedded submanifold of $M$ and $U\subseteq M$ is open, then $S\cap U$ is an embedded submanifold of $U$. If $S$ is properly embedded in $M$, then $S\cap U$ is properly embedded in $U$.
\end{lemma}
\begin{proof}
The first statement follows almost immeditately from the definition of an embedded submanifold: Since $U \cap S$ is open in $M$, it is an open subset of $S$ in the subspace topology, and hence an open submanifold. Since $i$ restricts to a smooth embedding $i \colon U\cap S \to M$, we deduce that $U \cap S$ is an embedded submanifold.

For the second statement, if $S$ is properly embedded in $M$, then it is closed in $M$ by \cite[Proposition~5.5]{Lee:Intro}. Hence $S \cap U$ is closed in the subspace topology on $U$. So $S \cap U$ is properly embedded in $U$ by a second application of \cite[Proposition~5.5]{Lee:Intro}.
\end{proof}

\begin{lemma}[{\cite[Propositions 5.4 and 5.7]{Lee:Intro}}]\label{lem:graph is mfd}
Suppose that $M$ and $N$ are smooth manifolds. Let $f\colon M\to N$ be a smooth map. Then its graph $\Gamma_f\subseteq M\times N$ is a properly embedded submanifold of the same dimension as $M$.
\end{lemma}

\begin{proof}
Let $m$ and $n$ denote the manifold dimensions of $M$ and $N$. Fix $(x,f(x))$ in $\Gamma_f$, and let $\psi\colon V\approx \mbb{R}^n$ be a chart around $f(x)$ in $N$. Since $f$ is continuous, $f\inv (V)$ is a neighbourhood around $x$. So there is a chart  $\varphi\colon U\approx\mbb{R}^m$  around $x$ in $M$ with $U\subseteq f\inv (V).$
Consider the smooth bijective map
\[
    \Omega\colon U\times\mbb{R}^n \to U\times\mbb{R}^n,
        (\vec{v}, \vec{w}) \mapsto
            \biggl(
                \vec{v}, \vec{w} - \psi\Bigl(f\bigl(\varphi\inv(\vec{v})\bigr)\Bigr)
            \biggr).
\]
Let $\Lambda\coloneqq \Omega\circ (\varphi\times \psi)$. Then $(U\times V, \Lambda)$ is a chart around $(x,f(x))$ in $M\times N$, and
\[
    \Lambda\bigl(U\times V \cap \Gamma_{f}\bigr)
        = \Omega \Bigl\{\bigl(\varphi(y),\psi(f(y))\bigr) : y\in U\Bigr\}
        = \Bigl\{\bigl(\varphi(y),\vec{0}\bigr) : y\in U\Bigr\}
        = \mbb{R}^m \times \{\vec{0}\}.
\]
This proves that $\Gamma_f$ is a submanifold of dimension $m$.
\end{proof}

\begin{defn}[{\cite[p.\ 143]{Lee:Intro}}]
Two embedded submanifolds $K_{1},K_{2}\subseteq M$ are
said to be {\em transverse} if for each $x\in K_{1}\cap K_{2}$, the tangent spaces $T_x K_{1}$ and $T_x K_{2}$ together span $T_x M$.
If $f\colon M \to N$ is a smooth map and $K\subseteq N$ is an embedded submanifold, we say that $f$ is {\em transverse to $K$} if for every $x\in  f\inv(K)$, the spaces $T_{f(x)}K$ and $\mathrm{im} (\dd f_{x})$ together span $T_{f(x)}N$.
A pair of maps $f \colon M \to K$ and $g \colon N \to K$ are {\em transverse} if whenever $f(x) = g(y) = z \in K$, the images $\mathrm{im}(\dd f_x)$ and $\mathrm{im}(\dd g_y)$ together span $T_z(K)$.
\end{defn}

\begin{remark}\label{rmk:submersion implies transverse}
If $f\colon M\to N$ is a submersion, then it is transverse to {\em any} embedded submanifold of $N$, since $\mathrm{im} (\dd f_{x})$ alone already spans all of $T_{p}N$.
\end{remark}

\begin{thm}[{\cite[Theorem~6.30]{Lee:Intro}}]\label{thm:Generalized Preimage Thm}
Suppose that $M$ and $N$ are smooth manifolds, that $K$ is an embedded submanifold of $N$ of dimension $k$, and that $f\colon M\to N$ is smooth.
\begin{enumerate}[label=\textup{(\arabic*)}]
    \item\label{item:Generalized Preimage Thm}
        If $f$  is transverse to $K$, then $f\inv (K)$ is a submanifold of $M$ whose codimension in $M$ coincides with the codimension of $K$ in~$N$.
    \item\label{item:intersec of trans submfds is submfd}
        If $K'$ is an embedded submanifold of $N$ that is transverse to $K$, then $K\cap K'$ is also a submanifold and its codimension is the sum of the codimensions of $K$ and $K'$.
\end{enumerate}
\end{thm}

The following two results are well-known, but since we were unable to track a published references for them, we provide proofs.
\begin{prop}\label{prop:transverse maps implies fibred product is mfd}
Suppose that $M,N,K$ are smooth manifolds of dimensions $m,n,k$ respectively. If $f\colon M\to K$, $g\colon N\to K$ are transverse, then the fibred product $M\fp{f}{g}N$ is a submanifold of $M\times N$ of dimension $m+n-k$.
\end{prop}
\begin{proof}
We follow the idea of the proof of \cite[Proposition~3.3]{Gualtieri:Mfds}.
By Lemma~\ref{lem:graph is mfd}, the graphs  $\Gamma_f\subseteq M\times K$ and  $\Gamma_g\subseteq N\times K$ are submanifolds of dimension $m$ and $n$. We claim that the submanifolds $ \Gamma_f \times \Gamma_g $ and $H\coloneqq\{(x,k,y,k) : x\in M, k\in K, y\in N\}$ of $(M\times K)\times (N\times K)$ are transverse.
So for any $p=(x,f(x),y,g(y))=(x,k,y,k)\in (\Gamma_f \times \Gamma_g) \cap H$, we must show that
\[
    T_p (\Gamma_f \times \Gamma_g) + T_{p} H = T_{p} ((M\times K)\times (N\times K)).
\]
Pick a tangent vector at $(M\times K)\times (N\times K)$ at $p$, say $({a}, {b}, {c}, {d})$. Since $f$ and $g$ are transverse and since $f(x)=k=g(y)$, there exist $\eta\in T_{x}M$ and $\mu\in T_{y}N$ such that the element ${d}-{b}$ of $T_{k}K$ can be written as $\dd f_{x}(\eta)+\dd g_{y}(\mu)$. Thus,  $d= b + \dd f_{x}(\eta)+\dd g_{y}(\mu)$, and hence
\begin{align*}
    ({a}, {b}, {c}, {d})
        ={}&  \big({a}+\eta,\; {b}+\dd f_{x}(\eta),\;{c}-\mu,\;{b}+\dd f_{x}(\eta)\big)\\
            &{}+ \big(-\eta,\; -\dd f_{x}(\eta),\;0,\;0\big)\\
            &{}+ \big(0,\; 0,\; \mu,\; \dd g_{y}(\mu)\big).
\end{align*}

The first line is an element of $T_{p} H$, the second line is an element of $T_{(x,f(x))} \Gamma_f$, and the third line is an element of $T_{(y,g(y))} \Gamma_g$. Hence their sum is an element of $T_p (\Gamma_f \times \Gamma_g)$. This proves our claim.

It now follows from Theorem~\ref{thm:Generalized Preimage Thm}\ref{item:intersec of trans submfds is submfd} that the intersection $L$ of $ \Gamma_f \times \Gamma_g $ and $H$ is also a submanifold, namely of dimension
\begin{align*}
    \dim (L)
        &= \dim (\Gamma_f \times \Gamma_g) + \dim H - \dim \bigl((M\times K)\times (N\times K)\bigr)\\
        &= [m + n] + [m+k+n] - [m+k+n+k]
         = m + n - k.
\end{align*}
The map $\iota\colon L\to M\times N$ given by $\iota(x,f(x),y,g(y)) = (x,y),$ has image exactly $M\fp{f}{g}N$ and  has constant rank $\dim(L)$. \label{p:spot for roundabout argument}
Since $\iota$ is a topological homeomorphism onto its image, it is an embedding by \cite[Theorem~4.14(b)]{Lee:Intro}. By \cite[Theorem~11.13]{Tu:Intro}, this implies that $\iota(L)=M\fp{f}{g}N$ is a submanifold of $M\times N$ of dimension $\dim(L)$, which we computed earlier as $m+n-k$.
\end{proof}

\begin{lemma}\label{lem:local diffeo is submersion}
A local diffeomorphism  $f\colon M\to N$ between two smooth manifolds is a submersion.
\end{lemma}
\begin{proof}
We need to check that each differential $\dd f_{p}\colon T_p M\to T_{f(p)}N$ is surjective, so fix $p \in M$ and let $D\in T_{f(p)}N$ be an arbitrary derivative of $\mathfrak{G}_{f(p)}.$

For any fixed germ $[h]_{p}$ at $p$, let $h\colon U\to\mbb{R}$ be a representative of the germ. By assumption on $f$, there is a neighbourhood $V$ of $p$ such that $f(V)$ is open in $N$ and $f|_V\colon V\to f(V)$ is a diffeomorphism. We may shrink $V$ to an even smaller neighbourhood $W$ of $p$ such that $W\subseteq U$. Since the map $h\circ (f|_{W})\inv\colon f(W)\to\mbb{R}$ is a smooth function defined around $f(p)$, we can evaluate $D$ at $[h\circ (f|_{W})\inv]_{f(p)}$. The resulting value does not depend on the choices of $W$ or $h$: if $W'\subseteq U$ is another open neighbourhood of $p$ and $h'$ is another representative of $[h]_p$, then there is an open neighbourhood $W''$ of $p$ such that $h|_{W''} = h'|_{W''}$, and then $f(W \cap W' \cap W'')$ is an open neighbourhood of $f(p)$ on which $h\circ (f|_{W})\inv$ and $h'\circ (f|_{W'})\inv$ coincide. Hence
\[
    [h\circ (f|_{W})\inv]_{f(p)}
        = [h'\circ (f|_{W'})\inv]_{f(p)}.
\]

As the value of $D$ at $[h\circ (f|_{W})\inv]_{f(p)}$ depends neither on the choice of $W$ nor of $h$, we may define
\[
    \tilde{D} [h]_p \coloneqq  D [h\circ (f|_{W})\inv]_{f(p)}
\]
To see that $\tilde{D}$ is a derivative of $\mathfrak{G}_p$, we calculate:
\begin{align*}
    \tilde{D} ([h_{1}]_p [h_{2}]_p)
        &= \tilde{D} ([h_{1}h_{2}]_p) \\
        &= D [(h_{1}h_{2})\circ (f|_{W})\inv]_{f(p)} \\
        &= D [(h_{1}\circ (f|_{W})\inv) (h_{2}\circ (f|_{W})\inv)]_{f(p)} \\
        &= D [(h_{1}\circ (f|_{W})\inv)]_{f(p)} e_{f(p)}\big(h_{2}\circ (f|_{W})\inv\big)
            + e_{f(p)}\big(h_{1}\circ (f|_{W})\inv\big)D [(h_{2}\circ (f|_{W})\inv)]_{f(p)}\\
        &= \tilde{D} ([h_{1}]_p) e_p([h_{2}]) + e_p([h_{1}]) \tilde{D} ([h_{2}]_p).
\end{align*}
For any germ $[g]_{f(p)}$ of $f(p),$ we have
\[
    \dd f_p (\tilde{D})[g]_{f(p)}
        = \tilde{D}[g\circ f]_{p}
        = D [(g\circ f)\circ (f|_{W})\inv]_{f(p)}
        = D[g]_{f(p)},
\]
proving that $\dd f_p (\tilde{D})=D,$ as claimed.
\end{proof}

\subsection{Principal bundles} Here we collect the facts about principal bundles, and in particular smooth principal bundles, that we use throughout the paper, for ease of reference.

\begin{defn}[{\cite[Definition 11.2.2]{Conlon:DiffMa}}]\label{dfn:principal bdl}
Let $M$ and $P$ be topological spaces, $\Gamma$ a topological group, $\pi\colon P\to M$ a continuous  map, and $\Gamma\times P\to P$, $(s,x) \mapsto s \cdot x$, a continuous left action that is free and transitive on each fibre $\pi\inv (m)$. Define a left action of $\Gamma$ on $\Gamma \times P$ by
\begin{align*}
    \Gamma\times (\Gamma \times U_{\alpha})
        &\to  \Gamma \times U_{\alpha}\\
    (s, (t,x))
        &\mapsto (st,x).
\end{align*}
Suppose that $\{U_{\alpha}\}_{\alpha\in\mathfrak{A}}$ is an open cover of $M$ and that for each $\alpha$, the map $\psi_{\alpha} \colon \pi\inv(U_{\alpha}) \to \Gamma \times U_{\alpha}$ is a $\Gamma$-equivariant homeomorphism such that the diagram
\[
\begin{tikzcd}
    \pi\inv (U_{\alpha}) \ar[rr, "\psi_{\alpha}"] \ar[dr, "\pi"']&&  \Gamma\times  U_{\alpha} \ar[dl, "\mathrm{pr}_{2}"] \\
    &U_{\alpha} &
\end{tikzcd}
\]
commutes. Then, $\PB{P}{\pi}{M}$, together with the given $\Gamma$-action, is called a {\em topological principal
$\Gamma$-bundle over $M$}, $\Gamma$ is called the {\em structure group} of the bundle, and each $(U_{\alpha}, \psi_{\alpha})$ is called a {\em topological local trivialisation}.

We call $\PB{P}{\pi}{M}$ a {\em smooth principal bundle} if  $M,P$ are manifolds, $\Gamma$ is a Lie group, $\pi$ is smooth, the $\Gamma$-action on $P$ is smooth, and the $\psi_{\alpha}$ can be chosen to be diffeomorphisms, in which case we call them {\em smooth local trivialisations}.
\end{defn}

As with manifolds, if the term `principal bundle' appears without an adjective, we will implicitly mean that it is smooth.

\begin{remark}\label{rmk:pi submersion, restr bdl}
If $\PB{P}{\pi}{M}$ is a smooth principal $\Gamma$-bundle, then $\pi$ is a submersion \cite[Exercise~10-19]{Lee:Intro}. If $N\subseteq M$ is an embedded submanifold, then $P_{N}\coloneqq \pi\inv(N)$ is therefore an embedded submanifold of $P$ \cite[Corollary 6.31]{Lee:Intro}. One can further show that $\PB{P_{N}}{\pi|_{P_{N}}}{N}$ is a smooth principal $\Gamma$-bundle and that the inclusion $P_{N}\hookrightarrow P$ is a principal-$\Gamma$-bundle morphism, covering the inclusion $N\hookrightarrow M$. %
\end{remark}

\begin{lemma}\label{lem:principal bundles are rigid}
Suppose that $\PB{P_i}{\pi_i}{M}$, $i = 1,2$ are smooth principal $\Gamma$-bundles and that $\Psi\colon P_{1}\to P_{2}$ is a $\Gamma$-equivariant bundle map. If $\Psi$ is a smooth homeomorphism, then $\Psi$ is a diffeomorphism.
\end{lemma}

\begin{proof}
First consider the case in which  $P_{1}=P_{2}$ as topological spaces (not necessarily as manifolds) and that $\Psi=\mathrm{id}_{P_{1}\to P_{2}}$ is the identity map; in particular, since we are asking for this map to be a bundle map, we must have $\pi_{1}=\pi_{2}$. Fix an arbitrary $e_{2}\in P_{2}$, and write $e_{1}$ for the same point regarded as an element of $P_{1}$. Recall that $\pi_{i}$ is an open map, so for $i=1,2$, there is an open neighbourhood $V_i$ of $\pi_{i}(e_{i})$ such that there is a $\Gamma$-equivariant diffeomorphism $\psi_{i}$ making the diagram
\[
\begin{tikzcd}
    P_{i}\supseteq&[-3em]\pi_{i}\inv (V_{i}) \ar[rr, "\psi_{i}"] \ar[dr, "\pi_{i}"']&& \Gamma \times V_{i}
    \ar[dl, "\mathrm{pr}_{2}"] &[-3em]\subseteq \Gamma\times M \\
    &&V_{i} &&
\end{tikzcd}
\]
commute. Now, $\pi_{1}(e_{1})=\pi_{2}(e_{2})$, so  $V\coloneqq V_{1}\cap V_{2}$ is an open neighbourhood of $\pi_i(e)$ for each $i$. Consider
\[
\begin{tikzcd}
    \Gamma \times V
    \ar[r,"\psi_{1}\inv", yshift = 1mm]
        & \ar[l,"\psi_{1}", yshift = -1mm]
            \pi_{1}\inv (V)
            \ar[r,"\mathrm{id}", yshift = 1mm, <->]
        & \pi_{2}\inv (V)
            \ar[r,"\psi_{2}", yshift = 1mm]
        &  \ar[l,"\psi_{2}\inv", yshift = -1mm]
        \Gamma \times V.
\end{tikzcd}
\]
Since $\mathrm{pr}_{2}\circ \psi_{i} = \pi_{i}$, there is a unique map $T\colon \Gamma \times V \to \Gamma$ such that
\[
    (\psi_{2}\circ \psi_{1}\inv)(\gamma,m) = (T(\gamma,m),m).
\]
Since $\psi_{i}$ and $\Psi$ are $\Gamma$-equivariant, it follows that
\[
    T(\gamma,m)=\gamma T(1_{\Gamma},m),
\]
so the map $t\colon V\to \Gamma$ given by $t(m)=T(1_{\Gamma},m)$ satisfies
\[
    (\psi_{2}\circ \psi_{1}\inv)(\gamma,m) = (\gamma t(m),m).
\]
An easy computation now shows that interchanging the roles of $\psi_{1}$ and $\psi_{2}$ amounts to composing $t$ with the inversion map in $\Gamma$:
\[
    (\psi_{1}\circ \psi_{2}\inv) (\gamma, m)=(\psi_{2}\circ \psi_{1}\inv)\inv (\gamma,m)
        = (\gamma t(m)\inv,m).
\]
By assumption, the map $\mathrm{id}\colon \pi_{1}\inv (V) \to \pi_{2}\inv(V)$ is smooth, so $\psi_{2}\circ \psi_{1}\inv$ is smooth, and thus so are $T$ and $t$. Since $\Gamma$ is a Lie group, the map $\Gamma\times V\to \Gamma$ given by $(\gamma, m)\mapsto \gamma t(m)\inv$ is smooth. In particular, $\psi_{1}\circ \psi_{2}\inv$ is smooth. Hence $\mathrm{id}\colon \pi_{2}\inv (V) \to \pi_{1}\inv(V)$ is smooth around $e_{2}$. In other words, $\mathrm{id}\colon P_{1}\to P_{2}$ is a diffeomorphism.

We now deal with the general situation; for reasons that will become clear in a moment, we stray from the notation in the lemma. Fix given smooth principal $\Gamma$-bundles $\PB{P_{1}}{\pi_{1}}{M}$ and $\PB{Q}{\pi^Q}{M}$ and a $\Gamma$-equivariant bundle map $\Psi\colon P_{1}\to Q$ which is a smooth homeomorphism. We show that $\Psi\inv$ is smooth.

Since $\Psi$ is a {\em homeomorphism}, we may let $P_{2}$ be the topological space $P_{1}$ endowed with the smooth structure for which $\Psi$ is a diffeomorphism.
In other words, if we again write $e_{1}$ for the point $e_{2}\in P_{2}$ but regarded as a point of $P_{1}$, then the map $\Psi_{2}\colon P_{2}\to Q,e_{2}\mapsto \Psi(e_{1})$, is a smooth map between manifolds with smooth inverse. We equip $P_{2}$ with the projection map $\pi_{2}\colon P_{2}\to M$ given by $e_{2}\mapsto \pi^Q (\Psi_{2}(e_{2}))$ and with the $\Gamma$-action $\gamma e_{2} \coloneqq \Psi_{2}\inv (\gamma \cdot \Psi_{2}(e_{2}))$. Since $\PB{Q}{\pi^Q}{M}$ is a smooth principal $\Gamma$-bundle, so is $\PB{P_{2}}{\pi_{2}}{M}$, and by construction, $\Psi_{2}$ is an isomorphism of smooth principal bundles:
\[
\begin{tikzcd}[column sep = large, row sep = large]
    P_{2} \ar[d, "\pi_{2}"] \ar[dr, phantom, very near start, "{\circlearrowleft}" description]\ar[r, "{\Psi_{2}, \cong}"] &Q\ar[dl, "\pi^{Q}"]\\
    M & {}
\end{tikzcd}
\]

Since $\Psi$ is a bundle map,
\[
\pi_{2}(e_{2})
    = \pi^Q (\Psi_{2}(e_{2}))
    = \pi^Q (\Psi(e_{1}))
    = \pi_{1}(e_{1}).
\]
Moreover, since $\Psi$ is $\Gamma$-equivariant,
\[
\gamma \cdot e_{2}
    = \Psi_{2}\inv (\gamma \cdot \Psi_{2}(e_{2}))
    = \Psi_{2}\inv (\gamma \cdot \Psi(e_{1}))
    = \Psi_{2}\inv (\Psi(\gamma \cdot e_{1}))
    = \gamma \cdot e_{1}.
\]
In other words, the diagram
\[\begin{tikzcd}[column sep = large, row sep = large]
    P_{1}\ar[r, "\mathrm{id}"] \ar[dr, "\pi_{1}"']\ar[rr, "\Psi" {name=Psi}, bend left=50] &
    \ar[dl, phantom, very near start, "{\circlearrowleft}" description]
    \ar[to=Psi, phantom, "{\circlearrowleft}" description]
    P_{2} \ar[d, "\pi_{2}" description] \ar[dr, phantom, very near start, "{\circlearrowleft}" description]\ar[r, "{\Psi_{2}}"] &
    Q\ar[dl, "\pi^{Q}"]
    \\
    {} & M & {}
\end{tikzcd}\]
of principal $\Gamma$-bundles commutes.
By construction of the smooth structure and by the assumption that $\Psi$ is smooth, the map  $\mathrm{id}_{P_{1}\to P_{2}}=\Psi_{2}\inv \circ \Psi$ is smooth. We are now in the setting of the first paragraph, and conclude that $\mathrm{id}$ is a diffeomorphism. Thus, $\Psi\inv = \Psi_{2}\circ \mathrm{id}_{P_{2}\to P_{1}}$ is smooth, as claimed.
\end{proof}

\begin{lemma}\label{lem:principal bundles:uniquely induced}
Suppose that $\PB{P_{1}}{\pi_{1}}{M}$ is a smooth principal $\Gamma$-bundle, that $\PB{P_{2}}{\pi_{2}}{M}$ is a {\em topological} principal $\Gamma$-bundle, and  that $\Psi\colon P_{1}\to P_{2}$ is a $\Gamma$-equivariant bundle map.
If $\Psi$ is homeomorphism and $P_{2}$ is given the unique smooth structure that makes $\Psi$ a diffeomorphism, then $\pi_{2}$ is a smooth principal bundle.
\end{lemma}
\begin{proof}
Since $\Gamma$ acts smoothly on $P_{1}$ and $\Psi$ is $\Gamma$-equivariant, the action on $P_{2}$ is automatically smooth. Since $\pi_{1}$ is smooth and $\Psi$ is a bundle map, $\pi_{2}$ is smooth, and since $\pi_{1}$ allows smooth trivialisations, so does $\pi_{2}$.
\end{proof}

\begin{lemma}\label{lem:sections => triv's}
Suppose that $\PB{P}{\pi}{M}$ is a topological (respectively smooth) principal $\Gamma$-bundle  for which the $\Gamma$-action is proper. Then for any $U\subseteq M$ open  and any continuous (smooth)  section $\sigma\colon U\to\pi\inv(U)$, the map
\[
    \Gamma\times U \to \pi\inv(U),
        \quad (\gamma, m) \mapsto \gamma\cdot \sigma(m),
\]
is a homeomorphism (respectively diffeomorphism)  and thus  defines a topological (smooth) local trivialisation $\psi_\sigma\colon \pi\inv(U) \to \Gamma\times U$ of $P$.
\end{lemma}
\begin{proof}
For this proof, let $f\colon \Gamma\times U \to \pi\inv(U)$ be the map in the displayed equation; once we have shown that $f$ is invertible, we will care more about its inverse $\psi_\sigma$ and thus write $f=\psi_\sigma\inv$.

Recall that $P_{U}\coloneqq \pi\inv(P)$ is an embedded submanifold of $P$ and $\PB{P_{U}}{\pi}{U}$ is itself a topological (smooth) principal $\Gamma$-bundle (Remark~\ref{rmk:pi submersion, restr bdl}). Likewise, with respect to the projection $q$ onto the second coordinate, $\PB{\Gamma\times U}{q}{U}$ is a topological (smooth) principal $\Gamma$-bundle. Note that $f$ is equivariant by construction and satisfies $\mathrm{pr}_{2}=\pi\circ f$. Thus, in the topological case, we only need to show that $f$ is a homeomorphism; in the smooth case, we further need to check that $f$ is smooth so that we can invoke Lemma~\ref{lem:principal bundles are rigid}

The map $f$ is injective since the $\Gamma$-action on $P$ is free and since $\sigma$ is injective, being a section, and $f$ is surjective since the $\Gamma$-action is transitive. To see that $f$ is open, we invoke \cite[Proposition~1.1]{Wil2019}, so suppose that $\{e_\lambda\}_\lambda$ is a net in $\pi\inv(U)$ such that $e_\lambda\to \gamma\cdot \sigma(m)$ for some $\gamma\in \Gamma$ and $m\in U$; we must show that (a subnet of) $\{e_\lambda\}_\lambda$ can be lifted to elements
in $\Gamma\times U$ which converge to $(\gamma,m)$. Let $m_\lambda\coloneqq \pi(e_\lambda)$; since $\pi$ is continuous and $M$ is Hausdorff, $m_\lambda\to m$. For each $\lambda$, we have $m_\lambda=\pi(\sigma(m_\lambda))$ since $\sigma$ is a section. Hence transitivity and freeness of the $\Gamma$-action implies that there exists a unique $\gamma_\lambda\in\Gamma$ such that $e_\lambda = \gamma_\lambda \cdot \sigma(m_\lambda)$. It remains to check that (a subnet of) $\{\gamma_\lambda\}_{\lambda}$ converges to $\gamma$. This follows immediately from properness of the $\Gamma$-action: both $\gamma_\lambda \cdot \sigma(m_\lambda)$ and $\sigma(m_\lambda)$ converge.

In the smooth case, $f$ is  smooth as a map $\Gamma\times U \to P$ since $\sigma$ is smooth and since the $\Gamma$-action on $P$ is smooth.
It follows that $f$ is also smooth as a map $\Gamma\times U \to P_{U}$ \cite[Corollary~5.30]{Lee:Intro}. By Lemma~\ref{lem:principal bundles are rigid}, it now follows that $f$ is a diffeomorphism.
\end{proof}

\emergencystretch=2.5em 
\printbibliography
\end{document}